\documentclass[a4paper,reqno,11pt]{amsart}
\usepackage{amsfonts,amsmath,amsthm,amssymb,stmaryrd}
\usepackage{mathrsfs}

\newtheorem{theorem}{Theorem}[section]
\newtheorem{lemma}[theorem]{Lemma}
\newtheorem{Proposition}[theorem]{Proposition}
\newtheorem{Remark}[theorem]{Remark}
\numberwithin{equation}{section}
\usepackage{color}
\allowdisplaybreaks

\usepackage[left=1 in, right=1 in,top=1 in, bottom=1 in]{geometry}

\providecommand{\abs}[1]{\left\vert#1\right\vert}
\providecommand{\norm}[1]{\left\Vert#1\right\Vert}

\providecommand{\Rn}[1]{\mathbb{R}^{#1}}

\providecommand{\snormspace}[3]{\left\Vert#1\right\Vert_{H^{#2}({#3})}}
\providecommand{\sd}[1]{\mathcal{D}_{#1}}
\providecommand{\se}[1]{\mathcal{E}_{#1}}
\providecommand{\sdb}[1]{\bar{\mathcal{D}}_{#1}}

\providecommand{\snormspace}[3]{\left\Vert#1\right\Vert_{H^{#2}({#3})}}
\providecommand{\ns}[1]{\norm{#1}^2}

\def\Lbrack{\left \llbracket}
\def\Rbrack{\right \rrbracket}

\DeclareMathOperator{\trace}{tr}
\DeclareMathOperator{\diverge}{div}

\providecommand{\Rn}[1]{\mathbb{R}^{#1}}
\providecommand{\norm}[1]{\left\Vert#1\right\Vert}
\def\ls{\lesssim}
\def\gss{\gtrsim}

\def\dt{\partial_t}
\def\H{{}_0H^1}
\def\Hd{(\H)^\ast}
\def\na{\nabla}
\def\pa{\partial}

\def\rj{\Lbrack \bar{\rho} \Rbrack}
\providecommand{\jump}[1]{\left\llbracket #1 \right\rrbracket }

\def\RRvert2{\right \vert\! \right\vert}
\def\Lvert3{\left \vert\!\left\vert\!\left\vert}
\def\Rvert3{\right \vert\!\right\vert\!\right\vert}

\def\nab{\nabla}
\def\al{\alpha}
\def\dt{\partial_t}

\def\hal{\frac{1}{2}}
\def\ep{\varepsilon}

\def\ls{\lesssim}
\def\p{\partial}
\def\sg{\mathbb{D}}
\def\sgz{\mathbb{D}^0}

\def\da{\Delta_{\mathcal{A}}}
\def\naba{\nab_{\mathcal{A}}}
\def\diva{\diverge_{\mathcal{A}}}

\def\a{\mathcal{A}}
\def\f{\mathcal{F}_{2N}}
\def\gs{\mathcal{G}_{2N}^\sigma}

\def\fj1{\mathcal{J}^{-1}}

\def\n{\mathcal{N}}

\def\y{\mathcal{Y}}
\def\z{\mathcal{Z}}
\def\q{q}
\def\S{\mathbb{S}}
\def\ks{K(\sigma,\rj)}

\def\Ef{\mathfrak{E}_{2N}^\sigma}
\def\Df{\mathfrak{D}_{2N}^\sigma}
\def\Af{\mathfrak{A}_{2N}^{j,k}}
\def\Bf{\mathfrak{B}_{2N}^{j,k}}
\def\Hf{\mathfrak{H}_{2N}^{j,k}}

\title[Compressible viscous surface-internal waves]{The compressible viscous surface-internal wave problem: stability and vanishing surface tension limit}

\author{Juhi Jang}
\address{
Department of Mathematics\\
University of California, Riverside\\
Riverside, CA 92521, USA
}
\email[J. Jang]{juhijang@math.ucr.edu}
\thanks{J. Jang was supported in part by NSF grant DMS-1212142 and DMS-1351898.}

\author{Ian Tice}
\address{
Department of Mathematical Sciences\\
Carnegie Mellon University\\
Pittsburgh, PA 15213, USA
}
\email[I. Tice]{iantice@andrew.cmu.edu}

\author{Yanjin Wang}
\address{
School of Mathematical Sciences\\
Xiamen University\\
Xiamen, Fujian 361005, China}
\email[Y. J. Wang]{yanjin$\_$wang@xmu.edu.cn}
\thanks{Y. J. Wang was supported by the National Natural Science Foundation of China (No. 11201389)}


\begin{document}

\begin{abstract}

This paper concerns the dynamics of two layers of compressible, barotropic, viscous fluid lying atop one another. The lower fluid is bounded below by a rigid bottom, and the upper fluid is bounded above by a trivial fluid of constant pressure.  This is a free boundary problem: the interfaces between the fluids and above the upper fluid are free to move.  The fluids are acted on by gravity in the bulk, and at the free interfaces we consider both the case of surface tension and the case of no surface forces.  We establish a sharp nonlinear global-in-time stability criterion and give the explicit decay rates to the equilibrium.   When the upper fluid  is heavier than the lower fluid along the  equilibrium interface, we characterize the set of surface tension values in which the equilibrium is nonlinearly stable.  Remarkably, this set is non-empty, i.e. sufficiently large surface tension can prevent the onset of the Rayleigh-Taylor instability.  When the lower fluid is heavier than the upper fluid, we show that the equilibrium is stable for all non-negative surface tensions and we establish the zero surface tension limit.
\end{abstract}

\maketitle


\section{Introduction}
\subsection{Formulation in Eulerian coordinates}

We consider two distinct,
immiscible, viscous, compressible, barotropic fluids evolving in a moving domain $\Omega(t)=\Omega_+(t)\cup \Omega_-(t)$ for time $t\ge0$. One fluid $(+)$, called the ``upper fluid", fills the upper domain
\begin{equation}\label{omega_plus}
\Omega_+(t)=\{y\in  \mathrm{T}^2\times \mathbb{R}\mid \eta_-(y_1,y_2,t)<y_3< \ell +\eta_+(y_1,y_2,t)\},
\end{equation}
and the other fluid $(-)$, called the ``lower fluid", fills the lower domain
\begin{equation}\label{omega_minus}
\Omega_-(t)=\{y\in  \mathrm{T}^2\times \mathbb{R}\mid  -b <y_3<\eta_-(y_1,y_2,t)\}.
\end{equation}
Here we assume the domains are horizontally periodic by setting $\mathrm{T}^2=(2\pi L_1\mathbb{T}) \times (2\pi L_2\mathbb{T})$ for $\mathbb{T} = \mathbb{R}/\mathbb{Z}$ the usual 1--torus and $L_1,L_2>0$ the periodicity lengths.  We assume that $\ell,b >0$ are two fixed and given constants, but the two surface functions $\eta_\pm$ are free and unknown.  The surface $\Gamma_+(t) = \{y_3= \ell  + \eta_+(y_1,y_2,t)\}$ is the moving upper boundary of $\Omega_+(t)$ where the upper fluid is in contact with the atmosphere, $\Gamma_-(t) = \{y_3=\eta_-(y_1,y_2,t)\}$ is the moving internal interface between the two fluids, and $\Sigma_b = \{y_3=-b \}$ is the fixed lower boundary of $\Omega_-(t)$.

The two fluids are described by their density and velocity functions, which are given for each $t\ge0$ by $\tilde{\rho}_\pm
(\cdot,t):\Omega_\pm (t)\rightarrow \mathbb{R}^+$ and $\tilde{u}_\pm (\cdot,t):\Omega_\pm (t)\rightarrow \mathbb{R}^3$, respectively.  In each fluid the pressure is a function of density: $P_\pm =P_\pm(\tilde{\rho}_\pm)>0$, and the pressure function is assumed to be smooth, positive, and strictly increasing.  For a vector function $u\in \Rn{3}$ we define the symmetric gradient by $(\sg u)_{ij} =   \p_i u_j + \p_j u_i$ for $i,j=1,2,3$;  its deviatoric (trace-free) part  is then
\begin{equation}
 \sgz u = \sg u - \frac{2}{3} \diverge{u} I,
\end{equation}
where $I$ is the $3 \times 3$ identity matrix.  The viscous stress tensor in each fluid is then given by
\begin{equation}
\S_\pm(\tilde{u}_\pm) := \mu_\pm \sgz \tilde u_\pm +\mu'_\pm \diverge\tilde{u}_\pm I,
\end{equation}
where $\mu_\pm$ is the shear viscosity and  $\mu'_\pm$ is the bulk viscosity; we assume these satisfy the usual physical conditions
\begin{equation}\label{viscosity}
\mu_\pm>0,\quad \mu_\pm'\ge 0.
\end{equation}
The tensor $P_\pm(\tilde{\rho}_\pm) I-\S_\pm(\tilde{u}_\pm)$ is known as the stress tensor.  The divergence of a symmetric tensor $\mathbb{M}$ is defined to be the vector with components $(\diverge \mathbb{M})_i = \p_j \mathbb{M}_{ij}$.  Note then that
\begin{equation}
 \diverge\left( P_\pm(\tilde{\rho}_\pm) I-\S_\pm(\tilde{u}_\pm) \right) = \nab P_\pm(\tilde{\rho}_\pm) - \mu_\pm \Delta \tilde{u}_\pm - \left(\frac{\mu_\pm}{3} + \mu_\pm' \right)\nab \diverge{\tilde{u}_\pm}.
\end{equation}

For each $t>0$ we require that $(\tilde{\rho}_\pm,\tilde{u}_\pm,\eta_\pm)$ satisfy the following equations
\begin{equation}\label{ns_euler}
\begin{cases}
\partial_t\tilde{\rho}_\pm+\diverge (\tilde{\rho}_\pm \tilde{u}_\pm)=0 & \text{in }\Omega_\pm(t)
\\\tilde{\rho}_\pm   (\partial_t\tilde{u}_\pm  +   \tilde{u}_\pm \cdot \nabla \tilde{u}_\pm ) +\nab P_\pm(\tilde{\rho}_\pm) -  \diverge \S_\pm(\tilde{u}_\pm) =-g\tilde{\rho}_\pm e_3 & \text{in } \Omega_\pm(t)
\\\partial_t\eta_\pm=\tilde{u}_{3,\pm}-\tilde{u}_{1,\pm}\partial_{y_1}\eta_\pm-\tilde{u}_{2,\pm}\partial_{y_2}\eta_\pm &\hbox{on } \Gamma_\pm(t)
\\(P_+(\tilde{\rho}_+)I-\S_+(\tilde{u}_+))n_+=p_{atm}n_+-\sigma_+ \mathcal{H}_+ n_+ &\hbox{on }\Gamma_+(t)
\\ (P_+(\tilde{\rho}_+)I-\S_+( \tilde{u}_+))n_-=(P_-(\tilde{\rho}_-)I-\S_-( \tilde{u}_-))n_-+ \sigma_- \mathcal{H}_-n_- &\hbox{on }\Gamma_-(t) \\
\tilde{u}_+=\tilde{u}_-  &\text{on }\Gamma_-(t) \\
\tilde{u}_-=0 &\text{on }\Sigma_b.
\end{cases}
\end{equation}
In the equations $-g \tilde{\rho}_\pm e_3$ is the gravitational force with the constant $g>0$ the acceleration of gravity and $e_3$ the vertical unit vector. The constant $p_{atm}>0$ is the atmospheric pressure, and we take $\sigma_\pm\ge 0$ to be the constant coefficients of surface tension. In this paper, we let $\nabla_\ast$ denote the horizontal gradient, $\diverge_\ast$ denote the horizontal divergence and $\Delta_\ast$ denote the horizontal Laplace operator. Then the upward-pointing unit normal of $\Gamma_\pm(t)$, $n_\pm$,  is given by
\begin{equation}
n_\pm=\frac{(-\nabla_\ast\eta_\pm,1)}
{\sqrt{1+|\nabla_\ast\eta_\pm|^2}},
\end{equation}
and  $\mathcal{H}_\pm$, twice the mean curvature of the surface $\Gamma_\pm(t)$, is given by the formula
\begin{equation}
\mathcal{H}_\pm=\diverge_\ast\left(\frac{\nabla_\ast\eta_\pm}
{\sqrt{1+|\nabla_\ast\eta_\pm|^2}}\right).
\end{equation}
The third equation in \eqref{ns_euler} is called the kinematic boundary condition since it implies that the free surfaces are advected with the fluids. The boundary equations in \eqref{ns_euler} involving the stress tensor are called the dynamic boundary conditions. Notice that on $\Gamma_-(t)$, the continuity of velocity, $\tilde{u}_+ = \tilde{u}_-$,  means that it is the common value of $\tilde{u}_\pm$ that advects the interface. For a more physical description of the equations and the boundary conditions in \eqref{ns_euler}, we refer to \cite{3WL}.

To complete the statement of the problem, we must specify the initial conditions. We suppose that the initial surfaces $\Gamma_\pm(0)$ are given by the graphs of the functions $\eta_\pm(0)$, which yield the open sets $\Omega_\pm(0)$ on which we specify the initial data for the density, $\tilde{\rho}_\pm(0): \Omega_\pm(0) \rightarrow  \mathbb{R}^+$, and the velocity, $\tilde{u}_\pm(0): \Omega_\pm(0) \rightarrow  \mathbb{R}^3$. We will assume that $\ell+\eta_+(0)>\eta_-(0)>-b $ on $\mathrm{T}^2$, which means that at the initial time the boundaries do not intersect with each other.

\subsection{Equilibria}

We seek a steady-state equilibrium solution to \eqref{ns_euler} with $\tilde{u}_\pm=0, \eta_\pm =0$, and the equilibrium domains given by
\begin{equation}
\Omega_+=\{y\in   \mathrm{T}^2\times \mathbb{R}\mid 0 < y_3< \ell \} \text{ and }
\Omega_-=\{y\in  \mathrm{T}^2\times \mathbb{R}\mid  -b <y_3< 0 \}.
\end{equation}
Then \eqref{ns_euler} reduces to an ODE for the equilibrium densities  $\tilde\rho_\pm = \bar{\rho}_\pm(y_3)$:
\begin{equation}\label{steady}
\begin{cases}
\displaystyle\frac{d(P_+ (\bar{\rho}_+ ))}{dy_3} = -g\bar{\rho}_+, & \text{for }y_3 \in (0,\ell), \\
\displaystyle\frac{d(P_- (\bar{\rho}_- ))}{dy_3} = -g\bar{\rho}_-, & \text{for } y_3 \in (-b,0), \\
P_+(\bar{\rho}_+(\ell)) = p_{atm}, \\
P_+(\bar{\rho}_+(0))  =P_-(\bar{\rho}_-(0)).
\end{cases}
\end{equation}
The system \eqref{steady} admits a solution $\bar{\rho}_\pm >0$ if and only if we assume that the equilibrium heights $b,\ell>0$, the pressure laws $P_\pm$, and the atmospheric pressure $p_{atm}$ satisfy a collection of admissibility conditions, which  we enumerate below.

First, in order for the third equation in \eqref{steady} to be achievable, we must have that $p_{atm} \in P_+(\Rn{+})$.  This then allows us to uniquely define
\begin{equation}
 \bar{\rho}_1 := P_+^{-1}(p_{atm}) > 0.
\end{equation}
We then introduce the enthalpy function $h_+:(0,\infty) \to \Rn{}$ given by
\begin{equation}\label{h'}
 h_+(z) = \int_{\bar{\rho}_1}^z \frac{P'_+(r)}{r}dr,
\end{equation}
which is smooth, strictly increasing, and positive on $(\bar{\rho}_1, \infty)$.  The first ODE in \eqref{steady} then reduces to $d h_+(\bar{\rho}_+)/dy_3 = -g$, subject to the condition that $\bar{\rho}_+(\ell) = \bar{\rho}_1$, and hence
\begin{equation}\label{rho_plus_def}
 \bar{\rho}_+(y_3) = h_+^{-1}(h_+(\bar{\rho}_1) + g(\ell-y_3)) =  h_+^{-1}( g(\ell-y_3)).
\end{equation}
This gives a well-defined, smooth, and decreasing function $\bar{\rho}_+:[0,\ell] \to [\bar{\rho}_1,\infty)$ if and only if
\begin{equation}
 g \ell \in h_+((\bar{\rho}_1,\infty)) \Leftrightarrow 0 < \ell < \frac{1}{g} \int_{\bar{\rho}_1}^\infty \frac{P'_+(r)}{r}dr.
\end{equation}
The equation \eqref{rho_plus_def} yields the value of the upper density at the equilibrium internal interface:
\begin{equation}
 \bar{\rho}^+ := \bar{\rho}_+(0) = h_+^{-1}(g\ell).
\end{equation}
Then the last equation in \eqref{steady} is achievable if and only if $P_+(\bar{\rho}_+(0)) = P_+(h_+^{-1}(g\ell)) \in P_-(\Rn{+})$ and in this case yields the value of the lower density at the equilibrium internal interface:
\begin{equation}
  \bar{\rho}^- := \bar{\rho}_-(0) = P_-^{-1}(P_+(\bar{\rho}^+))  = P_-^{-1}(h_+^{-1}(g\ell)).
\end{equation}
Then we define $h_-:(0,\infty) \to \Rn{}$ via
\begin{equation}
  h_-(z) = \int_{\bar{\rho}^-}^z \frac{P'_-(r)}{r}dr
\end{equation}
and solve the  ODE for $\bar{\rho}_-(y_3)$ as before to obtain
\begin{equation}
 \bar{\rho}_-(y_3) = h_-^{-1}(h_-(\bar{\rho}^-) - g y_3 ) =  h_-^{-1}( -gy_3).
\end{equation}
This gives a well-defined, smooth, and decreasing function $\bar{\rho}_-:[-b,0] \to [\bar{\rho}^-,\infty)$ if and only if
\begin{equation}
 g b \in h_-((\bar{\rho}^-,\infty)) \Leftrightarrow 0 < b < \frac{1}{g} \int_{\bar{\rho}^-}^\infty \frac{P'_-(r)}{r}dr.
\end{equation}

For the sake of clarity we now summarize the admissibility conditions that are necessary and sufficient for the existence of an equilibrium:
\begin{enumerate}
 \item $p_{atm} \in P_+(\Rn{+})$, which defines $\bar{\rho}_1 := P_+^{-1}(p_{atm})$.
 \item $\ell$ satisfies
\begin{equation}
  0 < \ell < \frac{1}{g} \int_{\bar{\rho}_1}^\infty \frac{P'_+(r)}{r}dr,
\end{equation}
which defines $\bar{\rho}^+ := h_+^{-1}(g\ell)$.
 \item $\ell$ also satisfies  $P_+(h_+^{-1}(g\ell)) \in P_-(\Rn{+})$, which defines $\bar{\rho}^- := P_-^{-1}(h_+^{-1}(g\ell))$.
 \item $b$ satisfies
\begin{equation}
 0 < b < \frac{1}{g} \int_{\bar{\rho}^-}^\infty \frac{P'_-(r)}{r}dr.
\end{equation}
\end{enumerate}

As an example, consider the case
\begin{equation}
 P_\pm(z) = K_\pm z^{\alpha_\pm} \text{ for } K_\pm >0 \text{ and } \alpha_\pm >1.
\end{equation}
Then for any $p_{atm}, \ell, b>0$ we may solve to see that
\begin{equation}
\begin{split}
 \bar{\rho}_+(y_3) &= \left[\left(\frac{p_{atm}}{K_+} \right)^{(\alpha_+-1)/\alpha_+} +\frac{g(\alpha_+-1)}{K_+ \alpha_+}(\ell-y_3) \right]^{1/(\alpha_+ -1)} \\
 \bar{\rho}_-(y_3) &= \left[ \left( \frac{K_+ (\bar{\rho}_+(0))^{\alpha_+} }{K_-}  \right)^{(\alpha_--1)/\alpha_-}  +\frac{g(\alpha_--1)}{K_- \alpha_-}(-y_3)  \right]^{1/(\alpha_- - 1)}
\end{split}
\end{equation}
are the desired solutions to \eqref{steady}.  Note in particular that in this case the admissibility conditions are satisfied trivially.

Throughout the rest of the paper we will assume that the above admissibility conditions are satisfied for a fixed $\ell, b>0$, which then uniquely determine the equilibrium density $\rho_\pm(y_3)$.  In turn these fix the equilibrium masses
\begin{equation}\label{eq_mass}
\begin{split}
M_+ &= \int_{\Omega_+} \bar{\rho}_+ = (4 \pi^2 L_1 L_2) \int_0^\ell \bar{\rho}_+(y_3) dy_3 = \frac{4 \pi^2 L_1 L_2}{g} (P_+(h_+^{-1}(g\ell)) - p_{atm}) ,\\
 M_- &= \int_{\Omega_-} \bar{\rho}_- = (4 \pi^2 L_1 L_2) \int_{-b}^0 \bar{\rho}_-(y_3) dy_3 = \frac{4 \pi^2 L_1 L_2}{g} ( P_-(h_-^{-1}(gb))- P_+(h_+^{-1}(g\ell))).
\end{split}
\end{equation}
It is perhaps more natural to specify $M_+, M_- >0$ and then to determine $\ell, b >0$ in terms of $M_+, M_-, p_{atm}$ and the pressure laws by way of \eqref{eq_mass}.  This can be done under the assumption that $M_+, M_-, p_{atm}$ satisfy certain admissibility conditions similar to those written above for $\ell, b, p_{atm}$.  For the sake of brevity we will neglect to write these explicitly.

Notice that the equilibrium density can jump across the internal interface. The jump, which we denote by
\begin{equation}\label{rho+-}
\rj := \bar{\rho}_+(0)-\bar{\rho}_-(0)= \bar{\rho}^+ - \bar{\rho}^-,
\end{equation}
is of fundamental importance in the the analysis of solutions to \eqref{ns_euler} near equilibrium.  Indeed, if $\rj > 0$ then the upper fluid is heavier than the lower fluid along the equilibrium interface, and the fluid is susceptible to the well-known Rayleigh-Taylor gravitational instability.

\subsection{History and known results}

 Free boundary problems in fluid mechanics have attracted much interest in the mathematical community.  A thorough survey of the literature would prove impossible here, so we will primarily mention the work most relevant to our present setting, namely that related to layers of viscous fluid.   We refer to the review of Shibata and Shimizu \cite{ShSh} for a more proper survey of the literature.

The dynamics of a single layer of viscous incompressible fluid lying above a rigid bottom, i.e. the incompressible viscous surface wave problem, have attracted the attention of many mathematicians since the pioneering work of Beale \cite{B1}.  For the case without surface tension, Beale \cite{B1} proved the local well-posedness in the Sobolev spaces.   Hataya \cite{H} obtained the global existence of small, horizontally periodic solutions with an algebraic decay rate in time.    Guo and Tice \cite{GT_per, GT_inf,GT_lwp} developed a two-tier energy method to prove global well-posedness and decay of this problem.  They proved that if the free boundary is horizontally infinite, then the solution decays to equilibrium at an algebraic rate; on the other hand, if the free boundary is horizontally periodic, then the solution decays at an almost exponential rate.  The proofs were subsequently refined by the work of Wu \cite{Wu}.  For the case with surface tension, Beale \cite{B2} proved global well-posedness of the problem, while Allain \cite{A} obtained a local existence theorem in two dimensions using a different method.  Bae \cite{B} showed the global solvability in Sobolev spaces via energy methods.  Beale and Nishida \cite{BN} showed that the solution obtained in \cite{B2} decays in time with an optimal algebraic decay rate.   Nishida, Teramoto and Yoshihara \cite{NTY}  showed the global existence of periodic solutions with an exponential decay rate in the case of a domain with a flat fixed lower boundary.  Tani \cite{Ta} and Tani and Tanaka \cite{TT} also discussed the solvability of the problem with or without surface tension by using methods developed by Solonnikov  in \cite{So,So_2,So_3}.  Tan and Wang \cite{TW} studied the vanishing surface tension limit of the problem.

There are fewer results on  two-phase incompressible problems, i.e. the incompressible viscous surface-internal wave or internal wave problems. Hataya \cite{H2} proved an existence result for a periodic free interface problem with surface tension, perturbed around Couette flow; he showed the local  existence of small  solution for any physical constants, and the existence of exponentially decaying small solution if the viscosities of the two fluids are sufficiently large and their difference is small. Pr\"uss and Simonett \cite{PS} proved the local well-posedness of a free interface problem with surface tension in which two layers of viscous fluids fill  the whole space and are separated by a horizontal interface.  For two horizontal fluids of finite depth with surface tension,  Xu and Zhang \cite{XZ} proved  the local solvability for general  data and global solvability for data near the equilibrium state using Tani and Tanaka's method.  Wang and Tice \cite{WT} and Wang, Tice and Kim \cite{WTK} adapted the two-tier energy methods of \cite{GT_per, GT_inf,GT_lwp} to develop the nonlinear Rayleigh-Taylor instability theory for the problem, proving the existence of a sharp stability criterion given in terms of the surface tension coefficient, gravity, periodicity lengths, and $\rj$.

The free boundary problems corresponding to a single horizontally periodic layer of compressible viscous fluid with surface tension have been studied by several authors.  Jin \cite{jin} and Jin-Padula \cite{jin_padula} produced global-in-time solutions using Lagrangian coordinates, and Tanaka and Tani \cite{tanaka_tani} produced global solutions with temperature dependence.  However, to the best of our knowledge, the problem for two layers of compressible viscous fluids remains open except for the linear analysis of Guo and Tice \cite{GT_RT}.

The two-layer problem is important because it allows for the development of the classical Rayleigh-Taylor instability \cite{3R,3T}, at least when the equilibrium has a heavier fluid on top and a lighter one below and there is a downward gravitational force.   In this paper we identify a stability criterion and show that global solutions exist and decay to equilibrium when the criterion is met.  In our companion paper \cite{JTW_nrt} we show that the stability criterion is sharp, as in the incompressible case \cite{WT, WTK}, and that the Rayleigh-Taylor instability persists at the nonlinear level (the linear analysis was developed in \cite{GT_RT}).

\section{Reformulation and main results}

In this section, we reformulate the system \eqref{ns_euler} by using a special flattening coordinate
transformation and deliver the precise statements of the main results.

\subsection{Reformulation in flattened coordinates}

The movement of the free surfaces $\Gamma_\pm(t)$ and the subsequent change of the domains $\Omega_\pm(t)$ create numerous mathematical difficulties. To circumvent these, we will switch to coordinates in which the boundaries and the domains stay fixed in time. Since we are interested in the nonlinear stability of the equilibrium state, we will use the equilibrium domain.  We will not use a Lagrangian coordinate transformation, but rather utilize a special flattening coordinate transformation motivated by Beale \cite{B2}.

To this end, we define the fixed domain
\begin{equation}
\Omega = \Omega_+\cup\Omega_-\text{ with }\Omega_+:=\{0<x_3<\ell \} \text{ and } \Omega_-:=\{-b<x_3<0\},
\end{equation}
for which we have written the coordinates as $x\in \Omega$. We shall write $\Sigma_+:=\{x_3= \ell\}$ for the upper boundary, $\Sigma_-:=\{x_3=0\}$ for the internal interface and $\Sigma_b:=\{x_3=-b\}$ for the lower boundary.  Throughout the paper we will write $\Sigma = \Sigma_+ \cup \Sigma_-$.   We think of $\eta_\pm$ as a function on $\Sigma_\pm$ according to $\eta_+: (\mathrm{T}^2\times\{\ell\}) \times \mathbb{R}^{+} \rightarrow\mathbb{R}$ and $\eta_-:(\mathrm{T}^2\times\{0\}) \times \mathbb{R}^{+} \rightarrow \mathbb{R}$, respectively. We will transform the free boundary problem in $\Omega(t)$ to one in the fixed domain $\Omega $ by using the unknown free surface functions $\eta_\pm$. For this we define
\begin{equation}
\bar{\eta}_+:=\mathcal{P}_+\eta_+=\text{Poisson extension of }\eta_+ \text{ into }\mathrm{T}^2 \times \{x_3\le \ell\}
\end{equation}
and
\begin{equation}
\bar{\eta}_-:=\mathcal{P}_-\eta_-=\text{specialized Poisson extension of }\eta_-\text{ into }\mathrm{T}^2 \times \mathbb{R},
\end{equation}
where $\mathcal{P}_\pm$ are defined by \eqref{P+def} and \eqref{P-def}. Motivated by Beale \cite{B2}, the Poisson extensions $\bar{\eta}_\pm$ allow us to flatten the coordinate domains via the following special coordinate transformation:
\begin{equation}\label{cotr}
\Omega_\pm \ni x\mapsto(x_1,x_2, x_3+ \tilde{b}_1\bar{\eta}_++\tilde{b}_2\bar{\eta}_-):=\Theta (t,x)=(y_1,y_2,y_3)\in\Omega_\pm(t),
\end{equation}
where we have chosen $\tilde{b}_1=\tilde{b}_1(x_3), \tilde{b}_2=\tilde{b}_2(x_3)$ to be two smooth functions in $\mathbb{R}$ that satisfy
\begin{equation}\label{b function}
\tilde{b}_1(0)=\tilde{b}_1(-b)=0, \tilde{b}_1(\ell)=1\text{ and }\tilde{b}_2(\ell)=\tilde{b}_2(-b)=0, \tilde{b}_2(0)=1.
\end{equation}
Note that $\Theta(\Sigma_+,t)=\Gamma_+(t),\ \Theta (\Sigma_-,t)=\Gamma_-(t)$ and $\Theta(\cdot,t) \mid_{\Sigma_b} = Id \mid_{\Sigma_b}$.

Note that if $\eta $ is sufficiently small (in an appropriate Sobolev space), by Lemma \ref{Poi} and usual Sobolev embeddings, then the mapping $\Theta $ is a diffeomorphism.  This allows us to transform the problem \eqref{ns_euler} to one in the fixed spatial domain $\Omega$ for each $t\ge 0$. In order to write down the equations in the new coordinate system, we compute
\begin{equation}\label{A_def}
\begin{array}{ll} \nabla\Theta  =\left(\begin{array}{ccc}1&0&0\\0&1&0\\A  &B  &J  \end{array}\right)
\text{ and }\mathcal{A}  := \left(\nabla\Theta
^{-1}\right)^T=\left(\begin{array}{ccc}1&0&-A   K  \\0&1&-B   K  \\0&0&K
\end{array}\right)\end{array}.
\end{equation}
Here the components in the matrix are
\begin{equation}\label{ABJ_def}
A  =\p_1\theta ,\
B  =\p_2\theta,\
J = 1 + \p_3\theta,\  K  =J^{-1},
\end{equation}
where we have written
\begin{equation}\label{theta}
\theta:=\tilde{b}_1\bar{\eta}_++\tilde{b}_2\bar{\eta}_-.
\end{equation}
Notice that $J={\rm det}\, \nabla\Theta $ is the Jacobian of the coordinate transformation. It is straightforward to check that, because of how we have defined $\bar{\eta}_-$ and $\Theta $, the matrix $\mathcal{A}$ is regular across the interface $\Sigma_-$.

We now define the density $\rho_\pm$ and the velocity $u_\pm$ on $\Omega_\pm$ by the compositions $\rho_\pm(x,t)=\tilde \rho_\pm(\Theta_\pm(x,t),t)$ and $  u_\pm(x,t)=\tilde u_\pm(\Theta_\pm(x,t),t)$. Since the domains $\Omega_\pm$ and the boundaries $\Sigma_\pm$ are now fixed, we henceforth consolidate notation by writing $f$ to refer to $f_\pm$ except when necessary to distinguish the two; when we write an equation for $f$ we assume that the equation holds with the subscripts added on the domains $\Omega_\pm$ or $\Sigma_\pm$. To write the jump conditions on $\Sigma_-$, for a quantity $f=f_\pm$, we define the interfacial jump as
\begin{equation}
\jump{f} := f_+ \vert_{\{x_3=0\}} - f_- \vert_{\{x_3=0\}}.
\end{equation}
Then in the new coordinates, the PDE \eqref{ns_euler} becomes the following system for $(\rho,u,\eta)$:
\begin{equation}\label{ns_geometric}
\begin{cases}
\partial_t \rho-K\p_t\theta\p_3\rho +\diverge_\a (  {\rho}   u)=0 & \text{in }
\Omega  \\
\rho (\partial_t    u -K\p_t\theta\p_3 u+u\cdot\nabla_\a u  ) + \nabla_\a P ( {\rho} )    -\diva \S_\a (u) =- g\rho e_3 & \text{in }
\Omega
\\ \partial_t \eta = u\cdot \n &
\text{on }\Sigma
\\ (P ( {\rho} ) I- \S_{\a}(u))\n =p_{atm}\n -\sigma_+  \mathcal{H} \n  &\hbox{on }\Sigma_+ \\
\jump{P ( {\rho} ) I- \S_\a(u)}\n
= \sigma_-  \mathcal{H} \n  &\hbox{on }\Sigma_- \\
\jump{u}=0   &\hbox{on }\Sigma_-\\  {u}_- =0 &\text{on }\Sigma_b.
\end{cases}
\end{equation}
Here we have written the differential operators $\naba$, $\diva$, and $\da$ with their actions given by
\begin{equation}
 (\naba f)_i := \a_{ij} \p_j f, \diva X := \a_{ij}\p_j X_i, \text{ and }\da f := \diva \naba f
\end{equation}
for appropriate $f$ and $X$.  We have also written
\begin{equation}\label{n_def}
\n := (-\p_1 \eta, - \p_2 \eta,1)
\end{equation}
for the non-unit normal to $\Sigma(t)$, and we have written
\begin{multline}
\S_{\a,\pm}(u): =\mu_\pm \sgz_{\a} u+ \mu_\pm' \diva u I, \qquad \sgz_{\a} u = \sg_{\a} u - \frac{2}{3} \diva u I,\\
 \text{and } (\sg_{\a} u)_{ij} =  \a_{ik} \p_k u_j + \a_{jk} \p_k u_i.
\end{multline}
Note that if we extend $\diva$ to act on symmetric tensors in the natural way, then $\diva \S_{\a} u =\mu\Delta_\a u+(\mu/3+\mu')\nabla_\a \diverge_\a u$. Recall that $\a$ is determined by $\eta$ through \eqref{A_def}. This means that all of the differential operators in \eqref{ns_geometric} are connected to $\eta$, and hence to the geometry of the free surfaces.

\begin{Remark}
The equilibrium state given by \eqref{steady} corresponds to the static solution $(\rho, u, \eta)=(\bar{\rho},0,0)$ of \eqref{ns_geometric}.
\end{Remark}

\subsection{Local well-posedness}

As mentioned in the introduction, the purpose of this article is to establish  a sharp nonlinear stability criterion for the equilibrium state $(\bar{\rho},0,0)$ in the compressible viscous surface-internal wave problem \eqref{ns_geometric}.  Since the equilibrium velocity and free surfaces are trivial, it suffices to study $u$ and $\eta$ directly.  However, since the equilibrium density is nontrivial, we must define a perturbation of it.  We do so by defining
\begin{equation}\label{q_def}
 q := \rho - \bar{\rho} - \p_3 \bar{\rho} \theta.
\end{equation}
Note that if we consider the natural density perturbation $\varrho=\rho-\bar\rho$, then the resulting PDE has a linear term $-g\varrho e_3$ in the right hand side of the second momentum equation, which is not convenient for employing elliptic regularity theory; our special density perturbation $q$ defined by \eqref{q_def}  eliminates this linear term (see \eqref{adv3.3}). For the sake of brevity, we will not record here the equations for $(q,u,\eta)$ obtained from plugging this perturbation into \eqref{ns_geometric}; they may be found later in the paper in \eqref{ns_perturb}.

Before discussing our main results we mention the local well-posedness theory for \eqref{ns_geometric}, perturbed around the equilibrium state.  The energy space in which local solutions exist is defined in terms of certain functionals, which are sums of Sobolev norms.  We refer to  Section \ref{def-ter} for details on our notation for Sobolev spaces and norms.  For a generic integer $n\ge 3$, we define the  energy as
\begin{multline}\label{p_energy_def}
 \se{n}^\sigma := \sum_{j=0}^{n}  \ns{\dt^j u}_{2n-2j} + \ns{\q}_{2n}  + \sum_{j=1}^{n} \ns{ \dt^j \q}_{2n-2j+1} + \sum_{j=1}^n \ns{\dt^j \eta}_{2n-2j+3/2}
\\
+\mathcal{H}(\rj) \min\{1,\sigma_+,\sigma_- - \sigma_c\}  \ns{ \eta}_{2n+1}
+\mathcal{H}(-\rj)\left[   \ns{\eta}_{2n} + \min\{1,\sigma\}  \ns{\eta}_{2n+1}  \right]
\end{multline}
and the corresponding dissipation as
\begin{multline}\label{p_dissipation_def}
  \sd{n}^\sigma :=  \sum_{j=0}^{n} \ns{\dt^j u}_{2n-2j+1}  +   \ns{\q}_{2n} + \ns{\dt \q}_{2n-1} + \sum_{j=2}^{n+1} \ns{\dt^j \q}_{2n-2j+2}
\\
+  \mathcal{H}(\rj) \min\{1,\sigma_+, \sigma_--\sigma_c, \sigma_+^2,(\sigma_- - \sigma_c)^2\} \ns{ \eta}_{2n+3/2}
\\
+ \mathcal{H}(-\rj) \left(\ns{\eta}_{2n-1/2} + \min\{1,\sigma^2\} \ns{ \eta}_{2n+3/2} \right)
\\
+\ns{\partial_t \eta}_{2n-1/2} + \sigma^2\ns{\partial_t \eta}_{2n+1/2}
 + \sum_{j=2}^{n+1} \ns{\dt^j \eta}_{2n-2j+5/2},
\end{multline}
where here $\mathcal{H} = \chi_{(0,\infty)}$ is the Heaviside function.  Throughout the paper we will consider both $n=2N$ and $n=N+2$ for the integer $N\ge 3$.  We also define
\begin{equation}\label{fff}
\f:=  \ns{\eta}_{4N+1/2}.
\end{equation}
Note that we have included the surface tension coefficient $\sigma_\pm$ in the definitions \eqref{p_energy_def} and \eqref{p_dissipation_def} so that we will be able to treat the cases with and without surface tension together.  Throughout the paper, we assume that $\sigma_\pm$ have the same sign: either $\sigma_+=\sigma_-=0$ or else $\sigma_\pm>0$. Note that this assumption excludes the cases of $\sigma_-=0, \;\sigma_+>0$ and $\sigma_->0, \;\sigma_+=0$.

In order to produce local solutions on $[0,T]$ for which $\sup_{0\le t \le T} \se{2N}^\sigma(t)$ remains bounded, we must assume that the initial data $(q_0,u_0,\eta_0)$ satisfy $2N$ systems of compatibility conditions (a system for each time derivative of order $0$ to $2N-1$).  These are natural for solutions to \eqref{ns_geometric} in our energy space, but they are cumbersome to write.  We shall neglect to record them in this paper and instead refer to our companion paper \cite{JTW_lwp} for their precise enumeration.

Our local existence results for \eqref{ns_geometric} is then as follows.

\begin{theorem}\label{lwp}
Assume that either $\sigma_\pm =0$ or else $\sigma_\pm >0$.  Let $N \ge 3$ be an integer.  Assume that the initial data $(q_0, u_0, \eta_0)$ satisfy $\ns{u_0}_{4N} + \ns{q_0}_{4N}+ \ns{\eta_0}_{4N+1/2} + \sigma\ns{ \nab_\ast \eta_0}_{4N} < \infty$ as well as the compatibility conditions enumerated in \cite{JTW_lwp}.  There exist $0 < \delta_0,T_0 < 1$ so that if $\ns{u_0}_{4N} + \ns{q_0}_{4N}+ \ns{\eta_0}_{4N+1/2} + \sigma\ns{ \nab_\ast \eta_0}_{4N}\le \delta_0$ and $0 < T \le T_0$, then there exists a unique triple $(q,u,\eta)$, achieving the initial data, so that $(\rho=\bar{\rho}+q + \p_3 \bar{\rho} \theta,u,\eta)$ solve \eqref{ns_geometric}.  The solution obeys the estimates
\begin{multline}\label{lwp_01}
 \sup_{0 \le t \le T} \se{2N}^\sigma(t) +\sup_{0 \le t \le T}\f(t)   + \int_0^T \sd{2N}^\sigma(t)dt + \int_0^T \ns{\rho(t)J(t) \dt^{2N+1} u(t)}_{\Hd} dt \\
\ls \ks\left(\ns{u_0}_{4N} + \ns{q_0}_{4N}+ \ns{\eta_0}_{4N+1/2} + \sigma\ns{ \nab_\ast \eta_0}_{4N}\right),
\end{multline}
where $\ks$ is a constant of the form \eqref{ks_def}.

\end{theorem}

Theorem \ref{lwp} can be deduced readily from Theorem 2.1 of \cite{JTW_lwp}. Indeed, Theorem 2.1 of \cite{JTW_lwp} is stated in more general form, where we only require $\norm{\eta_0}_{4N-1/2}$ to be small, and no smallness condition is imposed on $u_0$ or $q_0$.  We record the version of local well-posedness in Theorem \ref{lwp} so that it can be employed directly in proving our global well-posedness result.

\subsection{Main results}
To state our results, we now define some quantities.  For a given jump value in the equilibrium density $\rj $, we first define the critical surface tension value by
\begin{equation}\label{sigma_c}
\sigma_c:= \rj g\max\{L_1^2,L_2^2\},
\end{equation}
where we recall that $L_1,L_2$ are the periodicity lengths.  Note that this is the same critical value identified for the incompressible problem in \cite{WT,WTK}.

For the global well-posedenss theory, we assume further that the initial masses are equal to the equilibrium masses:
\begin{equation}\label{masscon}
\int_{\Omega_+} \rho_{0,+} J_{0,+}  = \int_{\Omega_+} \bar{\rho}_+\text{ and }
\int_{\Omega_-} \rho_{0,-} J_{0,-} = \int_{\Omega_-} \bar{\rho}_- .
\end{equation}
Note that for sufficiently regular solutions to the problem \eqref{ns_geometric} (and in particular our global solutions), the condition \eqref{masscon} persists in time, i.e.
\begin{equation}\label{masscont}
\int_{\Omega_+} \rho_+(t) J_+(t)  = \int_{\Omega_+} \bar{\rho}_+\text{ and }
\int_{\Omega_-} \rho_-(t) J_-(t) = \int_{\Omega_-} \bar{\rho}_- .
\end{equation}
Indeed, from the first continuity equation and the equations $\dt \eta=u\cdot \n$ on $\Sigma$ and $u_-=0$ on $\Sigma_b$, we find, after integrating by parts, that
\begin{equation}\label{masscontt}
\begin{split}
\frac{d}{dt}\int_{\Omega_+} \rho_+  J_+  & =\int_{\Omega_+}\rho_+  \dt J_+ +\dt \rho_+  J_+ =\int_{\Omega_+} \rho_+\dt\pa_3\theta_++\p_t\theta_+\p_3\rho_+ +J_+\diverge_{\a_+} (  {\rho}_+   u_+)
\\
&=\int_{\Sigma_+}\rho_+ (\dt\eta_+-u_+\cdot\n_+ )-\int_{\Sigma_-}\rho_+ (\dt\eta_--u\cdot\n_-)=0,
\end{split}
\end{equation}
and
\begin{equation}\label{masscontt2}
\frac{d}{dt}\int_{\Omega_-} \rho_-  J_- =\int_{\Sigma_-}\rho_- (\dt\eta_--u\cdot\n_-)=0.
\end{equation}
The conservation of masses \eqref{masscont} plays a key role when $\rj>0$, as it allows us to show the positivity of the energy in the stability regime. Moreover,  we are interested in showing $\eta(t)\rightarrow0$ and $q(t)\rightarrow0$ (and hence $\rho(t)\rightarrow\bar\rho$) as $t\rightarrow\infty$ in a strong sense; due to the conservation of \eqref{masscont}, we cannot expect this unless \eqref{masscon} is satisfied at initial time.

We will employ a functional that combines control of the energy functionals given by \eqref{p_energy_def}, the dissipation functionals given by \eqref{p_dissipation_def}, and the auxiliary energy given by \eqref{fff}.  We define
\begin{equation}\label{G_def}
\gs(t) := \sup_{0 \le r \le t} \se{2N}^\sigma(r) + \int_0^t \sd{2N}^\sigma(r) dr + \sup_{0 \le r \le t} (1+r)^{4N-8} \se{N+2}^\sigma(r) + \sup_{0 \le r \le t} \frac{\f(r)}{(1+r)}.
\end{equation}

We can now state our global well-posedness result for \eqref{ns_geometric} with or without surface tension.

\begin{theorem}\label{th_gwp}
Suppose that one of the following three cases holds:
\begin{align}
  \rj &<0,  \sigma_+ =0, \sigma_- = 0 \label{th_gwp_01} \\
  \rj &<0,  \sigma_+ >0, \sigma_- > 0 \label{th_gwp_02} \\
  \rj &>0,  \sigma_+ >0, \sigma_- > \sigma_c, \label{th_gwp_03}
\end{align}
where in the last case $\sigma_c>0$ is defined by \eqref{sigma_c}.  Let $N \ge 3$ be an integer.  Assume that the initial data $(q_0, u_0, \eta_0)$ satisfy $\se{2N}^\sigma(0) + \f(0) < \infty$ as well as the compatibility conditions of Theorem \ref{lwp} and the mass condition \eqref{masscon}. Let $\delta_0 >0$ be the constant from Theorem \ref{lwp}.

There exists a constant $0 < \kappa \le \delta_0$ such that $1/\kappa = \ks$, where $\ks$ is a constant of the form \eqref{ks_def}, such that if $\se{2N}^\sigma(0) + \f(0) \le \kappa$, then there exists a unique triple  $(q,u,\eta)$, achieving the initial data, such that $(\rho=\bar{\rho}+q + \p_3 \bar{\rho} \theta,u,\eta)$ solve \eqref{ns_geometric} on the time interval $[0,\infty)$.  The solution obeys the estimate
\begin{equation}\label{th_gwp_04}
 \gs(\infty) \le \ks\left( \se{2N}^\sigma(0) + \f(0)\right).
\end{equation}

With surface tension, i.e. cases \eqref{th_gwp_02} and \eqref{th_gwp_03}, there exists a constants $\ks$, of the form \eqref{ks_def}, and $M(\sigma,\rj)$, defined by \eqref{ex_dec_01} and \eqref{ex_dec_02}, so that
\begin{equation}\label{th_gwp_05}
\sup_{t \ge 0} \left[ \exp\left(\frac{t}{\ks  M(\sigma,\rj)} \right) \mathcal{E}_{N+2}^\sigma(t) \right] \lesssim \ks \mathcal{E}_{N+2}^\sigma(0).
\end{equation}

\end{theorem}

Theorem \ref{th_gwp} is proved later in the paper in Section \ref{section_proof}.  Before stating our second main result, some remarks are in order.

\begin{Remark}
Without surface tension (case \eqref{th_gwp_01}), the estimate \eqref{th_gwp_04} guarantees that
\begin{equation}
\sup_{t \ge 0 } (1+t)^{4N-8} \left[\ns{u(t)}_{2N+4}  +\ns{q(t)}_{2N+4}+ \ns{\eta(t)}_{2N+4}   \right] \ls (\se{2N}^\sigma(0) + \f(0)) \ls \kappa.
\end{equation}
Since $N$ may be taken to be arbitrarily large, this decay result can be regarded as an ``almost exponential'' decay rate.

With surface tension (cases \eqref{th_gwp_02} and \eqref{th_gwp_03}), the estimate \eqref{th_gwp_05} implies that
\begin{multline}
 \sup_{t \ge 0 } \left[ \exp\left(\frac{t}{\ks  M(\sigma,\rj)} \right)\left(  \ns{u(t)}_{2N+4} +\ns{q(t)}_{2N+4} +C(\sigma,\rj)\ns{ \eta(t)}_{2N+5}  \right)      \right] \\
\ls \ks (\se{2N}^\sigma(0) + \f(0)) \ls \ks \kappa,
\end{multline}
where  $C(\sigma,\rj) = \mathcal{H}(\rj) \min\{1,\sigma_+,\sigma_- - \sigma_c\}
+\mathcal{H}(-\rj) \min\{1,\sigma_+,\sigma_- \}$ and  $\mathcal{H} = \chi_{(0,\infty)}$ is the usual Heaviside function.  This is exponential decay.   We then see that one of the main effects of surface tension is that it induces a faster decay to equilibrium.
\end{Remark}

\begin{Remark}
In our companion paper \cite{JTW_nrt} we show that if the stability criterion is not satisfied, then the local solutions are nonlinearly unstable.  This establishes  sharp nonlinear stability criteria for the equilibrium state in the compressible viscous surface-internal wave problem.   We summarize these and the rates of decay to equilibrium in the following table.

\begin{displaymath}
\begin{array}{| c | c | c |  c |}
\hline
 & \rj < 0 & \rj =0 & \rj >0  \\ \hline
 \sigma_\pm =0 &
\begin{array}{c}
\textnormal{nonlinearly stable} \\
\textnormal{almost exponential decay}
\end{array} &
\textnormal{locally well-posed} &
\textnormal{nonlinearly unstable}     \\ \hline
\begin{array}{c}
0 < \sigma_+  \\ 0 < \sigma_- < \sigma_c
\end{array} &
\begin{array}{c}
\textnormal{nonlinearly stable} \\
\textnormal{exponential decay}
\end{array} &
\begin{array}{c}
\textnormal{nonlinearly stable} \\
\textnormal{exponential decay}
\end{array} &
\textnormal{nonlinearly unstable}      \\ \hline
\begin{array}{c}
0 < \sigma_+  \\  \sigma_c = \sigma_-
\end{array} &
\begin{array}{c}
\textnormal{nonlinearly stable} \\
\textnormal{exponential decay}
\end{array} &
\begin{array}{c}
\textnormal{nonlinearly stable} \\
\textnormal{exponential decay}
\end{array} &
\textnormal{locally well-posed}
\\ \hline
\begin{array}{c}
0 < \sigma_+  \\  \sigma_c < \sigma_-
\end{array} &
\begin{array}{c}
\textnormal{nonlinearly stable} \\
\textnormal{exponential decay}
\end{array} &
\begin{array}{c}
\textnormal{nonlinearly stable} \\
\textnormal{exponential decay}
\end{array} &
\begin{array}{c}
\textnormal{nonlinearly stable} \\
\textnormal{exponential decay}
\end{array}
\\ \hline
\end{array}
\end{displaymath}

\end{Remark}

\begin{Remark}
If we introduce the critical density jump value by
\begin{equation}
\rj _c:=\frac{\sigma_-}{g\max\{L_1^2,L_2^2\}}\ge 0
\end{equation}
for a given lower surface tension value $\sigma_-$, then Theorem \ref{th_gwp} characterizes  the stable equilibria in terms of $\rj _c$: the equilibrium is nonlinear stable if $\rj < \rj _c$.  In particular, this immediately implies that if
$\rj <0$,  the equilibrium is stable for all surface tension values $\sigma_\pm\geq 0$.
\end{Remark}

\begin{Remark}
 Our methods could be readily applied to the internal-wave problem, i.e. the problem \eqref{ns_geometric} posed with a rigid top in place of the upper free surface.   In this case Theorem \ref{th_gwp} again holds, and we identify a nonlinear stability criterion for the compressible viscous internal-wave problem.  This agrees with the linear stability criterion identified in \cite{GT_RT}.
\end{Remark}

\begin{Remark}
The high regularity context of Theorem \ref{th_gwp} is not entirely necessary in the case with surface tension.  Indeed, the theorem shows in this case that the integer $N$ plays no significant role in determining the decay rate, which is exponential.  We believe that we could refine the analysis in the case with surface tension to match the lower regularity context used in the incompressible case \cite{WT,WTK}.    We have chosen to forgo this approach here for two reasons.  First, it allows us to analyze the case with surface tension and the case without simultaneously.  Second, it allows us to establish the vanishing surface tension limit.
\end{Remark}

In our analysis leading to the proof of Theorem \ref{th_gwp} we have made a serious effort to track the dependence of various constants on the surface tension parameters.  The upshot of this is that our estimates are sufficiently strong to allow us to pass to the vanishing surface tension limit, establishing a sort of continuity between the problems with and without surface tension.  We state this result  now.

\begin{theorem}\label{th_vanish_st}
Let $N \ge 3$ be an integer.  Suppose that $\rj <0$ and $\sigma_\pm >0$, and assume that $(q_0^\sigma,u_0^\sigma,\eta_0^\sigma)$ satisfy the compatibility conditions of Theorem \ref{th_gwp} and the mass condition \eqref{masscon} in addition to the bound $\se{2N}^\sigma(0) + \f(0) \le \kappa$, where $\kappa$ is as in the theorem.  Further assume that $\sigma_\pm \to 0$ and that
\begin{multline}
q_0^\sigma \to q_0 \text{ in }H^{4N}(\Omega), \; u_0^\sigma \to u_0 \text{ in }H^{4N}(\Omega), \\
 \eta_0^\sigma \to \eta_0 \text{ in } H^{4N+1/2}(\Sigma),  \text{ and } \sqrt{\sigma} \nab_\ast \eta_0^\sigma \to 0 \text{ in }H^{4N}(\Sigma)
\end{multline}
as $\sigma_\pm \to 0$.

Then the following hold.
\begin{enumerate}
 \item The triple $(q_0,u_0,\eta_0)$ satisfy the compatibility conditions of Theorem \ref{th_gwp} with $\sigma_\pm =0$.
 \item Let $(q^\sigma, u^\sigma,\eta^\sigma)$ denote the global solutions to \eqref{ns_geometric} produced by Theorem \ref{th_gwp} from the data $(q_0^\sigma,u_0^\sigma,\eta_0^\sigma)$.  Then as $\sigma_\pm \to 0$, the triple  $(q^\sigma, u^\sigma,\eta^\sigma)$ converges to  $(q, u,\eta)$, where the latter triple is the unique solution to \eqref{ns_geometric} with $\sigma_\pm=0$ and initial data $(q_0,u_0,\eta_0)$.  The convergence occurs in any space into which the space of triples $(q,u, \eta)$ obeying $\mathcal{G}^0_{2N}(\infty) < \infty$ compactly embeds.
\end{enumerate}
\end{theorem}

\subsection{Strategy and plan of the paper}

Note that sufficiently regular solutions of \eqref{ns_geometric}  obey the physical energy-dissipation law
\begin{multline}\label{basic_law}
\frac{d}{dt} \left( \int_{\Omega} \frac{\rho J}{2} \abs{v}^2 +  R(\rho) J +g\rho (x_3 + \theta )J  +  \int_{\Sigma_+} p_{atm} \eta + \sigma_+ \sqrt{1+ \abs{\nab_\ast \eta}^2}  \right.
 \\
\left. +  \int_{\Sigma_-}  \sigma_- \sqrt{1+ \abs{\nab_\ast \eta}^2} \right)
+ \int_{\Omega}\frac{\mu}{2} J\abs{\sgz_\mathcal{A} v}^2+\mu'J\abs{\diverge_\mathcal{A} v}^2 dx =0,
\end{multline}
where
\begin{equation}
 R_\pm(z) = z \int_{c_\pm}^z \frac{P_\pm(s)}{s^2} ds
\end{equation}
for some choice of $c_\pm>0$ such that $R_\pm \ge 0$.   This energy-dissipation equation is too weak on its own to even control the Jacobian of our coordinate transformation, and so we must employ a higher-regularity framework for global-in-time analysis.  Our framework is based on the energy and dissipation functionals obtained from linearizing \eqref{basic_law} around the equilibrium.

With these functionals in hand, we develop a nonlinear energy method for studying the problem \eqref{ns_geometric}.  This method is similar to that employed in the incompressible setting \cite{GT_per,GT_inf,WTK}, though of course the compressibility adds many new complications.  Below we summarize some of the principal features of our analysis.

\textbf{Horizontal energy estimates:} The basic strategy in any high order energy method is to apply derivatives to the equations and use the basic energy-dissipation structure to get estimates.  This is complicated by boundary conditions, which impose restrictions on which derivatives may be applied; indeed, the differentiated problem must have the same structure as the original in order to produce the same energy-dissipation estimate.  In a moving domain this is particularly annoying, as the acceptable directions vary not only in space but also in time.  While it's possible to get around this difficulty in a moving domain (see for example the work of Bae \cite{B}), the more common approach is to change coordinates into a flattened (or at the very least unmoving) domain.  This is what we have done in \eqref{ns_geometric}.  In the flattened framework, the admissible derivatives are the temporal and horizontal spatial ones because they preserve the structure of the boundary conditions.  As such, the basic energy estimates can only control temporal and horizontal derivatives of the solution.  Control of vertical derivatives must be acquired from other methods, namely elliptic estimates and intermediate energy estimates coming from special structure in the equations.

The trade-off for flattening the coordinate domain is that the operators become far more complicated and fail to be constant-coefficient.  We are then faced with a choice: work directly with the system \eqref{ns_geometric}, or rewrite the problem as a perturbation of the corresponding constant-coefficient problem.  The benefit of the latter approach is that the differential operators commute with all derivatives, making the basic energy estimates cleaner.  The downside of this approach is that the forcing terms are of higher regularity and may be outside of what can be controlled by the energy method.  Unfortunately, this is what occurs for \eqref{ns_geometric}; estimates of the highest-order temporal derivatives do not close in the perturbed form because we cannot control the highest temporal derivatives of $u$ on the internal interface.   Because of this, we are forced to work directly with the ``geometric form'' \eqref{ns_geometric} for some of the energy estimates.  Note, though, that we can use the ``perturbed form'' for some of the energy estimates, and indeed, this is more convenient for some.

\textbf{Energy positivity:} Ensuring the positivity of the energy is one of key issues in using an energy method for global-in-time analysis.  In the case of a possible Rayleigh-Taylor instability, when $\rj >0$, the energy appearing in our analysis does not have a definite sign a priori.  It is only by imposing a condition on $\sigma_-$ that we can restore this positivity, and this actually reveals the stability criterion: it is exactly the threshold at which the energy can be guaranteed to be definite.  Actually proving this positivity amounts to using the sharp Poincar\'{e} inequality on $\mathrm{T}^2$.  This in turn requires that we control the average of $\eta$.  In the incompressible problem this is trivial since the average is constant in time, but in the compressible case, we must use a more delicate argument, employing the conservation of mass to control the average with other controlled terms.

\textbf{Intermediate energy estimates:} To go from control of horizontal derivatives to control of full derivatives we will employ elliptic estimates.  Unfortunately, the horizontal energy and dissipation do not control everything necessary for such an elliptic estimate, so we are forced to seek intermediate energy estimates.  These come from the special dissipative structure that arises by taking an appropriate linear combination of the density equation and the third component of the velocity equation in \eqref{ns_geometric}.  To the best of our knowledge, this structure was first exploited by Matsumura and Nishida \cite{MN83}.  Using this trick, we get an energy-dissipation equation for vertical derivatives of the density perturbation and derive estimates.  When combined with the basic horizontal estimates and elliptic regularity estimates for a certain one-phase Stokes problem, these yield control of intermediate energy and dissipation functionals that control nearly all the desired derivatives.

\textbf{Comparison estimate:} We gain control of all desired derivatives by way of comparison estimates.  In these we use elliptic regularity estimates to show that the intermediate energy and dissipation control the ``full'' energy and dissipation, up to some ultimately harmless error terms. The comparison estimate for the energy follows by applying the elliptic regularity theory for a two-phase Lam\'e problem.  The comparison estimate for the dissipation is particularly tricky because the horizontal dissipation itself provides no control of the free surface functions. Indeed, the hardest terms to control are $\norm{q}_0$ and $\norm{\eta}_0$.  For them, we again use the conservation of mass to construct an auxiliary function $w$ that allows us to derive the control of $\norm{q}_0$ and $\norm{\eta}_0$ together. The energy positivity mentioned above again plays a crucial role here.

\textbf{Decay estimates and the two-tier scheme:}  The basic goal in a nonlinear energy method is to derive an energy-dissipation equation of the form
\begin{equation}\label{strat_1}
 \frac{d}{dt} \mathcal{E} + \mathcal{D} \ls \mathcal{E}^\theta \mathcal{D} \text{ for some } \theta>0.
\end{equation}
In a small-energy regime (verified by the local theory and choice of small data), the right-hand term can be absorbed onto the left, resulting in an equation of the form
\begin{equation}\label{strat_2}
 \frac{d}{dt} \mathcal{E} + \hal \mathcal{D} \le 0.
\end{equation}
In the event that the dissipation is coercive over the energy space, i.e. $\mathcal{D} \ge C \mathcal{E}$, we may immediately deduce exponential decay of $\mathcal{E}$.   When we consider \eqref{ns_geometric} with surface tension and $\sigma_- > \sigma_c$ in the case $\rj >0$, we can essentially prove an estimate of the form \eqref{strat_1} and the coercivity of the dissipation, and we can deduce exponential decay of solutions.

However, when we neglect surface tension, we run into two major difficulties.  The first is that we cannot derive an estimate of the form \eqref{strat_1} because the regularity demands for $\eta$ in  the nonlinear forcing terms are half a derivative beyond the control of the energy or dissipation.  This would be disastrous if it were not for auxiliary estimate for $\eta$ arising from the kinematic transport equation.  This estimate provides enough regularity to guarantee that the nonlinearity is well-defined, but the estimate can only say that this term grows no faster than linearly in time.  This prevents us from using a small-energy argument to absorb the nonlinear term as above.   The second difficulty is that the dissipation is not coercive over the energy space.  Indeed, we find that there is a half-derivative gap between the dissipation and energy for $\eta$.  This prevents us from deducing exponential decay of solutions, even if we manage to derive \eqref{strat_2}.

Our solution to this problem is to implement the two-tier energy method devised by Guo and Tice in the analysis of the one-phase incompressible problem \cite{GT_per,GT_inf}.  The idea is to  employ two tiers of energies and dissipation, $\se{l}$, $\sd{l}$, $\se{h}$, and $\sd{h}$, where $l$ stands for low regularity and $h$ stands for high regularity.  We then aim to prove estimates of the form
\begin{equation}
 \frac{d}{dt} \se{h} + \sd{h} \ls (\se{h})^\theta \sd{h} + t (\se{l})^{1/2} \sd{h} \text{ and } \frac{d}{dt} \se{l} + \sd{l} \ls (\se{h})^\theta \sd{l},
\end{equation}
where the $t$ term results from the transport estimate.  If we know that $\se{l}$ decays at a sufficiently fast polynomial rate and that $\se{h}$ is bounded and small, we can deduce from these that
\begin{equation}
 \sup_{t} \se{h} + \int_0^\infty \sd{h} \ls \se{h}(0) \text{ and } \frac{d}{dt} \se{l} + \hal \sd{l} \le 0.
\end{equation}
We still don't have that $\sd{l} \ge C \se{l}$, but we can use an interpolation argument to bound $\se{l} \ls (\se{h})^{1-\gamma} (\sd{l})^{\gamma}$ for some $\gamma \in (0,1)$.  Plugging this in leads to an algebraic decay estimate for $\se{l}$ with the rate determined by $\gamma$.  If the high and low regularity scales are appropriately chosen, this scheme of a priori estimates closes and leads to a method of producing global-in-time solutions that decay to equilibrium.  An interesting feature of this is that the existence of global solutions is predicated on the decay of the solutions.

The rest of the paper is organized as follows.  In the remainder of this section we record some definitions and bits of terminology. Section \ref{sec_prelim} contains some preliminary estimates and energy-dissipation identities that are useful throughout the paper.  Section \ref{sec_horiz} produces the estimates for temporal and horizontal derivatives.  Section \ref{sec_aux} produces energy-dissipation estimates for vertical derivatives of the density perturbation.  Section \ref{sec_combo} combines the horizontal and temporal estimates with the density perturbation estimates to produce estimates for the intermediate energy and dissipation.  Section \ref{sec_comparison} shows that the intermediate energy and dissipation control the full energy and dissipation up to error terms.  Section \ref{sec_apriori} records the finished a priori estimates.  Section \ref{section_proof} contains the proofs of Theorems \ref{th_gwp} and \ref{th_vanish_st}.

\subsection{Definitions and terminology} \label{def-ter}

We now mention some of the definitions, bits of notation, and conventions that we will use throughout the paper.

{ \bf Einstein summation and constants:}  We will employ the Einstein convention of summing over  repeated indices for vector and tensor operations.  Throughout the paper $C>0$ will denote a generic constant that can depend $N$, $\Omega$, or any of the parameters of the problem except for surface tension.  We refer to such constants as ``universal.''  They are allowed to change from one inequality to the next.   We will employ the notation $a \ls b$ to mean that $a \le C b$ for a universal constant $C>0$.

To track the appearance of constants depending on the surface tension coefficients, we employ the notation
\begin{equation}\label{ks_def}
 \ks =
\begin{cases}
C(\sigma) & \text{if }\rj >0  \\
B(\sigma) & \text{if } \rj <0,
\end{cases}
\end{equation}
where $C(\sigma) >0$ is a constant depending on $\sigma$ (and the other parameters of the problem) and $B(\sigma)>0$ denotes a constant (also depending on the other parameters) satisfying
\begin{equation}
 \liminf_{(\sigma_+,\sigma_-)\to 0} B(\sigma) \in (0,\infty)
\end{equation}
and
\begin{equation}
 \limsup_{(\sigma_+,\sigma_-)\to \infty} \frac{B(\sigma)} {\max\{\sigma_+,\sigma_- \}^j } < \infty \text{ for some integer } j >0.
\end{equation}
Notice in particular that linear combinations of powers of constants of the form $\ks$ remain of the form $\ks$ so that, for example, we may write ``equalities'' of the form $1 + \ks^5 + \sqrt{\ks} = \ks$.

{ \bf Norms:} We write $H^k(\Omega_\pm)$ with $k\ge 0$ and and $H^s(\Sigma_\pm)$ with $s \in \Rn{}$ for the usual Sobolev spaces.   We will typically write $H^0 = L^2$.  If we write $f \in H^k(\Omega)$, the understanding is that $f$ represents the pair $f_\pm$ defined on $\Omega_\pm$ respectively, and that $f_\pm \in H^k(\Omega_\pm)$.  We employ the same convention on $\Sigma_\pm$.  We will refer to the space $\H(\Omega)$ defined as follows:
\begin{equation}
 \H(\Omega) = \{ v \in H^1(\Omega) \; \vert \; \jump{v}=0 \text{ on } \Sigma_- \text{ and } v_- = 0 \text{ on } \Sigma_b\}.
\end{equation}

To avoid notational clutter, we will avoid writing $H^k(\Omega)$ or $H^k(\Sigma)$ in our norms and typically write only $\norm{\cdot}_{k}$, which we actually use to refer to  sums
\begin{equation}
 \ns{f}_k = \ns{f_+}_{H^k(\Omega_+)} + \ns{f_-}_{H^k(\Omega_-)}   \text{ or }  \ns{f}_k = \ns{f_+}_{H^k(\Sigma_+)} + \ns{f_-}_{H^k(\Sigma_-)}.
\end{equation} Since we will do this for functions defined on both $\Omega$ and $\Sigma$, this presents some ambiguity.  We avoid this by adopting two conventions.  First, we assume that functions have natural spaces on which they ``live.''  For example, the functions $u$, $\rho$, $q$, and $\bar{\eta}$ live on $\Omega$, while $\eta$ lives on $\Sigma$.  As we proceed in our analysis, we will introduce various auxiliary functions; the spaces they live on will always be clear from the context.  Second, whenever the norm of a function is computed on a space different from the one in which it lives, we will explicitly write the space.  This typically arises when computing norms of traces onto $\Sigma$ of functions that live on $\Omega$.

Occasionally we will need to refer to the product of a norm of $\eta$ and a constant that depends on $\pm$.  To denote this we will write
\begin{equation}
\gamma \ns{\eta}_{k} = \gamma_+ \ns{\eta_+}_{H^k(\Sigma_+)} + \gamma_- \ns{\eta_-}_{H^k(\Sigma_-)}.
\end{equation}

{ \bf Derivatives }

We write $\mathbb{N} = \{ 0,1,2,\dotsc\}$ for the collection of non-negative integers.  When using space-time differential multi-indices, we will write $\mathbb{N}^{1+m} = \{ \alpha = (\alpha_0,\alpha_1,\dotsc,\alpha_m) \}$ to emphasize that the $0-$index term is related to temporal derivatives.  For just spatial derivatives we write $\mathbb{N}^m$.  For $\alpha \in \mathbb{N}^{1+m}$ we write $\partial^\alpha = \dt^{\alpha_0} \p_1^{\alpha_1}\cdots \p_m^{\alpha_m}.$ We define the parabolic counting of such multi-indices by writing $\abs{\alpha} = 2 \alpha_0 + \alpha_1 + \cdots + \alpha_m.$  We will write $\nab_{\ast}f$ for the horizontal gradient of $f$, i.e. $\nab_{\ast}f = \p_1 f e_1 + \p_2 f e_2$, while $\nab f$ will denote the usual full gradient.

For a \textit{given norm} $\norm{\cdot}$ and an integer $k\ge 0$, we introduce the following notation for sums of spatial derivatives:
\begin{equation}
 \norm{\nab_{\ast}^k f}^2 := \sum_{\substack{\alpha \in \mathbb{N}^2 \\  \abs{ \alpha}\le k} } \norm{\pa^\al  f}^2 \text{ and }
\norm{\nab^k f}^2 := \sum_{\substack{\alpha \in \mathbb{N}^{3} \\   \abs{\alpha}\le k} } \norm{\pa^\al  f}^2.
\end{equation}
The convention we adopt in this notation is that $\nab_{\ast}$ refers to only ``horizontal'' spatial derivatives, while $\nab$ refers to full spatial derivatives.   For space-time derivatives we add bars to our notation:
\begin{equation}
 \norm{\bar{\nab}_{\ast}^k f}^2 := \sum_{\substack{\alpha \in \mathbb{N}^{1+2} \\   \abs{\alpha}\le k} } \norm{\pa^\al  f}^2 \text{ and }
\norm{\bar{\nab}^k f}^2 := \sum_{\substack{\alpha \in \mathbb{N}^{1+3} \\   \abs{\alpha}\le k} } \norm{\pa^\al  f}^2.
\end{equation}
We allow for composition of derivatives in this counting scheme in a natural way; for example, we write
\begin{equation}
 \norm{\nab_{\ast} \nab_{\ast}^{k} f}^2 = \norm{ \nab_{\ast}^k \nab_{\ast} f}^2 = \sum_{\substack{\alpha \in \mathbb{N}^{2} \\   \abs{\alpha}\le k} } \norm{\pa^\al  \nab_{\ast} f}^2  = \sum_{\substack{\alpha \in \mathbb{N}^{2} \\  1\le \abs{\alpha}\le k+1} } \norm{\pa^\al   f}^2.
\end{equation}

\section{Preliminaries}\label{sec_prelim}

In this section we record some preliminary results that are useful throughout the paper.  We rewrite the problem \eqref{ns_geometric} as a perturbation around the equilibrium, and then we present various energy equalities.  We conclude with the estimate of certain nonlinear terms.

\subsection{Perturbed formulation around the steady-state -- geometric form }

We will now rephrase the PDE \eqref{ns_geometric} in a perturbation formulation around the steady state solution $(\bar\rho,0,0)$. We recall our special density perturbation defined by \eqref{q_def}. In order to deal with the pressure term $P(\rho)=P(\bar\rho+\q+ \p_3\bar\rho\theta)$ we introduce the Taylor expansion, by \eqref{steady},
\begin{equation}\label{R1}
P (\bar\rho+\q+\p_3\bar\rho\theta)=P (\bar{\rho} )+P '(\bar{\rho} )(\q+\p_3\bar\rho\theta)+\mathcal{R}=P (\bar{\rho} )+P '(\bar{\rho} ) \q -g\bar\rho\theta+\mathcal{R},
\end{equation}
where the remainder term is defined by the integral form
\begin{equation}\label{R_def}
\mathcal{R} =\int_{\bar{\rho} }^{\bar{\rho} +\q+\p_3\bar\rho\theta}(\bar{\rho} +\q+\p_3\bar\rho\theta-z)  P ^{\prime\prime}(z)\,dz.
\end{equation}
Then one can see that the perturbation $\q$ is defined in its special way enjoying the following advantage that in $\Omega$,
\begin{equation}
\begin{split}\label{adv3.3}
&\a_{ij}\p_j P ( {\rho} )+g\rho \delta_{i3}=\a_{ij}\p_j P ( {\rho} )+g\rho \a_{ij}\p_j\Theta_3
\\&\quad=\a_{ij}\p_j (P (\bar{\rho} )+P '(\bar{\rho} ) \q -g\bar\rho\theta+\mathcal{R})+g(\bar\rho+\q+\p_3\bar\rho\theta) \a_{ij}\p_j(x_3+\theta)
\\&\quad=\a_{ij}\p_j ( P '(\bar{\rho} ) \q )-g\a_{ij}\p_j  (\bar\rho \theta )+\a_{ij}\p_j\mathcal{R}+g \bar\rho  \a_{ij}\p_j\theta +g(\q+\p_3\bar\rho\theta) \a_{i3} +g( \q+\p_3\bar\rho\theta) \a_{ij}\p_j\theta
\\&\quad=\a_{ij}\p_j ( P '(\bar{\rho} ) \q ) +\a_{ij}\p_j\mathcal{R}+g \q  \a_{i3}+g( \q+\p_3\bar\rho\theta) \a_{ij}\p_j\theta
\\&\quad=\bar\rho\a_{ij}\p_j ( h '(\bar{\rho} ) \q ) +\a_{ij}\p_j\mathcal{R}+g( \q+\p_3\bar\rho\theta) \a_{ij}\p_j\theta,
\end{split}%
\end{equation}
where we have used \eqref{steady} and \eqref{h'}. Recalling also \eqref{rho+-}, \eqref{theta} and \eqref{b function}, we have
\begin{equation}
 -g\bar\rho_+\theta=-\rho_1  g\eta_+\text{ on }\Sigma_+,\text{ and } \jump{-g\bar\rho\theta}  =-\rj g \eta_-\text{ on }\Sigma_-.
\end{equation}

By perturbing the density as above, we are led from \eqref{ns_geometric} to the following system:
\begin{equation}\label{geometric}
\begin{cases}
\partial_t \q +\diverge_\a ( \bar{\rho}   u)=F^1 & \text{in }
\Omega  \\
( \bar{\rho} +  \q+\p_3\bar\rho\theta)\partial_t    u    + \bar{\rho}\nabla_\a \left(h'(\bar{\rho})\q\right)   -\diva \S_{\a} u =F^2 & \text{in }
\Omega  \\
\partial_t \eta = u\cdot \n &
\text{on }\Sigma  \\
(  P'(\bar\rho)\q I- \S_{\a}(  u))\n  =  \rho_1  g \eta_+ \n-\sigma_+ \Delta_\ast\eta_+  \n +F_+^3
 & \text{on } \Sigma_+
 \\ \jump{P'(\bar\rho)\q I- \S_\a(u)}\n_-
= \rj g\eta_-\n +\sigma_- \Delta_\ast\eta_-\n-F_-^3&\hbox{on }\Sigma_-
 \\\jump{u}=0 &\hbox{on }\Sigma_-
\\ u_-=0 &\hbox{on }\Sigma_b,
\end{cases}%
\end{equation}
where
\begin{equation}
 F^{1}=\p_3^2\bar\rho K \theta \p_t\theta+K\p_t\theta \pa_3  \q -\diverge_\a((\q+\p_3\bar\rho\theta) u ),
\end{equation}
\begin{equation}
\begin{split}
 F^2 =  -( \bar{\rho} +  \q+\p_3\bar\rho\theta)
(-K\p_t\theta \pa_3  u
 +   u \cdot \nab_\a  u )
       - \nabla_\a\mathcal{R}-g( \q+\p_3\bar\rho\theta ) \nabla_\a \theta ,
      \end{split}
\end{equation}
 \begin{equation}
F^3_{+}= - \mathcal{R} \n-\sigma_+\diverge_\ast(((1+|\nab_\ast\eta_+|^2)^{-1/2}-1)\nab_\ast\eta_+) \n,
\end{equation}
and
 \begin{equation}
 -F^3_-= - \jump{ \mathcal{R} }\n
 +\sigma_-\diverge_\ast(((1+|\nab_\ast\eta_-|^2)^{-1/2}-1)\nab_\ast\eta_-) \n.
\end{equation}

We will employ the form of the equations \eqref{geometric} primarily for estimating the temporal derivatives of the solutions.  Applying the temporal differential operator $\dt^j$ for $j=0,\dots,2N$ to \eqref{geometric}, we find that
\begin{equation}\label{linear_geometric}
\begin{cases}
\partial_t (\dt^j \q )+\diverge_\a ( \bar{\rho}  \dt^j u)=F^{1,j} & \text{in }
\Omega  \\
( \bar{\rho} +  \q+\p_3\bar\rho\theta)\partial_t   (\dt^j u )  + \bar{\rho}\nabla_\a  (h'(\bar{\rho}) \dt^j\q )   -\diva\S_{\a}(\dt^j u) = F^{2,j}  & \text{in } \Omega \\
  \dt (\dt^j \eta)   = \dt^j u\cdot \n+F^{4,j} & \text{on } \Sigma\\
(  P'(\bar{\rho}) \dt^j \q I- \S_{\a}(\dt^j u))\n  =  \rho_1  g \dt^j \eta_+ \n-\sigma_+ \Delta_\ast(\dt^j \eta_+)  \n +F_+^{3,j}
 & \text{on } \Sigma_+
 \\ \jump{P'(\bar\rho)\dt^j \q I- \S_\a(\dt^j u)}\n
= \rj g \dt^j \eta_-\n +\sigma_- \Delta_\ast(\dt^j \eta_-)\n-F_-^{3,j}&\hbox{on }\Sigma_-
 \\\jump{\dt^j u}=0 &\hbox{on }\Sigma_-\\
 \dt^j u_- =0 & \text{on } \Sigma_b,
\end{cases}
\end{equation}
where
\begin{equation}\label{F1j_def}
 F^{1,j}  =  \dt^j F^1-\sum_{0 < \ell \le j }  C_j^\ell \dt^{ \ell}\mathcal{A}_{lk}\p_k( \bar{\rho}  \dt^{j-\ell} u_l),
 \end{equation}
\begin{equation}\label{F2j_def}
\begin{split}
 F_i^{2,j}  &=\dt^j F_i^2+\sum_{0 < \ell \le j }  C_j^\ell\left\{\mu\mathcal{A}_{lk} \p_k (\dt^{ \ell}\mathcal{A}_{lm}\dt^{j - \ell}\p_m u_i)
+\mu\dt^{ \ell}\mathcal{A}_{lk} \dt^{j - \ell}\p_k (\mathcal{A}_{lm}\p_m u_i)
  \right.
\\
 &\quad+   (\mu/3+\mu')\mathcal{A}_{ik} \p_k (\dt^{ \ell}\mathcal{A}_{lm}\dt^{j - \ell}\p_m u_l)
+(\mu/3+\mu')\dt^{ \ell}\mathcal{A}_{ik} \dt^{j - \ell}\p_k (\mathcal{A}_{lm}\p_m u_l)
\\&\quad\left.-  \bar{\rho} \dt^{\ell} \mathcal{A}_{ik} \p_k
(h'(\bar{\rho})\dt^{j - \ell} \q)-\dt^\ell(\q+\p_3\bar\rho\theta)\partial_t   (\dt^{j-\ell} u ) \right\} , \ i=1,2,3,
 \end{split}
\end{equation}
\begin{equation}\label{F3j+_def}
\begin{split}
F_{i,+}^{3,j} &= \dt^jF_{i,+}^{3}+ \sum_{0 < \ell \le j} C_j^\ell \left\{\mu_+\dt^{ \ell} ( \n_{l} \mathcal{A}_{ik} ) \dt^{j - \ell} \p_k u_{l} + \mu_+ \dt^{ \ell} ( \n_{l} \mathcal{A}_{lk} ) \dt^{j - \ell} \p_k u_{i}\right.
\\&\quad \left.+(\mu_+'-2\mu_+/3)\dt^{ \ell} ( \n_{i}\mathcal{A}_{lk} ) \dt^{j - \ell} \p_k u_{l}+ \dt^{\ell} \n_{i} \dt^{j - \ell}(\rho_1 g \eta_+- P'(\bar{\rho})   \q-\sigma_+   \Delta_\ast\eta_+ )\right\}, \  i=1,2,3,
\end{split}
\end{equation}
\begin{equation}\label{F3j-_def}
\begin{split}
-F_{i,-}^{3,j} &= -\dt^jF_{i,-}^{3}+ \sum_{0 < \ell \le j} C_j^\ell \left\{\dt^{ \ell} ( \n_{l} \mathcal{A}_{ik} ) \dt^{j - \ell} \jump{\mu \p_k u_{l}} + \dt^{ \ell} ( \n_{l} \mathcal{A}_{lk} ) \dt^{j - \ell} \jump{\mu\p_k u_{i}}\right.
\\&\quad\left.+\dt^{ \ell} ( \n_{i}\mathcal{A}_{lk} ) \dt^{j - \ell} \jump{(\mu'-2\mu/3)\p_k u_{l}}+ \dt^{\ell} \n_{i} \dt^{j - \ell}( \rj g\eta_-- \jump{P'(\bar\rho)   \q}+\sigma_-  \Delta_\ast\eta_-) \right\},
\end{split}
\end{equation}
for  $i=1,2,3,$ and
\begin{equation}\label{F4j_def}
 F^{4,j} =  \sum_{0 < \ell \le j}  C_j^\ell \dt^{ \ell} \n\cdot \dt^{j - \ell}  u.
\end{equation}

Note that the equations \eqref{linear_geometric} for $(\partial_t^j q, \partial_t^ju,  \partial_t^j \eta)$ for $j=0,\dots,2N$ have the same structure in their left-hand sides. We present their energy identities more general way as follows:

\begin{Proposition}\label{geo_en_evolve}
Suppose that $(q,u,\eta)$ solve \eqref{ns_geometric} and that $\a, J,$ and $\n$ are determined through $\eta$ as in \eqref{A_def}, \eqref{ABJ_def}, and \eqref{n_def}.  Further suppose that $(Q,v,\zeta)$ solve
\begin{equation}\label{gee_0}
\begin{cases}
\partial_t Q +\diverge_\a ( \bar{\rho}   v)=\mathfrak{F}^1 & \text{in }
\Omega  \\
( \bar{\rho} +  \q+\p_3\bar\rho\theta)\partial_t    v    + \bar{\rho}\nabla_\a \left(h'(\bar{\rho})Q\right)   -\diva \S_{\a}(v) =\mathfrak{F}^2 & \text{in }
\Omega  \\
\partial_t \zeta = v\cdot \n + \mathfrak{F}^4 &
\text{on }\Sigma  \\
(  P'(\bar\rho)Q I- \S_{\a}(  v))\n  =  \rho_1  g \zeta_+ \n-\sigma_+ \Delta_\ast\zeta_+  \n +\mathfrak{F}_+^3
 & \text{on } \Sigma_+
 \\\jump{P'(\bar\rho)Q I- \S_\a(v)}\n
= \rj g\zeta_-\n +\sigma_- \Delta_\ast\zeta_-\n-\mathfrak{F}_-^3&\hbox{on }\Sigma_-
 \\\jump{v}=0 &\hbox{on }\Sigma_-
\\ v_-=0 &\hbox{on }\Sigma_b.
\end{cases}
\end{equation}
Then
\begin{multline}\label{gee_01}
 \frac{1}{2}\frac{d}{dt}\left(\int_\Omega   ( \bar{\rho} +  \q+\p_3\bar\rho\theta)J\abs{ v }^2+h'(\bar{\rho})J \abs{Q}^2+\int_{\Sigma_+} \rho_1  g\abs{ \zeta_+}^2+\sigma_+\abs{\nab_\ast  \zeta_+}^2\right. \\
\left.+\int_{\Sigma_-}- \rj g\abs{ \zeta_-}^2+ \sigma_-\abs{\nab_\ast  \zeta_-}^2\right)+ \int_{\Omega}\frac{\mu}{2}
J\abs{\sgz_\mathcal{A} v}^2+\mu'J\abs{\diverge_\mathcal{A} v}^2
\\ =\frac{1}{2} \int_\Omega \dt(J( \bar{\rho} +  \q+\p_3\bar\rho\theta)) \abs{ v }^2+h'(\bar{\rho})\dt J \abs{Q}^2+\int_\Omega J(h'(\bar{\rho}) Q   \mathfrak{F}^{1}+ v\cdot \mathfrak{F}^{2})
\\ +\int_{\Sigma}- v\cdot \mathfrak{F}^{3}+\int_{\Sigma_+} \rho_1  g \zeta_+\mathfrak{F}_+^{4} - \int_{\Sigma_-}\rj g \zeta_-\mathfrak{F}_-^{4} -\int_\Sigma\sigma\Delta_\ast  \zeta \mathfrak{F}^{4}.
\end{multline}

\end{Proposition}
\begin{proof}
Taking the dot product of the second equation of $\eqref{gee_0}$ with $J   v$ and then integrating by parts over the domain $\Omega$, using the Dirichlet boundary conditions of $v$, we obtain
\begin{align} \label{gee_1}
& \frac{1}{2}\frac{d}{dt}\int_\Omega   ( \bar{\rho} +  \q +\p_3\bar\rho\theta) J\abs{v}^2
+ \int_{\Omega}\frac{\mu}{2}J\abs{\sgz_\mathcal{A} v}^2+\mu'J\abs{\diverge_\mathcal{A} v}^2\nonumber
\\&\quad = \frac{1}{2} \int_\Omega   \dt(J( \bar{\rho} +  \q+\p_3\bar\rho\theta))  \abs{v}^2   +\int_\Omega J h'(\bar{\rho})
Q \diva(\bar{\rho} v) +\int_\Omega J  v\cdot \mathfrak{F}^{2}
\\&\qquad - \int_{\Sigma_+}(  P'(\bar{\rho})  Q I- \S_{\a}( v))\n\cdot  v+\int_{\Sigma_-}\jump{P'(\bar\rho) Q I- \S_\a(v)}\n \cdot  v .\nonumber
\end{align}
Using the first equation of $\eqref{gee_0}$, we have
\begin{equation}\label{gee_2}
\begin{split}
&\int_\Omega Jh'(\bar{\rho}) Q \diva(\bar{\rho} v)
= \int_\Omega Jh'(\bar{\rho}) Q (-\partial_t ( Q )+\mathfrak{F}^{1})
\\&\quad= -\frac{1}{2}\frac{d}{dt}\int_\Omega h'(\bar{\rho})J \abs{Q}^2 +\frac{1}{2} \int_\Omega h'(\bar{\rho})\dt J \abs{Q}^2   +\int_\Omega Jh'(\bar{\rho}) Q  \mathfrak{F}^{1} .
\end{split}
 \end{equation}
 Using the fourth equation of $\eqref{gee_0}$, we have
 \begin{equation}
 - \int_{\Sigma_+}(  P'(\bar{\rho})  Q I- \S_{\a}( v))\n\cdot  v
= -  \int_{\Sigma_+}(\rho_1  g  \zeta_+ \n-\sigma_+ \Delta_\ast \zeta_+  \n +\mathfrak{F}_+^{3})\cdot  v.
 \end{equation}
 The third equation of $\eqref{gee_0}$ further implies
 \begin{equation}
\begin{split}
- &\int_{\Sigma_+} (\rho_1  g \zeta_+-\sigma_+\Delta_\ast  \zeta_+) \n\cdot  v
 \\&= - \int_{\Sigma_+}(\rho_1  g  \zeta_+-\sigma_+\Delta_\ast  \zeta_+)(\dt  \zeta_+ -\mathfrak{F}_+^{4})
\\&=-\hal \frac{d}{dt}\int_{\Sigma_+} \rho_1  g\abs{ \zeta_+}^2+\sigma_+\abs{\nab_\ast  \zeta_+}^2
+\int_{\Sigma_+}(\rho_1  g \zeta_+-\sigma_+\Delta_\ast  \zeta_+)\mathfrak{F}_+^{4}.
\end{split}
\end{equation}
Hence
 \begin{equation}\label{gee_3}
\begin{split}
 - &\int_{\Sigma_+}(  P'(\bar{\rho})  Q I- \S_{\a}( v))\n\cdot  v
\\& =-\hal \frac{d}{dt}\int_{\Sigma_+} \rho_1  g\abs{ \zeta_+}^2+\sigma_+\abs{\nab_\ast  \zeta_+}^2
+\int_{\Sigma_+}- v\cdot \mathfrak{F}_+^{3}+(\rho_1  g \zeta_+-\sigma_+\Delta_\ast  \zeta_+)\mathfrak{F}_+^{4}.
\end{split}
\end{equation}
Similarly, we use the fifth and third equations of $\eqref{gee_0}$ to have
 \begin{equation}\label{gee_4}
\begin{split}
\int_{\Sigma_-} & \jump{P'(\rho) Q I- \S_\a( v)}\n \cdot  v
\\&=-\hal \frac{d}{dt}\int_{\Sigma_+}- \rj g\abs{ \zeta_+}^2+\sigma_-\abs{\nab_\ast  \zeta_-}^2
   +\int_{\Sigma_-}- v\cdot \mathfrak{F}_-^{3}+(-\rj g \zeta_--\sigma_-\Delta_\ast  \zeta_-)\mathfrak{F}_-^{4}.
\end{split}
\end{equation}
Consequently, plugging \eqref{gee_2}, \eqref{gee_3} and \eqref{gee_4} into \eqref{gee_1}, we obtain \eqref{gee_01}.
\end{proof}

\subsection{Perturbed formulation around the steady-state -- linear form }

It turns out to be convenient to write the system \eqref{ns_geometric} in a linear form to derive the subsequent estimates.  The reason for this is that the operators become constant-coefficient, which is more convenient for  elliptic regularity.  We rewrite the PDE \eqref{geometric}  for $(\q,u,\eta)$ as
\begin{equation}\label{ns_perturb}
\begin{cases}
\partial_t \q +\diverge ( \bar{\rho}   u)=G^1 & \text{in }
\Omega  \\
 \bar{\rho} \partial_t    u   + \bar{\rho}\nabla \left(h'(\bar{\rho})\q\right)   -\diverge \S(u) =G^2 & \text{in }
\Omega  \\
\partial_t \eta = u_3+G^4 &
\text{on }\Sigma  \\
(  P'(\bar\rho)\q I- \S(  u))e_3  = (\rho_1  g \eta_+ -\sigma_+ \Delta_\ast \eta_+ ) e_3 +G_+^3
 & \text{on } \Sigma_+
 \\ \jump{P'(\bar\rho)\q I- \S(u)}e_3
=(\rj g\eta_- +\sigma_- \Delta_\ast \eta_-)e_3-G_-^3&\hbox{on }\Sigma_-
\\\jump{u}=0 &\hbox{on }\Sigma_-
\\ u_-=0 &\hbox{on }\Sigma_b,
\end{cases}%
\end{equation}
where we have written the function $G^1=G^{1,1}+G^{1,2}$ for
\begin{equation}\label{G1_def}
G^{1,1}= K  \p_t\theta\pa_3  \q-  u_l \mathcal{A}_{lk}\pa_k  \q,
\end{equation}
\begin{equation}
G^{1,2}= \p_3^2\bar\rho K \theta \p_t\theta- \q \mathcal{A}_{lk}\pa_k u_l-\mathcal{A}_{lk}\pa_k(\p_3\bar\rho\theta   u_l)-(\mathcal{A}_{lk}-\delta_{lk} )\pa_k ( \bar{\rho}   u_l),
\end{equation}
the vector $G^2$ for
\begin{equation}
\begin{split}
G^2_i=&-( \q+\p_3\bar\rho\theta )\partial_t    u_i+(\bar\rho+\q+\p_3\bar\rho\theta )
(K  \p_t\theta \pa_3  u_i
 -   u_l\mathcal{A}_{lk}\pa_k u_i)
  \\& +  \mu \mathcal{A}_{lk} \pa_k\mathcal{A}_{lm}\pa_mu_i  +
  \mu (\mathcal{A}_{lk}\mathcal{A}_{lm}- \delta_{lk}\delta_{lm}) \pa_{km} u_i
   \\&
  +(\mu/3+\mu') \mathcal{A}_{ik} \pa_k\mathcal{A}_{lm}\pa_mu_l
   +(\mu/3+\mu')(\mathcal{A}_{ik}\mathcal{A}_{lm}- \delta_{ik}\delta_{lm}) \pa_{km} u_l
   \\& - \bar{\rho}(\mathcal{A}_{il}-\delta_{il})\pa_l ( h'(\bar{\rho})\q)
       - \mathcal{A}_{il}\pa_l  \mathcal{R}-g( \q+\p_3\bar\rho\theta ) \a_{il}\p_l \theta,\  i=1,2,3,
\end{split}
\end{equation}
the vector $G^3_+ = G^{3,1}_+ + \sigma_+ G^{3,2}_+$ for
\begin{equation}
\begin{split}
G^{3,1}_{i,+}&=
  \mu_+  (\mathcal{A}_{il} \partial_l u_k+\mathcal{A}_{kl} \partial_l u_i)(\n_{k }-\delta_{k3})
 + \mu_+  (\mathcal{A}_{il}-\delta_{il}) \partial_l u_3+ \mu (\mathcal{A}_{3l}-\delta_{3l}) \partial_l u_i
 \\& +(\mu_+'-2\mu_+/3)  \mathcal{A}_{lk} \partial_k u_l (\n_{i }-\delta_{i3})
+ (\mu_+'-2\mu_+/3)   (\mathcal{A}_{lk}-\delta_{lk}) \partial_k u_l\delta_{i3}
+  \rho_1 g \eta_+ (\n_i-\delta_{i3}) \\
&-\mathcal{R}\n_{i }
 + P'(\bar{\rho})\q(\delta_{i3} - \n_i)
  \end{split}
\end{equation}
and
\begin{equation}
 G^{3,2}_{i,+} = -\Delta_\ast\eta_+ (\n_i-\delta_{i3}) - \diverge_\ast(((1+|\nab_\ast\eta_+|^2)^{-1/2}-1)\nab_\ast\eta_+)\n_i
\end{equation}
for $i=1,2,3,$ and the vector $G^3_- = G^{3,1}_- + \sigma_- G^{3,2}_-$ for
\begin{equation}
\begin{split}
-G^3_{i,-}&=
  (\mathcal{A}_{il} \jump{\mu  \partial_l u_k}+\mathcal{A}_{kl} \jump{\mu \partial_l u_i})(\n_{k}-\delta_{k3})
+ (\mathcal{A}_{il}-\delta_{il}) \jump{\mu \partial_l u_3}-(\mathcal{A}_{3l}-\delta_{3l}) \jump{\mu \partial_l u_i}
 \\&   + \mathcal{A}_{lk} \jump{(\mu'-2\mu/3) \partial_k u_l} (\n_{i }-\delta_{i3})
+ (\mathcal{A}_{lk}-\delta_{lk}) \jump{(\mu'-2\mu/3)\partial_k u_l}\delta_{i3}
\\&  + \rj g\eta_-(\n_i-\delta_{i3}) -\jump{\mathcal{R}}\n_{i }
  + \jump{P'(\bar{\rho})\q} (\delta_{i3} - \n_i),
  \end{split}
\end{equation}
and
\begin{equation}
 G^{3,2}_- =  \Delta_\ast\eta_-(\n_i-\delta_{i3})  + \diverge_\ast(((1+|\nab_\ast\eta_-|^2)^{-1/2}-1)\nab_\ast\eta_-)\n_i
\end{equation}
for  $i=1,2,3,$  and the function $G^4$ for
\begin{eqnarray}\label{G4_def}
 G^4= -u_1\pa_1\eta-u_2\pa_2\eta.
\end{eqnarray}
In all of these, $\mathcal{R}$ is as defined in \eqref{R_def}.

We now present the energy identities related to \eqref{ns_perturb}.

\begin{Proposition}\label{lin_en_evolve}
Suppose that $(Q,v,\zeta)$ solve
\begin{equation}\label{lee_0}
\begin{cases}
\partial_t Q +\diverge ( \bar{\rho}   v)=\mathfrak{G}^1 & \text{in }
\Omega  \\
 \bar{\rho} \partial_t    v   + \bar{\rho}\nabla \left(h'(\bar{\rho})Q\right)   -\diverge \S(v) =\mathfrak{G}^2 & \text{in }
\Omega  \\
\partial_t \zeta = v_3+\mathfrak{G}^4 &
\text{on }\Sigma  \\
(  P'(\bar\rho)Q I- \S(  v))e_3  = (\rho_1  g \zeta_+ -\sigma_+ \Delta_\ast \zeta_+ ) e_3 +\mathfrak{G}_+^3
 & \text{on } \Sigma_+
 \\ \jump{P'(\bar\rho)Q I- \S(v)}e_3
=(\rj g\zeta_- +\sigma_- \Delta_\ast \zeta_-)e_3-\mathfrak{G}_-^3&\hbox{on }\Sigma_-
 \\\jump{v}=0 &\hbox{on }\Sigma_-
\\ v_-=0 &\hbox{on }\Sigma_b.
\end{cases}
\end{equation}
Then
\begin{multline}\label{lee_01}
\frac{1}{2}\frac{d}{dt}\left(\int_\Omega h'(\bar{\rho})\abs{  Q}^2+\bar{\rho}\abs{ v}^2 +\int_{\Sigma_+} \rho_1  g\abs{ \zeta_+}^2+\sigma_+\abs{\nab_\ast  \zeta_+}^2\right.
\\
 \left. +\int_{\Sigma_-}- \rj g\abs{ \zeta_-}^2+ \sigma_-\abs{\nab_\ast \zeta_-}^2\right)
 +\int_\Omega \frac{\mu}{2} \abs{ \sgz v}^2+\mu' \abs{{\rm div } v}^2
\\
= \int_\Omega h'(\bar{\rho}) Q \mathfrak{G}^1 + v\cdot\pa^\al \mathfrak{G}^2+\int_\Sigma - v\cdot  \mathfrak{G}^3
\\
+ \int_{\Sigma_+} \rho_1  g  \zeta_+  \mathfrak{G}^4_++\int_{\Sigma_-}-\rj g  \zeta_-  \mathfrak{G}^4_- -\int_\Sigma\sigma\Delta_\ast   \zeta   \mathfrak{G}^4.
\end{multline}
\end{Proposition}
\begin{proof}
The proof proceeds along the same lines as that of Proposition \ref{geo_en_evolve}: we multiply the second equation in \eqref{lee_0} by $v$, integrate over $\Omega_+$ and $\Omega_-$, integrate by parts and the sum resulting equations.  After employing the other equations in \eqref{lee_0} as in Proposition \ref{geo_en_evolve} we arrive at \eqref{lee_01}.
\end{proof}

\subsection{Estimates of the nonlinearities}

We assume throughout this subsection that the solutions obey the estimate $\gs(T) \le \delta$, where $\delta \in (0,1)$ is given in  Lemma \ref{eta_small}.

We first present the estimates of the nonlinear terms $G^i$ (defined by \eqref{G1_def}--\eqref{G4_def}) at the $2N$ level.
\begin{lemma}\label{p_G2N_estimates}
 It holds that
\begin{multline}\label{p_G_e_0}
 \ns{ \bar{\nab}^{4N-2} G^1}_{1}  +  \ns{ \bar{\nab}^{4N-2}  G^2}_{0} +
 \ns{ \bar{\nab}_{\ast}^{4N-2}  G^3}_{1/2}
+ \ns{\bar{\nab}_{\ast }^{ 4N-1} G^4}_{1/2} \\
 \ls \ks \se{2N}^0\se{2N}^\sigma + \mathcal{E}_{N+2}^0 \f,
\end{multline}
and
\begin{multline}\label{p_G_e_00}
\ns{ \bar{\nab}^{4N-1}  G^{1,1}}_{0}+\ns{ \bar{\nab}^{4N-2}\dt G^{1,1}}_{0}  + \ns{ \bar{\nab}^{4N}  G^{1,2}}_{0}+  \ns{ \bar{\nab}^{4N-1}  G^2}_{0}
\\
+ \ns{ \bar{\nab}_{\ast}^{4N-1} G^3}_{1/2} + \ns{\bar{\nab}_{\ast }^{  4N-1} G^4}_{1/2}
   + \ns{\bar{\nab}_{\ast }^{ 4N-2} \dt G^4}_{1/2}
\\
+\sigma^2\ns{\bar{\nab}_{\ast }^{ 4N} G^4}_{1/2}
\ls \ks \se{2N}^0 \sd{2N}^\sigma + \mathcal{E}_{N+2}^0 \f.
\end{multline}
\end{lemma}
\begin{proof}
We first prove the estimates in \eqref{p_G_e_00}. Note that all terms in the definitions of $G^i$ are at least quadratic. We apply these space-time differential operators to $G^i$ and then expand using the Leibniz rule; each term in the resulting sum is also at least quadratic. We then estimate one term in $H^k$ ($k = 0$ or $1/2$ depending on
$G^i$) and the other terms in $H^m$ for $m$ depending on $k$, using  trace
theory, and Lemmas \ref{Poi}--\ref{sobolev} along with the definitions of $\se{2N}^\sigma$ and $\sd{2N}^\sigma$ (\eqref{p_energy_def} and \eqref{p_dissipation_def}, respectively), and $\ks$.  With three exceptions, we can estimate the desired norms of all the resulting terms by  $\se{2N}^0 \sd{2N}^0$.

The first exceptional terms are ones involving either $\nab^{4N+1}\bar{\eta}$ in $\Omega$ or $\nab_\ast^{4N}\eta$ on $\Sigma$ when estimating  $\ns{ {\nab}^{4N} G^{1,2}}_{0}, \ns{  {\nab} ^{4N-1}  G^2}_{0} ,
 \ns{  {\nab}_{\ast }^{ 4N-1}  G^{3,1}}_{1/2}$ and $\ns{ {\nab}_{\ast }^{ 4N-1} G^4}_{1/2}$. But we can argue as Theorem 3.2 of \cite{GT_per} to estimate $\ns{ \bar{\eta}}_{4N+1}\ls \ns{  {\eta}}_{4N+1/2}\ls\f$ to bound these terms by $\mathcal{E}_{N+2}^0\f$.

The second exceptional terms result from the $\sigma G^{3,2}$ term  when estimating  $\ns{  {\nab}_{\ast }^{ 4N-1}  G^3}_{1/2}$ and $\ns{  {\nab}_{\ast }^{ 4N-3}  \dt G^3}_{1/2}$. The highest derivative appearing in  $\sigma G^{3,2}$ is $\sigma\nab_\ast^2\eta$. We make use of this $\sigma$ factor to estimate
\begin{multline}
\ns{\sigma \nab_{\ast}^{4N-1} \nab_\ast^2\eta}_{1/2}\le \sigma^2 \ns{  \eta}_{4N+3/2} = \frac{\sigma^2}{\min\{1,\sigma^2\}} \min\{1,\sigma^2\} \ns{\eta}_{4N+3/2} \\
\le \frac{\sigma^2}{\min\{1,\sigma^2\}} \sd{2N}^\sigma \ls \ks \sd{2N}^\sigma,
\end{multline}
and, similarly,
\begin{equation}
 \ns{\sigma \nab_{\ast}^{4N-3} \nab_\ast^2 \dt \eta}_{1/2} \ls \sigma^2 \ns{\dt \eta}_{4N-1/2} \ls \ks \sd{2N}^\sigma.
\end{equation}
Using these, we may estimate
\begin{equation}
 \ns{  \bar{\nab}_{\ast }^{ 4N-1}  \sigma G^{3,2}}_{1/2}  \ls \ks \se{2N}^0\sd{2N}^\sigma.
\end{equation}

The last exceptional term is the term $\sigma^2\ns{\nab_\ast^{4N+1}\eta}_{1/2}$ that appear when we estimate the term $\sigma^2\ns{\nab_{\ast }^{ 4N} G^4}_{1/2}$. Again, the factor $\sigma^2$ leads to the estimate $\sigma^2\ns{\nab_\ast^{4N+1}\eta}_{1/2}\ls \ks \sd{2N}^\sigma$.

Hence, in light of the above analysis, we may deduce \eqref{p_G_e_00}. The proof of \eqref{p_G_e_0} proceeds similarly.   We remark that the term $\mathcal{E}_{N+2}^0\f$  in \eqref{p_G_e_0} does not appear in Theorem 3.2 of \cite{GT_per}; it appears here  because we want to control $\ns{{\nab}_{\ast}^{4N-1} G^4}_{1/2}$.
\end{proof}

We then present the estimates of $G^i$ at the $N+2$ level.
\begin{lemma}\label{p_GN+2_estimates}
 It holds that
\begin{multline}\label{p_G_e_h_0}
\ns{ \bar{\nab}^{2(N+2)-2} G^1}_{1} +  \ns{ \bar{\nab}^{2(N+2)-2}  G^2}_{0} +
 \ns{ \bar{\nab}_{\ast }^{  2(N+2)-2} G^3}_{1/2}
+ \ns{\bar{\nab}_{\ast }^{2(N+2)-1} G^4}_{1/2}
\\
\ls \ks \se{2N}^0\se{N+2}^0,
\end{multline}
and
\begin{multline}\label{p_G_e_h_00}
\ns{ \bar{\nab}^{2(N+2)} G^1}_{0} +  \ns{ \bar{\nab}^{2(N+2)-1}  G^2}_{0}  +
 \ns{ \bar{\nab}_{\ast }^{2(N+2)-1} G^3}_{1/2}
+ \ns{\bar{\nab}_{\ast }^{ 2(N+2) } G^4}_{1/2}
\\
\ls \ks \se{2N}^0 \sd{N+2}^0.
\end{multline}
\end{lemma}
\begin{proof}
The estimates \eqref{p_G_e_h_0}--\eqref{p_G_e_h_00} follow from arguments similar to those used in the proof of Proposition \ref{p_G2N_estimates}.  In this case they are easier because when we estimate the terms appearing from the Leibniz rule expansions we do not have exceptional terms as in the proof of Proposition \ref{p_G2N_estimates}.  Indeed, we may write each term in the form $X Y$, where $X$ involves fewer temporal derivatives than $Y$; then we simply bound the various norms of $Y$ by $\ks \se{2N}^0$ and bound the various norms of $X$ by $\se{N+2}^0$ for \eqref{p_G_e_h_0} or $\sd{N+2}^0$ for \eqref{p_G_e_h_00}.  Note here that we use $\se{2N}^0$ and absorb all appearances of $\sigma$ into $\ks$ rather than employ $\se{2N}^\sigma$.
\end{proof}

Next we present some variants of these estimate involving integrals.  First we consider products with derivatives of $G^4$.

\begin{lemma}\label{lemma8}
Let $\al\in \mathbb{N}^2$ so that $|\al|=4N$. Then
\begin{equation}\label{eta es}
\abs{ \int_{\Sigma}   \pa^\al \eta \pa^\al G^4 } \lesssim \sqrt{\se{2N}^0}  \sd{2N}^0 +\sqrt{  \sd{2N}^0\mathcal{E}_{N+2}^0 \f}
\end{equation}
and
\begin{equation}\label{eta es2}
\abs{  \int_{\Sigma}   \sigma \Delta_\ast\pa^\al\eta \pa^\al G^4  } \lesssim  \ks \sqrt{\se{2N}^0\sd{2N}^0\sd{2N}^\sigma}  + \ks \sqrt{  \sd{2N}^\sigma\mathcal{E}_{N+2}^0 \f} .
\end{equation}
 \end{lemma}
\begin{proof}
We first use the Leibniz rule to expand that
\begin{equation}
 -\pa^\al G^4 = \pa^\al (\nab_\ast \eta \cdot u) =   \nab_\ast \pa^\al \eta \cdot   u+\sum_{ 0<\beta \le \alpha  } C_{\alpha}^{\beta} \nab_\ast\p^{\alpha-\beta} \eta \cdot\p^\beta u
\end{equation}
For the second part, we estimate for $|\beta|\ge 1$, similarly as in Lemma \ref{p_G2N_estimates},
\begin{equation}
 \ns{\nab_\ast \p^{\alpha-\beta} \eta \cdot\p^\beta u}_{1/2}
 \ls    \se{2N}^0  \sd{2N}^0 + \mathcal{E}_{N+2}^0 \f.
\end{equation}
Hence, we have
\begin{equation}
 \int_\Sigma   \pa^\al \eta  \nab_\ast\p^{\alpha-\beta} \eta \cdot\p^\beta u
 \le \norm{\pa^\al \eta }_{-1/2}   \norm{\nab_\ast\p^{\alpha-\beta} \eta \cdot\p^\beta u}_{1/2}
 \ls \sqrt{\sd{2N}^0}\sqrt{ \se{2N}^0  \sd{2N}^0 + \mathcal{E}_{N+2}^0 \f}
\end{equation}
and
\begin{equation}
\begin{split}
 \int_\Sigma  \sigma\Delta_\ast\pa^\al \eta    \nab_\ast\p^{\alpha-\beta} \eta \cdot\p^\beta u
 &\le \sigma\norm{\Delta_\ast\pa^\al \eta}_{-1/2}   \norm{\nab_\ast\p^{\alpha-\beta} \eta \cdot\p^\beta u}_{1/2}
\\ & = \frac{\sigma}{\min\{1,\sigma\}} \sqrt{\min\{1,\sigma^2\} \ns{\eta}_{4N+3/2}} \norm{\nab_\ast\p^{\alpha-\beta} \eta \cdot\p^\beta u}_{1/2}
 \\&\ls \ks \sqrt{\sd{2N}^\sigma}\sqrt{ \se{2N}^0  \sd{2N}^0 + \mathcal{E}_{N+2}^0 \f}.
 \end{split}
\end{equation}

For the first term, we use  integration by parts to see that, by Lemma \ref{sobolev},
\begin{equation}
\begin{split}
\int_\Sigma  \pa^\al \eta   \nab_\ast \pa^\al \eta \cdot u
  & = \hal\int_\Sigma     \nab_\ast \abs{\pa^\al \eta}^2 \cdot u
= -\hal\int_\Sigma  \abs{\pa^\al  \eta}^2 \diverge_\ast u
 \\&  \ls   \norm{\pa^\al\eta}_{ -1/2} \norm{\pa^\al \eta}_{ 1/2}\norm{\diverge_\ast u}_{H^2(\Sigma)}\le \sqrt{\sd{2N}^0\f \mathcal{E}_{N+2}^0}
\end{split}
\end{equation}
and
\begin{equation}
\begin{split}
 &\int_\Sigma  \sigma\Delta_\ast\pa^\al \eta   \nab_\ast \pa^\al \eta \cdot u
 = -\sum_{i=1}^2 \int_\Sigma  \sigma\p_i\pa^\al \eta   \p_i (\nab_\ast \pa^\al \eta \cdot u )
 \\&= -\hal \sigma\int_\Sigma \nab_\ast \abs{ \nab_\ast\pa^\al\eta}^2 \cdot u- \sum_{i=1}^2 \int_\Sigma \sigma \p_i\pa^\al\eta  (\nab_\ast \pa^\al \eta \cdot  \p_i u )
 \\&
=  \hal \sigma\int_\Sigma   \abs{ \nab_\ast\pa^\al\eta}^2 \diverge_\ast u- \sum_{i=1}^2 \int_\Sigma \sigma \p_i\pa^\al\eta  (\nab_\ast \pa^\al \eta \cdot  \p_i u )
 \\& \ls  \sigma\norm{ \nab_\ast\pa^\al\eta }_{1/2} \norm{\pa^\al \nab_\ast\eta }_{-1/2}\norm{\nab_\ast  u}_{H^2(\Sigma)}
\ls \ks \sqrt{\sd{2N}^\sigma\f \mathcal{E}_{N+2}^0}.
\end{split}
\end{equation}
The estimates \eqref{eta es}--\eqref{eta es2} then follow by summing the estimates.
\end{proof}

Next we consider products with derivatives of $G^{1,1}$.

\begin{lemma}\label{lemma7}
Let $\al\in \mathbb{N}^3$ so that $|\al|=4N$.  Let $f(\bar{\rho})$ denote either $1$ or $h'({\bar\rho})$.
Then
\begin{equation}\label{rho es}
 \abs{\int_{\Omega}   f(\bar{\rho})\pa^\al \q \pa^\al G^{1,1} }  \lesssim  \sqrt{\mathcal{D}_{2N}^0}\sqrt{\se{2N}^0  \sd{2N}^0 + \mathcal{E}_{N+2}^0 \f} .
\end{equation}

\end{lemma}
\begin{proof}
We first use the Leibniz rule to expand that
\begin{equation}
 \pa^\al G^{1,1}
  =   K\p_t\theta \pa_3 \pa^\al \q-u_j \mathcal{A}_{jk}\pa_k \pa^\al \q
 +\sum_{0<\beta \le \al}C_\al^{\beta }\left( \pa^{\beta }(K\p_t\theta) \pa_3 \pa^{\al-\beta } \q
 -\pa^{\beta }(u_j \mathcal{A}_{jk})\pa_k  \pa^{\al-\beta }\q\right).
\end{equation}
For the second term, we estimate for $|\beta|\ge 1$, similarly as Lemma \ref{p_G2N_estimates},
\begin{equation}  \ns{  \pa^{\beta }(K\p_t\theta) \pa_3 \pa^{\al-\beta } \q
 -\pa^{\beta }(u_j \mathcal{A}_{jk})\pa_k  \pa^{\al-\beta }\q }_0 \ls   { \se{2N}^0  \sd{2N}^0 + \mathcal{E}_{N+2}^0 \f}.
\end{equation}
Hence, we have
\begin{equation}\label{111}
\begin{split}
& \left| \int_\Omega   f(\bar{\rho})\pa^\al\q \left(\pa^{\beta }(K\p_t\theta) \pa_3 \pa^{\al-\beta } \q
-\pa^{\beta }(u_j \mathcal{A}_{jk})\pa_k  \pa^{\al-\beta }\q\right)\right|
 \\&\quad \ls \norm{\pa^\al\q}_0 \norm{  \pa^{\beta }(K\p_t\theta) \pa_3 \pa^{\al-\beta } \q
 -\pa^{\beta }(u_j \mathcal{A}_{jk})\pa_k  \pa^{\al-\beta }\q }_0 \\&\quad
\lesssim \sqrt{\mathcal{D}_{2N}^0}\sqrt{\se{2N}^0  \sd{2N}^0 + \mathcal{E}_{N+2}^0 \f} .
\end{split}
\end{equation}

For the first term, we observe that
\begin{equation}\label{222'}
\begin{split}
&\left|\int_\Omega   f(\bar{\rho}) \pa^\al \q   (K\p_t\theta \pa_3  \pa^\al \q-u_j \mathcal{A}_{jk}\pa_k \pa^\al  \q )\right|
\\&\quad=\hal
\left|\int_\Omega    f(\bar{\rho}) K\p_t\theta \pa_3  |\pa^\al \q|^2 -f(\bar{\rho}) u_j \mathcal{A}_{jk}\pa_k   |\pa^\al \q|^2  \right|.
\end{split}
\end{equation}
We then use the integration by parts to have
\begin{equation}\label{222}
\begin{split}
& \left|\int_{\Omega_+}    f(\bar{\rho}) K\p_t\theta \pa_3  |\pa^\al \q|^2 -f(\bar{\rho}) u_j \mathcal{A}_{jk}\pa_k   |\pa^\al \q|^2  \right|
\\&\quad=\left|\int_{\Sigma_+} f(\rho_1 )( K\p_t\theta-  u_{j} \mathcal{A}_{j3})  |\pa^\al \q|^2-\int_{\Sigma_-} f(\rho^+)( K\p_t\theta-  u_{j} \mathcal{A}_{j3})  |\pa^\al \q|^2\right.
\\&\qquad\left. -\int_{\Omega_+}     \left(\pa_3(f(\bar{\rho})K\p_t\theta)-\pa_k(f(\bar{\rho})u_j \mathcal{A}_{jk})\right) |\pa^\al \q|^2   \right|
\\&\quad=\left| \int_\Omega     \left(\pa_3(f(\bar{\rho})K\p_t\theta)-\pa_k(f(\bar{\rho})u_j \mathcal{A}_{jk})\right)  |\pa^\al \q|^2   \right|
\\&\quad\ls \norm{\pa_3(f(\bar{\rho})K\p_t\theta)-\pa_k(u_j \mathcal{A}_{jk})}_{L^\infty}\ns{\pa^\al \q}_0
\lesssim \sqrt{\mathcal{E}_{2N}^0}\mathcal{D}_{2N}^0
\end{split}
\end{equation}
and
\begin{equation}\label{2222}
\begin{split}
& \left|\int_{\Omega_-}    f(\bar{\rho}) K\p_t\theta \pa_3  |\pa^\al \q|^2 -f(\bar{\rho}) u_j \mathcal{A}_{jk}\pa_k   |\pa^\al \q|^2  \right|
\\&\quad=\left| \int_{\Sigma_-} f(\rho^-)( K\p_t\theta-  u_{j} \mathcal{A}_{j3})  |\pa^\al \q_-|^2-\int_\Omega     \left(\pa_3(f(\bar{\rho})K\p_t\theta)-\pa_k(u_j \mathcal{A}_{jk})\right) |\pa^\al \q|^2   \right|
\\&\quad=\left| \int_\Omega     \left(\pa_3(f(\bar{\rho})K\p_t\theta)-\pa_k(u_j \mathcal{A}_{jk})\right)  |\pa^\al \q|^2   \right|
\\&\quad\ls \norm{\pa_3(f(\bar{\rho})K\p_t\theta)-\pa_k(u_j \mathcal{A}_{jk})}_{L^\infty}\ns{\pa^\al \q}_0
\lesssim \sqrt{\mathcal{E}_{2N}^0}\mathcal{D}_{2N}^0.
\end{split}
\end{equation}
Here we have used the facts that $u_j \mathcal{A}_{j3}=K u_j\mathcal{N}_j=K \partial_t\eta =K \partial_t\theta$ on $\Sigma$, and $\theta=0$ and $u_-=0$ on $\Sigma_b$.

Then the estimate \eqref{rho es} follows from \eqref{111} and \eqref{222}.
\end{proof}

Next we consider a similar estimate involving weights and derivatives of $G^{1,1}$.

\begin{lemma}\label{G11_weighted}
Let $\alpha \in \mathbb{N}^{1+3}$ with $\abs{\alpha} \le 4N$ and $\alpha_0 \leq 2N-1$. Then
\begin{equation}
  \abs{\int_\Omega \left(1+\frac{ 4\mu/3+\mu'  }{h'(\bar\rho)\bar\rho^2}\right) \partial^\alpha(h'(\bar\rho) \q)\partial^\alpha \left(h'(\bar\rho) G^{1,1} \right)}
 \ls \ks \sqrt{ \mathcal{D}_{2N}^\sigma }\sqrt{{\mathcal{E}_{2N}^\sigma }\mathcal{D}_{2N}^\sigma+  \mathcal{E}_{N+2}^\sigma\f}.
\end{equation}
\end{lemma}
\begin{proof}
We split to cases.  For the case $1\le  \alpha_0 \le 2N-1$ or the case that $\alpha_0=0$ and $\abs{\alpha} \le 4N-1$, we use the estimates \eqref{p_G_e_00} of Lemma \ref{p_G2N_estimates} to bound
\begin{multline}
   \abs{\int_\Omega \left(1+\frac{ 4\mu/3+\mu'  }{h'(\bar\rho)\bar\rho^2}\right) \partial^\alpha(h'(\bar\rho) \q)\partial^\alpha \left(h'(\bar\rho) G^{1,1} \right)}
\\
 \ls \norm{\bar{\nab}^{4N} \q}_{0} \left(\norm{ \bar{\nab}^{4N-2}  G^{1,1}}_{1}+\norm{ \bar{\nab}^{4N-3}\dt G^{1,1}}_{1} \right)
\\
  \ls \ks \sqrt{ \mathcal{D}_{2N}^\sigma }\sqrt{\ks  \mathcal{E}_{2N}^0 \mathcal{D}_{2N}^\sigma+  \mathcal{E}_{N+2}^0\f}.
\end{multline}

For the remaining case in which $\alpha_0=0$ and $\abs{\alpha} = 4N$, we rewrite
\begin{equation}\label{rhoes99}
h'(\bar\rho) G^{1,1}=K \p_t\theta \pa_3  (h'(\bar\rho)\q)-u_j \mathcal{A}_{jk}\pa_k  (h'(\bar\rho)\q)-K \p_t\theta \pa_3(  h'(\bar\rho) )\q + u_j \mathcal{A}_{j3}\pa_3   (h'(\bar\rho))\q.
\end{equation}
Write $Z_1$ for the sum of the first two terms on the right of \eqref{rhoes99} and $Z_2$ for the second two.  We may argue as in the proof of Lemma \ref{p_G2N_estimates} to estimate
\begin{equation}
   \abs{\int_\Omega \left(1+\frac{ 4\mu/3+\mu'  }{h'(\bar\rho)\bar\rho^2}\right) \partial^\alpha(h'(\bar\rho) \q)\partial^\alpha \left(Z_2 \right)}
     \ls \sqrt{ \mathcal{D}_{2N}^\sigma }\sqrt{{\mathcal{E}_{2N}^\sigma }\mathcal{D}_{2N}^\sigma+  \mathcal{E}_{N+2}^\sigma\f}.
\end{equation}
For the $Z_1$ term we must utilize the total derivative structure
\begin{multline}
 \partial^\alpha(h'(\bar\rho) \q) (K \p_t\theta \pa_3 \p^\alpha (h'(\bar\rho)\q) - u_j \mathcal{A}_{jk}\pa_k  \p^\alpha (h'(\bar\rho)\q)\\
 = K \dt \theta \p_3 \frac{\abs{\p^\alpha (h'(\bar{\rho}) \q)}^2}{2} - u_j \mathcal{A}_{jk}\pa_k \frac{\abs{\p^\alpha (h'(\bar{\rho}) \q)}^2}{2}.
\end{multline}
This allows us to argue as in the proof of Lemma \ref{lemma7}
to bound \begin{equation}
   \abs{\int_\Omega \left(1+\frac{ 4\mu/3+\mu'  }{h'(\bar\rho)\bar\rho^2}\right) \partial^\alpha(h'(\bar\rho) \q)\partial^\alpha \left(Z_1 \right)}
     \ls \ks \sqrt{ \mathcal{D}_{2N}^\sigma }\sqrt{{\mathcal{E}_{2N}^\sigma }\mathcal{D}_{2N}^\sigma+  \mathcal{E}_{N+2}^\sigma\f}.
\end{equation}
Combining the $Z_1$ and $Z_2$ estimates then gives the desired estimate when $\alpha_0=0$ and $\abs{\alpha} = 4N$.

\end{proof}

We next present the estimates of the nonlinear terms $F^{i,j}$, defined by \eqref{F1j_def}--\eqref{F4j_def}, at the $2N$ level.
\begin{lemma}\label{p_F2N_estimates}
Let $F^{i,j}$ be defined by \eqref{F1j_def}--\eqref{F4j_def} for $0\le j\le 2N$.   Then
\begin{equation}\label{p_F_e_01}
 \ns{F^{1,j} }_{0}+ \ns{  F^{2,j} }_{0}   + \ns{F^{3,j}}_{0} + \norm{F^{4,j}}_{0} \ls \ks \se{2N}^0 \sd{2N}^\sigma
\end{equation}
\end{lemma}
\begin{proof}
Note that all terms in the definitions of $F^{i,j}$ are at least quadratic; each term can be written in the form $X Y$, where $X$ involves fewer temporal derivatives than $Y$. We may use the usual Sobolev embeddings, trace theory and Lemma \ref{Poi} along with the definitions of $\se{2N}^\sigma$ and $\sd{2N}^\sigma$ to estimate $\norm{X}_{L^\infty}^2\ls  \se{2N}^0$ and $\norm{Y}_{0}^2\ls \ks \sd{2N}^0$.
Then $\norm{XY}_0^2\le \norm{X}_{L^\infty}^2\norm{Y}_{0}^2\ls \ks \se{2N}^0\sd{2N}^\sigma$, and the estimate \eqref{p_F_e_01} follows by summing.
\end{proof}

Next, we present the estimates of $F^{i,j}$ at the $N+2$ level.
\begin{lemma}\label{p_FN+2_estimates}
 Let $F^{i,j}$ be defined by \eqref{F1j_def}--\eqref{F4j_def} for  $0\le j\le N+2$, we have
\begin{equation}\label{p_F_e_h_01}
 \ns{F^{1,j} }_{0}+ \ns{  F^{2,j} }_{0}   + \ns{F^{3,j}}_{0} + \norm{F^{4,j}}_{0} \ls \ks \se{2N}^0  \sd{N+2}^0
\end{equation}
\end{lemma}
\begin{proof}
The proof proceeds similarly as Lemma \ref{p_F2N_estimates} with the easy replacement of $\norm{Y}_{0}^2\ls \se{2N}^0$ and $\norm{X}_{0}^2\ls  \sd{N+2}0$.  As in Lemma \ref{p_GN+2_estimates} we absorb any appearances of $\sigma$ into $\ks$ so that we can use both $\se{2N}^0$ and $\sd{N+2}^0$.
\end{proof}

Finally, we present an estimate of an auxiliary nonlinearity that will be useful later.  Define
\begin{equation}\label{Phi_def}
\Phi_\pm=(\q_\pm+\p_3\bar\rho_\pm\theta)\p_3 \theta.
\end{equation}

\begin{lemma}\label{Phi_est}
Let $\Phi_\pm$ be as in \eqref{Phi_def}.  Then for   $n=2N$ or $N+2$ we have
\begin{equation}\label{Phe_0}
 \sum_{j=0}^n \ns{\dt^j \Phi}_0 \ls \mathcal{E}_{n}^0 \min\{\mathcal{E}_{n}^0, \mathcal{D}_{n}^0 \}
\end{equation}
\end{lemma}
\begin{proof}
The proof is straightforward, so we omit it for the sake of brevity.
\end{proof}

\section{Horizontal energy evolution}\label{sec_horiz}

In this section we derive estimates for temporal and horizontal spatial derivatives.  We assume throughout this section that the solutions obey the estimate $\gs(T) \le \delta$, where $\delta \in (0,1)$ is given in Lemma \ref{eta_small}.

\subsection{Energy evolution in geometric form}\label{stable1}

We now estimate the energy evolution of the pure temporal derivatives at the $2N$ level.
\begin{Proposition}\label{i_temporal_evolution 2N}
For $0\le j\le 2N$, we have
\begin{multline} \label{tem en 2N}
 \frac{d}{dt}\left(\int_\Omega( \bar{\rho} +  \q+\p_3\bar\rho\theta) J\abs{\partial_t^j u }^2+ h'(\bar{\rho})J \abs{\dt^j\q}^2
+\int_{\Sigma_+} \rho_1  g\abs{\dt^j \eta_+}^2+\sigma_+\abs{\nab_\ast \dt^j \eta_+}^2\right.
\\
 \left.+\int_{\Sigma_-}- \rj g\abs{\dt^j \eta_-}^2+ \sigma_-\abs{\nab_\ast \dt^j \eta_-}^2\right)
+\int_{\Omega}\frac{\mu}{2} \abs{\sgz \partial_t^j u}^2+\mu' \abs{\diverge \partial_t^j u}^2
\ls \ks \sqrt{\se{2N}^\sigma}\sd{2N}^\sigma.
\end{multline}
\end{Proposition}
\begin{proof}
Applying $\dt^j$ to \eqref{ns_geometric} leads to \eqref{linear_geometric}, which is of the same form as \eqref{gee_0} with $v = \dt^j u$, $Q = \dt^j q$, $\zeta = \dt^j \eta$, and $\mathfrak{F}^{k} = \mathcal{F}^{k,j}$ for $k=1,2,3$.  Then Proposition \ref{geo_en_evolve} yields the equation
\begin{multline}\label{identity1}
 \frac{1}{2}\frac{d}{dt}\left(\int_\Omega   ( \bar{\rho} +  \q+\p_3\bar\rho\theta)J\abs{\partial_t^j u }^2+h'(\bar{\rho})J \abs{\dt^j\q}^2+\int_{\Sigma_+} \rho_1  g\abs{\dt^j \eta_+}^2+\sigma_+\abs{\nab_\ast \dt^j \eta_+}^2\right.
\\ \left.+\int_{\Sigma_-}- \rj g\abs{\dt^j \eta_-}^2+ \sigma_-\abs{\nab_\ast \dt^j \eta_-}^2\right)+ \int_{\Omega}\frac{\mu}{2}J\abs{\sgz_\mathcal{A}\partial_t^j u}^2+\mu'J\abs{\diverge_\mathcal{A}\partial_t^j u}^2
\\  =\frac{1}{2} \int_\Omega \dt(J( \bar{\rho} +  \q+\p_3\bar\rho\theta)) \abs{\partial_t^j u }^2+h'(\bar{\rho})\dt J \abs{\dt^j\q}^2+\int_\Omega J(h'(\bar{\rho}) \dt^j\q   F^{1,j}+\partial_t^j u\cdot F^{2,j})
\\ +\int_{\Sigma}-\partial_t^j u\cdot F^{3,j}+\int_{\Sigma_+} \rho_1  g\dt^j \eta_+F_+^{4,j} - \int_{\Sigma_-}\rj g\dt^j \eta_-F_-^{4,j} -\int_\Sigma\sigma\Delta_\ast (\dt^j \eta) F^{4,j}.
\end{multline}

We now estimate the right hand side of \eqref{identity1} for $0\le j\le 2N$. For the first two terms, we may bound  as usual $\norm{\dt J}_{L^\infty}\ls \sqrt{ \mathcal{E}_{2N}^\sigma}$ and $\norm{\dt( J( \bar{\rho} +  \q+\p_3\bar\rho\theta))}_{L^\infty}\ls \sqrt{ \mathcal{E}_{2N}^\sigma}$ to have
 \begin{equation}\label{i_te_1}
\begin{split}
 &\frac{1}{2} \int_\Omega  \dt (J( \bar{\rho} +  \q+\p_3\bar\rho\theta))\abs{\partial_t^j u }^2+h'(\bar{\rho})\dt J \abs{\dt^j\q}^2
 \\&\quad\ls \sqrt{ \mathcal{E}_{2N}^\sigma}\left( \ns{\partial_t^j u}_0+ \ns{\partial_t^j \q}_0\right) \ls \sqrt{ \mathcal{E}_{2N}^\sigma}\mathcal{D}_{2N}^\sigma.
 \end{split}
 \end{equation}
By Lemma \ref{p_F2N_estimates}, we may bound the $F^{1,j}$ and $F^{2,j}$ terms as
\begin{multline}\label{i_te_2}
\int_\Omega J(h'(\bar{\rho}) \dt^j\q   F^{1,j}+\partial_t^j u\cdot F^{2,j}) \ls   \norm{\dt^j \q}_{0}   \norm{F^{1,j}}_0+ \norm{\dt^j u}_{0}   \norm{F^{2,j}}_0
\\
\ls \ks \sqrt{\sd{2N}^\sigma } \sqrt{\se{2N}^\sigma \sd{2N}^\sigma}.
\end{multline}
For the $F^{3,j}$ and $F^{4,j}$ terms, by Lemma \ref{p_F2N_estimates} and   trace theory, we have
\begin{equation}\label{i_te_3}
\begin{split}
 &\int_{\Sigma}-\partial_t^j u\cdot F^{3,j}+\int_{\Sigma_+} \rho_1  g\dt^j \eta_+F_+^{4,j}+\int_{\Sigma_-}-\rj g\dt^j \eta_-F_-^{4,j} -\int_\Sigma\sigma\Delta_\ast (\dt^j \eta) F^{4,j}
 \\&\quad\ls  \snormspace{\dt^{j} u}{0}{\Sigma} \norm{F^{3,j}}_{0} + \sigma\norm{\dt^{j} \eta}_{2} \norm{F^{4,j}}_{0} \\&\quad
\ls  \left( \norm{\dt^{j} u}_{1} + \norm{\dt^{j} \eta}_{2} \right)\ks \sqrt{\se{2N}^\sigma\sd{2N}^\sigma}
\ls \ks  \sqrt{\sd{2N}^\sigma } \sqrt{\se{2N}^\sigma \sd{2N}^\sigma}
.
\end{split}
\end{equation}

 Finally, we may argue similarly as in Proposition 4.3 of \cite{GT_per}, utilizing Lemma \ref{eta_small}, to get
\begin{equation} \label{i_te_4}
 \int_{\Omega}\frac{\mu}{2}
J\abs{\sgz_\mathcal{A}\partial_t^j u}^2+\mu'J\abs{\diverge_\mathcal{A}\partial_t^j u}^2
\ge  \int_{\Omega}\frac{\mu}{2}
 \abs{\sgz \partial_t^j u}^2+\mu' \abs{\diverge \partial_t^j u}^2-C\sqrt{ \mathcal{E}_{2N}^\sigma}\mathcal{D}_{2N}^\sigma.
 \end{equation}
We may then combine \eqref{i_te_1}--\eqref{i_te_4} to deduce \eqref{tem en 2N} from \eqref{identity1}.
\end{proof}

We then record a similar result at the $N+2$ level.
\begin{Proposition}\label{i_temporal_evolution N+2}
For $0\le j\le N+2$, we have
\begin{multline} \label{tem en N+2}
  \frac{d}{dt}\left(\int_\Omega ( \bar{\rho} +  \q+\p_3\bar\rho\theta) J\abs{\partial_t^j u }^2+h'(\bar{\rho})J \abs{\dt^j\q}^2 +\int_{\Sigma_+} \rho_1  g\abs{\dt^j \eta_+}^2+\sigma_+\abs{\nab_\ast \dt^j \eta_+}^2\right.
\\
 \left.+\int_{\Sigma_-}- \rj g\abs{\dt^j \eta_-}^2+ \sigma_-\abs{\nab_\ast \dt^j \eta_-}^2\right)  +\int_{\Omega}\frac{\mu}{2} \abs{\sgz \partial_t^j u}^2+\mu' \abs{\diverge \partial_t^j u}^2
\\
\ls \ks \sqrt{\se{2N}^\sigma}\sd{N+2}^\sigma.
 \end{multline}
\end{Proposition}
\begin{proof}
The proof proceeds along similar lines as the proof of  Proposition \ref{i_temporal_evolution 2N}, using instead Lemma \ref{p_FN+2_estimates} and replacing $\mathcal{D}_{2N}^\sigma$ by $\mathcal{D}_{N+2}^\sigma$ accordingly in the estimates \eqref{i_te_1} and \eqref{i_te_4}.
\end{proof}

\subsection{Energy evolution in the linear form}\label{stable2}

To derive the energy evolution of the mixed horizontal space-time derivatives of the solution we will use the linear formulation \eqref{ns_perturb}.

We now estimate the energy evolution of the mixed horizontal space-time derivatives at the $2N$ level.
 \begin{Proposition}\label{i_spatial_evolution 2N}
Let $\alpha\in \mathbb{N}^{1+2}$ so that $\al_0\le 2N-1$ and $|\alpha|\le 4N$. Then
\begin{equation}\label{energy 2N}
\begin{split}
& \frac{d}{dt}\left(\int_\Omega \bar{\rho}\abs{\pa^\al u}^2 +h'(\bar{\rho})\abs{ \pa^\al \q}^2+\int_{\Sigma_+} \rho_1  g\abs{\pa^\al \eta_+}^2+\sigma_+\abs{\nab_\ast \pa^\al \eta_+}^2\right.
\\&\quad\ \left. +\int_{\Sigma_-}- \rj g\abs{\pa^\al \eta_-}^2+ \sigma_-\abs{\nab_\ast \pa^\al \eta_-}^2\right)
+\int_\Omega \frac{\mu}{2} \abs{ \sgz\pa^\al u}^2+\mu' \abs{{\rm div }\pa^\al u}^2 \\
&\quad\lesssim \ks \sqrt{\mathcal{E}_{2N}^\sigma }\mathcal{D}_{2N}^\sigma
+ \ks\sqrt{ \mathcal{D}_{2N}^\sigma \mathcal{E}_{N+2}^\sigma\f}.
\end{split}
\end{equation}
 \end{Proposition}
\begin{proof}
Since all the three boundaries of $\Omega$ are horizontally flat we are free to take the horizontal space-time differential operator in the system \eqref{ns_perturb}. Applying $\partial^\alpha$  to $\eqref{ns_perturb}$, we arrive at an equation of the form \eqref{lee_0} with $v = \pa^\al u$, $Q = \pa^\al q$, $\zeta = \pa^\al \eta$, $\mathfrak{G}^k = \pa^\al G^k$ for $k=1,2,3$.  Then  Proposition \ref{lin_en_evolve} provides us with the equality
\begin{equation}\label{es_00}
\begin{split}
&\frac{1}{2}\frac{d}{dt}\left(\int_\Omega h'(\bar{\rho})\abs{ \pa^\al \q}^2+\bar{\rho}\abs{\pa^\al u}^2 +\int_{\Sigma_+} \rho_1  g\abs{\pa^\al \eta_+}^2+\sigma_+\abs{\nab_\ast \pa^\al \eta_+}^2\right.
\\&\qquad \left. +\int_{\Sigma_-}- \rj g\abs{\pa^\al \eta_-}^2+ \sigma_-\abs{\nab_\ast \pa^\al \eta_-}^2\right)
 +\int_\Omega \frac{\mu}{2} \abs{ \sgz\pa^\al u}^2+\mu' \abs{{\rm div }\pa^\al u}^2
\\&\qquad    = \int_\Omega h'(\bar{\rho})\pa^\al \q \pa^\al G^1 +\pa^\al u\cdot\pa^\al G^2+\int_\Sigma -\pa^\al u\cdot \pa^\al G^3
\\&\qquad\quad+ \int_{\Sigma_+} \rho_1  g \pa^\al \eta_+ \pa^\al G^4_++\int_{\Sigma_-}-\rj g\pa^\al \eta_- \pa^\al G^4_- -\int_\Sigma\sigma\Delta_\ast (\pa^\al \eta) \pa^\al G^4.
\end{split}
\end{equation}

 We now estimate the right hand side of \eqref{es_00}. We first estimate the $G^2,G^3,G^4$ terms. We assume initially that $|\al|\le 4N-1$. Then by the estimates \eqref{p_G_e_00} of Lemma \ref{p_G2N_estimates}, we have
\begin{equation}\label{gg1}
 \left|\int_\Omega  \pa^\al u\cdot\pa^\al G^2\right| \le  \norm{\pa^\al u}_0\norm{\pa^\al G^2}_0 \lesssim \sqrt{\mathcal{D}_{2N}^\sigma}\sqrt{\ks \mathcal{E}_{2N}^\sigma \mathcal{D}_{2N}^\sigma + \mathcal{E}_{N+2}^\sigma\f}.
\end{equation}
Similarly,  by the estimates \eqref{p_G_e_00} of Lemma \ref{p_G2N_estimates} along with  trace theory, we have
\begin{equation}
\begin{split}
 \left| \int_\Sigma \pa^\al u\cdot \pa^\al G^3 \right|&\le  \norm{\pa^\al u}_{H^0(\Sigma)}  \norm{\pa^\al G^3}_0 \ls  \norm{\pa^\al u}_{1}  \norm{\pa^\al G^3}_0
\\&\lesssim \sqrt{\mathcal{D}_{2N}^\sigma}\sqrt{\ks \mathcal{E}_{2N}^\sigma\mathcal{D}_{2N}^\sigma+\mathcal{E}_{N+2}^\sigma\f}
 \end{split}
\end{equation}
and
\begin{equation}
\begin{split}
 &\left|\int_{\Sigma_+} \rho_1  g \pa^\al \eta_+ \pa^\al G^4_++\int_{\Sigma_-}-\rj g\pa^\al \eta_- \pa^\al G^4_- -\int_\Sigma\sigma\Delta_\ast (\pa^\al \eta) \pa^\al G^4\right|
\\&\quad\lesssim\left(\norm{\pa^\al \eta}_0+\sigma\norm{\pa^\al \eta}_2\right)\norm{\pa^\al G^4}_0\lesssim \ks \sqrt{\mathcal{D}_{2N}^\sigma}\sqrt{\mathcal{E}_{2N}^\sigma\mathcal{D}_{2N}^\sigma+\mathcal{E}_{N+2}^\sigma\f}.
\end{split}
\end{equation}

Now we assume that $|\al|=4N$. We first estimate the $G^2,G^3$ terms. Since $\al_0\le 2N-1$, then $\pa^\al$ involves at least two spatial derivatives, we may write $\al=\beta+(\al-\beta)$ for some $\beta\in\mathbb{N}^2$ with $|\beta|=1$. We then integrate by parts and use the estimates \eqref{p_G_e_00} of Lemma \ref{p_G2N_estimates} to have
\begin{equation}
\begin{split}
 \left|\int_\Omega  \pa^\al u\cdot\pa^\al G^2\right|
&=\left|\int_\Omega  \pa^{\al+\beta} u\cdot\pa^{\al-\beta} G^2\right| \lesssim \norm{\pa^{\al+\beta}u}_{0}\norm{\pa^{\al-\beta} G^2}_{0}
\\&\lesssim \norm{\pa^\al u}_{1}\norm{\bar{\na}_{\ast}^{4N-1} G^2}_{0}\lesssim \sqrt{\mathcal{D}_{2N}^\sigma}\sqrt{\ks \mathcal{E}_{2N}^\sigma\mathcal{D}_{2N}^\sigma+\mathcal{E}_{N+2}^\sigma\f}.
\end{split}
\end{equation}
Similarly, by using additionally  trace theory, we have
\begin{equation}
\begin{split}
 \left|\int_\Sigma  \pa^\al u\cdot\pa^\al G^3\right|
&=\left|\int_\Sigma  \pa^{\al+\beta}  u\cdot\pa^{\al-\beta} G^3\right|\lesssim \norm{\pa^{\al+\beta} u}_{H^{-1/2}(\Sigma)}\norm{\pa^{\al-\beta} G^3}_{1/2}
\\&\lesssim \norm{\pa^{\al } u}_{H^{ 1/2}(\Sigma)}\norm{\bar{\na}_{\ast}^{\,\,4N-1} G^3}_{1/2}\lesssim \norm{\pa^\al u}_{1}\norm{\bar{\na}_{\ast}^{4N-1} G^3}_{1/2}\\&\lesssim \sqrt{\mathcal{D}_{2N}^\sigma}\sqrt{\ks \mathcal{E}_{2N}^\sigma\mathcal{D}_{2N}^\sigma+\mathcal{E}_{N+2}^\sigma\f}.
\end{split}
\end{equation}
We then estimate the $G^4$ term,  splitting to two cases: $\al_0\ge 1$ and $\al_0=0$. If $\al_0\ge 1$, then $\pa^\al$ involves at least one temporal derivative so that $\norm{\pa^\al \eta}_{3/2}\le \norm{\dt^{\al_0} \eta}_{4N-2\al_0+3/2}\le \mathcal{D}_{2N}^0$. This together with the estimates \eqref{p_G_e_00} of Lemma \ref{p_G2N_estimates} implies
\begin{equation}
\begin{split}
 &\left| \int_{\Sigma_+} \rho_1  g \pa^\al \eta_+ \pa^\al G^4_++\int_{\Sigma_-}-\rj g\pa^\al \eta_- \pa^\al G^4_- -\int_\Sigma\sigma\Delta_\ast (\pa^\al \eta) \pa^\al G^4\right|
\\&\quad\lesssim\left(\norm{\pa^\al \eta}_{0}+\sigma\norm{\pa^\al \eta}_{3/2}\right)\norm{\pa^\al G^4}_{1/2}
\lesssim \ks\sqrt{\mathcal{D}_{2N}^\sigma}\sqrt{\mathcal{E}_{2N}^\sigma\mathcal{D}_{2N}^\sigma+\mathcal{E}_{N+2}^\sigma\f}.
\end{split}
\end{equation}
If $\al_0=0$, we must resort to the special estimates \eqref{eta es}--\eqref{eta es2} of Lemma \ref{lemma8} to have
\begin{equation}
\begin{split}
 &\left| \int_{\Sigma_+} \rho_1  g \pa^\al \eta_+ \pa^\al G^4_++\int_{\Sigma_-}-\rj g\pa^\al \eta_- \pa^\al G^4_- -\int_\Sigma\sigma\Delta_\ast \pa^\al \eta \pa^\al G^4\right|
\\&\quad \lesssim
\ks\sqrt{\mathcal{D}_{2N}^\sigma}\sqrt{\ks \mathcal{E}_{2N}^\sigma\mathcal{D}_{2N}^\sigma+\mathcal{E}_{N+2}^\sigma\f}.
 \end{split}
\end{equation}

We now turn back to estimate the $G^1$ term, and we recall $G^1=G^{1,1}+G^{1,2}$. For the $G^{1,2}$ part, it follows directly from the estimates \eqref{p_G_e_00} of Lemma \ref{p_G2N_estimates} that
\begin{equation}\label{es09}
 \left|\int_\Omega h'(\bar{\rho}) \pa^\al \q \pa^\al G^{1,2} \right|\ls  \norm{\pa^\al \q}_{0}\norm{\pa^\al G^{1,2}}_{0} \lesssim \sqrt{\mathcal{D}_{2N}^\sigma}\sqrt{\mathcal{E}_{2N}^\sigma\mathcal{D}_{2N}^\sigma+\mathcal{E}_{N+2}^\sigma\f}.
\end{equation}
Now for the $G^{1,1}$ term we must split to two cases: $\al_0\ge 1$ and $\al_0=0$. If $\al_0\ge 1$, then by the estimates \eqref{p_G_e_00} of Lemma \ref{p_G2N_estimates}, we have
\begin{equation}
 \left|\int_\Omega h'(\bar{\rho}) \pa^\al \q \pa^\al G^{1,1} \right|\ls  \norm{\pa^\al \q}_{0}\norm{\pa^\al G^{1,1}}_{0}\lesssim \sqrt{\mathcal{D}_{2N}^\sigma}\sqrt{\mathcal{E}_{2N}^\sigma\mathcal{D}_{2N}^\sigma+\mathcal{E}_{N+2}^\sigma\f}.
\end{equation}
If $\al_0=0$, we must resort to the special estimates \eqref{rho es} with $f(\bar\rho)=h'(\bar{\rho})$ of Lemma \ref{lemma7} to have
\begin{equation}\label{gg10}
 \left|\int_\Omega h'(\bar{\rho}) \pa^\al \q \pa^\al G^{1,1} \right| \lesssim \sqrt{\mathcal{D}_{2N}^\sigma}\sqrt{\mathcal{E}_{2N}^\sigma\mathcal{D}_{2N}^\sigma+\mathcal{E}_{N+2}^\sigma\f}.
\end{equation}

Hence, in light of \eqref{gg1}--\eqref{gg10}, we deduce from \eqref{es_00} that
\begin{equation}\label{gg11}
\begin{split}
&\frac{1}{2}\frac{d}{dt}\left(\int_\Omega \bar{\rho}\abs{\pa^\al u}^2 +h'(\bar{\rho})\abs{ \pa^\al \q}^2+\int_{\Sigma_+} \rho_1  g\abs{\pa^\al \eta_+}^2+\sigma_+\abs{\nab_\ast \pa^\al \eta_+}^2\right.
\\&\qquad \left. +\int_{\Sigma_-}- \rj g\abs{\pa^\al \eta_-}^2+ \sigma_-\abs{\nab_\ast \pa^\al \eta_-}^2\right)
 +\int_\Omega \frac{\mu}{2} \abs{ \mathbb{D}\pa^\al u}^2+\mu' \abs{{\rm div }\pa^\al u}^2
 \\&\qquad \lesssim \ks \sqrt{\mathcal{D}_{2N}^\sigma}\sqrt{\mathcal{E}_{2N}^\sigma\mathcal{D}_{2N}^\sigma+\mathcal{E}_{N+2}^\sigma\f}
 \\&\qquad  \lesssim \ks \sqrt{\mathcal{E}_{2N}^\sigma }\mathcal{D}_{2N}^\sigma + \ks\sqrt{ \mathcal{D}_{2N}^\sigma \mathcal{E}_{N+2}^\sigma\f},
\end{split}
\end{equation}
which is \eqref{energy 2N}.
\end{proof}

We then record a similar result at the $N+2$ level.
\begin{Proposition}\label{i_spatial_evolution N+2}
Let $\alpha\in \mathbb{N}^{1+2}$ so that $\al_0\le N+1$ and $|\alpha|\le 2(N+2)$. Then
\begin{equation}\label{energy N+2}
\begin{split}
&  \frac{d}{dt}\left(\int_\Omega \bar{\rho}\abs{\pa^\al u}^2 +h'(\bar{\rho})\abs{ \pa^\al \q}^2+\int_{\Sigma_+} \rho_1  g\abs{\pa^\al \eta_+}^2+\sigma_+\abs{\nab_\ast \pa^\al \eta_+}^2\right.
\\&\quad\ \left. +\int_{\Sigma_-}- \rj g\abs{\pa^\al \eta_-}^2+ \sigma_-\abs{\nab_\ast \pa^\al \eta_-}^2\right)
 +\int_\Omega \frac{\mu}{2} \abs{ \sgz\pa^\al u}^2+\mu' \abs{{\rm div }\pa^\al u}^2\\
&\quad \lesssim \ks\sqrt{\mathcal{E}_{2N}^\sigma }\mathcal{D}_{N+2}^\sigma.
\end{split}
\end{equation}
 \end{Proposition}
\begin{proof}
The proof is the same as that of Proposition \ref{i_spatial_evolution 2N} except that we instead use  the estimates
\eqref{p_G_e_h_00} of Lemma \ref{p_GN+2_estimates}.
\end{proof}

\subsection{Energy positivity}

We will now verify the key issue that the energy expressions in the previous two subsections are positive.

We first prove the following lemma that provides the value of the critical surface tension value $\sigma_c$ defined by \eqref{sigma_c}.

\begin{lemma}\label{critical lemma}
Suppose that $\rj \ge 0$, which means that $\sigma_c$, defined by \eqref{sigma_c}, is non-negative.  If $\sigma_- > \sigma_c$, then the following hold for all $\zeta$ satisfying  $\int_{\mathrm{T}^2}\zeta=0$.
\begin{enumerate}
\item We have the estimate
\begin{equation}\label{criticals}
(\sigma_- - \sigma_c)\norm{\zeta}_1^2 \ls \sigma_-\norm{\nabla_\ast \zeta}_0^2-\rj g\norm{\zeta}_0^2.
\end{equation}
\item If $\zeta$ satisfies
\begin{equation}\label{tequa}
-\sigma_-\Delta_\ast \zeta-\rj g\zeta=\varphi\text{ on }\mathrm{T}^2,
\end{equation}
then we have for $r\ge 2$,
\begin{equation}\label{elliptt}
(\sigma_- - \sigma_c)\norm{\zeta}_{r} \lesssim \norm{\varphi}_{r-2}.
\end{equation}
\end{enumerate}
\end{lemma}

\begin{proof}
Note first that
\begin{equation}\label{cl_min}
 \min_{n \in (L_1^{-1}\mathbb{Z})\times (L_2^{-1}\mathbb{Z}) \backslash \{0\}} \abs{n}^2 = \min\{L_1^{-2},L_2^{-2}\}  = \frac{1}{\max\{L_1^2,L_2^2\}}.
\end{equation}
Then since $\int_{\mathrm{T}^2}\zeta=0$, i.e. $\hat{\zeta}(0)=0$, we may use the Parseval theorem to prove Poincar\'e's inequality with the precise constant:
\begin{equation}\label{cl_poin}
\begin{split}
\norm{\nabla_\ast\zeta}_0^2& =\sum_{n\in (L_1^{-1}\mathbb{Z})\times (L_2^{-1}\mathbb{Z})\backslash\{0\}} |n|^2|\hat{\zeta}(n)|^2
\\&\ge \frac{1}{\max\{L_1^2,L_2^2\}}  \sum_{n\in (L_1^{-1}\mathbb{Z}) \times (L_2^{-1}\mathbb{Z})\backslash\{0\}} |\hat{\zeta}(n)|^2=\frac{1}{\max\{L_1^2,L_2^2\}}\norm{ \zeta }_0^2.
\end{split}
\end{equation}
This inequality is clearly sharp.

From \eqref{cl_poin} we see that
\begin{equation}\label{ineq}
\begin{split}
\sigma_-\norm{\nabla_\ast \zeta}_0^2-\rj g\norm{\zeta}_0^2 &\ge (\sigma_--\rj g\max\{L_1^2,L_2^2\})\norm{\nabla_\ast \zeta}_0^2
\\&=(\sigma_- -\sigma_c) \norm{\nabla_\ast \zeta}_0^2.
\end{split}
\end{equation}
Combining \eqref{cl_poin} with \eqref{ineq} then yields  \eqref{criticals}.

To prove \eqref{elliptt}, we note that \eqref{tequa} implies that the Fourier coefficients of $\zeta$ and $\varphi$ satisfy
\begin{equation}
 (\sigma_- \abs{n}^2 - \rj g) \hat{\zeta}(n) = \hat{\varphi}(n) \text{ for all } n \in (L_1^{-1}\mathbb{Z})\times (L_2^{-1}\mathbb{Z}),
\end{equation}
which in particular implies that $\hat{\varphi}(0)=0$.   Again \eqref{cl_min} implies that
\begin{equation}
 \sigma_- \abs{n}^2 - \rj g =  \sigma_- \abs{n}^2 - \frac{\sigma_c}{\max\{L_1^2,L_2^2\}} \ge (\sigma_- - \sigma_c) \abs{n}^2 \ge 0.
\end{equation}
Then
\begin{equation}
 (\sigma_- - \sigma_c)^2 \abs{n}^4 \abs{\hat{\zeta}(n)}^2 \le \abs{(\sigma_- \abs{n}^2 - \rj g) \hat{\zeta}(n)}^2 = \abs{\hat{\varphi}(n)}^2,
\end{equation}
and hence
\begin{multline}
 (\sigma_- - \sigma_c)^2 \ns{\zeta}_{r} \ls (\sigma_- - \sigma_c)^2  \sum_{n\in (L_1^{-1}\mathbb{Z})\times (L_2^{-1}\mathbb{Z})\backslash\{0\}} (1+ \abs{n}^2)^r \abs{\hat{\zeta}(n)}^2 \\
 \le \sum_{n\in (L_1^{-1}\mathbb{Z})\times (L_2^{-1}\mathbb{Z})\backslash\{0\}} \frac{(1+ \abs{n}^2)^r}{\abs{n}^4} \abs{\hat{\varphi}(n)}^2 \ls \sum_{n\in (L_1^{-1}\mathbb{Z})\times (L_2^{-1}\mathbb{Z})\backslash\{0\}} (1+ \abs{n}^2)^{r-2} \abs{\hat{\varphi}(n)}^2 \\
\ls \ns{\varphi}_{r-2}.
\end{multline}

\end{proof}

We now show the following positivity:

\begin{Proposition}\label{Energy positivity pro}
Let $\alpha\in \mathbb{N}^{1+2}$ so that $|\alpha|\le 2n$ with $n=2N$ or $n=N+2$. If $\rj < 0$, then
\begin{equation}\label{epp_0}
\begin{split}
&\int_\Omega h'(\bar{\rho})\abs{ \pa^\al \q}^2 +\int_{\Sigma_+} \rho_1  g\abs{\pa^\al \eta_+}^2+\sigma_+\abs{\nab_\ast \pa^\al \eta_+}^2+\int_{\Sigma_-}- \rj g\abs{\pa^\al \eta_-}^2 +\sigma_-\abs{\nab_\ast \pa^\al \eta_-}^2
\\&\quad\gss \norm{\pa^\al \q}_0^2+\norm{\pa^\al\eta}_0^2+\sigma \norm{\nab_\ast \pa^\al\eta}_0^2.
\end{split}
\end{equation}
If $\rj \ge 0$ then
\begin{equation}\label{Energy positivity inequality}
\begin{split}
&\int_\Omega h'(\bar{\rho})\abs{ \pa^\al \q}^2 +\int_{\Sigma_+} \rho_1  g\abs{\pa^\al \eta_+}^2+\sigma_+\abs{\nab_\ast \pa^\al \eta_+}^2+\int_{\Sigma_-}- \rj g\abs{\pa^\al \eta_-}^2 +\sigma_-\abs{\nab_\ast \pa^\al \eta_-}^2
\\&\quad\gss
\norm{ \pa^\al \q}_0^2+  \min\{1,\sigma_+\}  \norm{\pa^\al \eta_+}_1^2  + \min\{1,\sigma_- - \sigma_c\} \ns{\pa^\al \eta_-}_{1} -\sqrt{\mathcal{E}_{n}^\sigma} \min\{\mathcal{E}_{n}^\sigma, \mathcal{D}_{n}^\sigma \}
\end{split}
\end{equation}
\end{Proposition}
\begin{proof}
In the case $\rj < 0$ the result \eqref{epp_0} is trivial, so we assume in what follows that $\rj \ge 0$ and focus on \eqref{Energy positivity inequality}.

For $\al$ with $\al_1+\al_2\neq 0$, we can easily get \eqref{Energy positivity inequality} by \eqref{criticals} since in this case $\zeta = \pa^\al\eta_-$ has the zero average over $\Sigma_-$, and so Lemma \ref{critical lemma} is applicable. Now for $\al_1=\al_2= 0$, that is, the pure temporal derivative case, and the problem is that they do not have the zero average.

We first deal with the case $\al_0=0$. We will make use of the conservation of mass:
\begin{equation}
\int_{\Omega_\pm}\rho_\pm J =\int_{\Omega_\pm}(\bar\rho_\pm+\q_\pm+\p_3\bar\rho_\pm\theta) (1+\p_3 \theta) =\int_{\Omega_\pm}\bar\rho_\pm.
\end{equation}
This implies
\begin{equation}
\int_{\Omega_\pm} \q_\pm+\int_{\Omega_\pm}\p_3(\bar\rho_\pm\theta)+ \int_{\Omega_\pm}\Phi_\pm  =0,
\end{equation}
where we have denoted $\Phi_\pm=(\q_\pm+\p_3\bar\rho_\pm\theta)\p_3 \theta$ as in \eqref{Phi_def}.  On the other hand, we have
\begin{equation}
\int_{\Omega_+} \p_3(\bar\rho_+\theta)  =\rho_1 \int_{\Sigma_+}\eta_+-\rho^+\int_{\Sigma_-}\eta_-
\end{equation}
and
\begin{equation}
\int_{\Omega_-} \p_3(\bar\rho_-\theta)  = \rho^-\int_{\Sigma_-}\eta_-
\end{equation}
Then we obtain
\begin{equation}\label{cons1}
\int_{\Omega_+} \q_+ + \int_{\Omega_+}\Phi_+  =-\rho_1 \int_{\Sigma_+}\eta_++\rho^+\int_{\Sigma_-}\eta_-
\end{equation}
and
\begin{equation}\label{cons2}
\int_{\Omega_-} \q_- + \int_{\Omega_-}\Phi_-  =-\rho^-\int_{\Sigma_-}\eta_-,
\end{equation}

We rewrite
\begin{equation}\label{e00}
\begin{split}
  \norm{\sqrt{h'(\bar\rho)}\q}_0^2 & +\rho_1 g\norm{\eta_+}_0^2-\rj g\norm{\eta_-}_0^2+\sigma_- \norm{\na_\ast\eta_-}_0^2 + \sigma_+ \norm{\na_\ast\eta_+}_0^2
 \\
&= \norm{\sqrt{h_+'(\bar\rho_+)}\q_+}_0^2+\rho_1 g\norm{\eta_+}_0^2-\rj g{|\mathrm{T}^2|}(\eta_-)^2   \\
&+ \sigma_- \norm{\na_\ast\eta_-}_0^2  -\rj g \norm{\eta_- -(\eta_-)}_0^2 + \norm{\sqrt{h_-'(\bar\rho_-)}\q_-}_0^2  + \sigma_+ \norm{\na_\ast\eta_+}_0^2,
 \end{split}
 \end{equation}
 where we use ``$(\eta_\pm)$" to denote the average of $\eta_\pm$ over $\mathrm{T}^2$.  We first estimate the second line of \eqref{e00}.  We deduce from \eqref{cons1}, Cauchy's inequality, and Lemma \ref{Phi_est} that for any $\kappa>0$,
\begin{multline}\label{e1}
\rj g{|\mathrm{T}^2|}(\eta_-)^2 =\frac{\rj g}{|\mathrm{T}^2|(\rho^+)^2}\left(\rho_1 {|\mathrm{T}^2|}(\eta_+) +\int_{\Omega_+} \q_++ \int_{\Omega_+}\Phi_+  \right)^2
\\
\le  \frac{\rj g(\rho_1 )^2{|\mathrm{T}^2|}}{ (\rho^+)^2}\left(1+\frac{1}{\kappa}\right) (\eta_+)^2 +\frac{\rj g(1+\kappa)}{|\mathrm{T}^2|(\rho^+)^2}\left(\int_{\Omega_+} \q_+\right)^2
\\
+C \sqrt{\mathcal{E}_{n}^\sigma} \min\{\mathcal{E}_{n}^\sigma, \mathcal{D}_{n}^\sigma \}.
\end{multline}
By H\"older's inequality, we have
\begin{equation}
\begin{split}
\left(\int_{\Omega_+} \q_+\right)^2&\le  \left(\int_{\Omega_+} h_+'(\bar\rho_+)|\q_+|^2 \right) \left(\int_{\Omega_+} \frac{1}{h_+'(\bar\rho_+)}\right)
\\&= \norm{\sqrt{h_+'(\bar\rho_+)}\q_+}_0^2\int_{\Omega_+} -\frac{\p_3\bar\rho_+}{g}= \frac{|\mathrm{T}^2|(\rho^+-\rho_1 )}{g}\norm{\sqrt{h_+'(\bar\rho_+)}\q_+}_0^2.
\end{split}
\end{equation}
Now we choose the value of $\kappa$ via
\begin{equation}
\frac{ (1+\kappa)(\rho^+-\rho_1 )}{  \rho^+  }=1 \Leftrightarrow \kappa=\frac{\rho_1 }{\rho^+-\rho_1 }.
\end{equation}
Then we obtain
\begin{equation}\label{e2}
\begin{split}
\frac{\rj g(1+\kappa)}{|\mathrm{T}^2|(\rho^+)^2}\left(\int_{\Omega_+} \q_+\right)^2
&\le \frac{\rj (1+\kappa)(\rho^+-\rho_1 )}{ (\rho^+)^2}\norm{\sqrt{h_+'(\bar\rho_+)}\q_+}_0^2
\\&= \frac{\rj }{ \rho^+}\norm{\sqrt{h_+'(\bar\rho_+)}\q_+}_0^2,
\end{split}
 \end{equation}
 and by H\"older's inequality,
 \begin{equation}\label{e3}
 \begin{split}
  \frac{\rj g(\rho_1 )^2{|\mathrm{T}^2|}}{ (\rho^+)^2}\left(1+\frac{1}{\kappa}\right) (\eta_+)^2
  =\frac{\rj g \rho_1  {|\mathrm{T}^2|}}{  \rho^+ }  (\eta_+)^2\le \frac{\rj g \rho_1   }{  \rho^+ }  \norm{\eta_+}_0^2.
\end{split}
 \end{equation}
 We deduce from \eqref{e1}, \eqref{e2} and \eqref{e3} that
 \begin{equation}\label{e4}
  \begin{split}
 &\norm{\sqrt{h_+'(\bar\rho_+)}\q_+}_0^2+\rho_1 g\norm{\eta_+}_0^2-\rj g {|\mathrm{T}^2|}(\eta_-)^2
 \\&\quad\ge \left(1-\frac{\rj }{ \rho^+}\right)\left(\norm{\sqrt{h_+'(\bar\rho_+)}\q_+}_0^2+\rho_1 g\norm{\eta_+}_0^2\right)-C\sqrt{\mathcal{E}_{n}^\sigma} \min\{\mathcal{E}_{n}^\sigma, \mathcal{D}_{n}^\sigma \}  \\
&\quad \gss \norm{ \q_+}_0^2+ \norm{\eta_+}_0^2- \sqrt{\mathcal{E}_{n}^\sigma} \min\{\mathcal{E}_{n}^\sigma, \mathcal{D}_{n}^\sigma \} .
 \end{split}
 \end{equation}

Next we use the estimate \eqref{ineq} to obtain a bound for the first two terms on the third line of  \eqref{e00}:
\begin{equation} \label{e0}
\sigma_- \norm{\na_\ast\eta_-}_0^2-\rj g \norm{\eta_- -(\eta_-)}_0^2 \ge (\sigma_- - \sigma_c)  \norm{\na_\ast\eta_-}_0^2 .
\end{equation}
From \eqref{cons2}, H\"older's inequality, and Lemma \ref{Phi_est} to obtain
\begin{equation}\label{e5}
\begin{split}
  \norm{\eta_-}_0^2 =\norm{\eta_- -(\eta_-)}_0^2
 +{|\mathrm{T}^2|}(\eta_-)^2
  &\ls \norm{\na_\ast\eta_-}_{0}^2+\left(\int_{\Omega_-} \q_-\right)^2+\sqrt{\mathcal{E}_{n}^\sigma} \min\{\mathcal{E}_{n}^\sigma, \mathcal{D}_{n}^\sigma \}
\\& \ls \norm{\sqrt{h_-'(\bar\rho_-)}\q_-}_0^2+\norm{\na_\ast\eta_-}_{0}^2+\sqrt{\mathcal{E}_{n}^\sigma} \min\{\mathcal{E}_{n}^\sigma, \mathcal{D}_{n}^\sigma \}.
 \end{split}
\end{equation}
Then we may combine \eqref{e0} and \eqref{e5} to estimate the full third line of \eqref{e00}:
\begin{multline}\label{e6}
  \sigma_- \norm{\na_\ast\eta_-}_0^2  -\rj g \norm{\eta_- -(\eta_-)}_0^2 + \norm{\sqrt{h_-'(\bar\rho_-)}\q_-}_0^2  + \sigma_+ \norm{\na_\ast\eta_+}_0^2 \\
\gss \sigma_+ \norm{\na_\ast\eta_+}_0^2 + \min\{1,\sigma_- - \sigma_c\} \ns{\eta_-}_{1}  - \sqrt{\mathcal{E}_{n}^\sigma} \min\{\mathcal{E}_{n}^\sigma, \mathcal{D}_{n}^\sigma \}
\end{multline}
We may then conclude from \eqref{e00},  \eqref{e4}, and \eqref{e6} that
 \begin{equation}
\begin{split}
 & \norm{\sqrt{h'(\bar\rho)}\q}_0^2+\rho_1 g\norm{\eta_+}_0^2+\sigma_+\norm{\na_\ast\eta_+}_0^2-\rj g\norm{\eta_-}_0^2+\sigma \norm{\na_\ast\eta_-}_0^2
 \\&\quad \gss \norm{ \q}_0^2+ \min\{1,\sigma_+\}   \norm{\eta_+}_1^2  + \min\{1,\sigma_- - \sigma_c\} \ns{\eta_-}_{1} -\sqrt{\mathcal{E}_{n}^\sigma} \min\{\mathcal{E}_{n}^\sigma, \mathcal{D}_{n}^\sigma \} .
 \end{split}
 \end{equation}
This is \eqref{Energy positivity inequality} when $\alpha =0.$

To derive \eqref{Energy positivity inequality} for $\alpha_1=\alpha_2 =0$, $\alpha_0 >0$ we first take the time derivatives in \eqref{cons1}--\eqref{cons2}.  Then we may argue as above to derive the desired estimate.
\end{proof}

\section{The evolution of energies controlling $\pa_3\q$} \label{sec_aux}

In this section we identify a dissipative structure for $\pa_3 q$ and derive some energy-dissipation estimates.  We again assume throughout this section that the solutions obey the estimate $\gs(T) \le \delta$, where $\delta \in (0,1)$ is given in  Lemma \ref{eta_small}.

\subsection{Identifying the dissipative structure }

Note that the energy evolution results presented in Sections \ref{stable1}--\ref{stable2} are not enough for us to get  full energy estimates by applying Stokes regularity estimates as in the incompressible case \cite{WT,WTK}.  The problem is that we do not yet have control of $\diverge u$.  To control $\diverge u$ we appeal to a structure  first exploited by Matsumura and Nishida \cite{MN83} in the case in which $\bar\rho$ is a positive constant.  We consider  the following quantity, which is the material derivative of $q$ in our coordinates:
\begin{equation}\label{dt}
\mathcal{Q}:= \pa_t\q- K \p_t\theta \pa_3  \q+u_j \mathcal{A}_{jk}\pa_k  \q = \pa_t\q-G^{1,1}=-\diverge(\bar\rho u)+G^{1,2}.
\end{equation}
From \eqref{ns_perturb} we find that $\mathcal{Q}$ obeys the equations
\begin{equation}\label{eeqq}
\begin{split}
& \pa_3\mathcal{Q} +\bar{\rho}\pa_3(\diverge   u)=\pa_3 G^{1,2}-\diverge (\pa_3\bar{\rho} u)-\pa_3 \bar{\rho}\pa_3u_3,
\\
 &  \bar\rho \partial_t    u_3   +\bar\rho  \pa_3(h'(\bar{\rho})\q)     -\mu \Delta u_3-(\mu/3+\mu')\pa_3(\diverge u) =G^2_3.
  \end{split}
\end{equation}
We then eliminate $\pa_{33} u_3$ from the equations \eqref{eeqq} to obtain
\begin{equation}\label{density 0}
\begin{split}
& \displaystyle
\frac{ 4\mu/3+\mu'  }{h'(\bar\rho)\bar\rho^2}
\pa_3\left(h'(\bar\rho)\mathcal{Q} \right)+ \pa_3(h'(\bar\rho) \q)
 = \frac{ 4\mu/3+\mu'  }{\bar\rho^2} \pa_3  G^{1,2} +\frac{ 1  }{\bar\rho } G^2_3+\frac{ 4\mu/3+\mu'  }{h'(\bar\rho)\bar\rho^2}
\pa_3h'(\bar\rho) \mathcal{Q}
\\& \qquad   -  \partial_t    u_3 -\frac{ 4\mu/3+\mu'  }{\bar\rho^2}(\diverge (\pa_3\bar{\rho} u)+\pa_3 \bar{\rho}\pa_3u_3)+\frac{ \mu }{\bar\rho }(\pa_{11} u _3+\pa_{22}u_3- \pa_{31} u_1-\pa_{32}  u_2 ) .
\end{split}%
\end{equation}
By the definition \eqref{dt} of $\mathcal{Q}$, we may view \eqref{density 0} as an evolution equation for $\pa_3\q$.  This equation resembles the ODE $\dt f + f= g$, up to some errors, and this ODE displays natural decay structure.   We will extract this kind of structure for the more complicated equation \eqref{density 0}.

\subsection{Estimates  }

We now estimate the energy evolution of $\p_3\q$ at the $2N$ level.
\begin{Proposition}\label{i_rho_evolution 2N}
Fix $0\le  j\le 2N-1$ and   $0\le k\le 4N-2 j-1$.  Then there exist universal constants $\lambda_{k',j}$ for $0 \le k' \le k$ so that
\begin{multline}\label{density es2N}
\frac{d}{dt} \sum_{k'\le k} \lambda_{k',j} \norm{\sqrt{1+\frac{ 4\mu/3+\mu'  }{h'(\bar\rho)\bar\rho^2} }
  \nab_{\ast}^{4N-2j-k'-1} \pa_3^{k'+1}\pa_t^ j(h'(\bar\rho) \q)}_0^2 \\
 + \sum_{k'\le k} \norm{\nab_{\ast}^{4N-2j-k'-1}  \pa_3^{k'+1}\pa_t^j \left( h'(\bar\rho)\q \right)}_0^2
+ \sum_{k'\le k} \norm{\nab_{\ast}^{4N-2j-k'-1}  \pa_3^{k'+1}\pa_t^j  \mathcal{Q}}_0^2
\\
\ls  \norm{\pa_t^{ j+1} u  }_{4N-2 j-1}^2 +\norm{{\na_{\ast}}^{4N-2 j-k-1} \pa_t^ j \mathcal{Q} }_{0}^2
+ \sum_{k'\le k} \norm{{\na_{\ast}}^{4N-2 j-k'} \pa_t^j    u }_{k'+1}^2
\\  + \ks\sqrt{\mathcal{E}_{2N}^\sigma }\mathcal{D}_{2N}^\sigma + \ks \sqrt{ \mathcal{D}_{2N}^\sigma \mathcal{E}_{N+2}^\sigma\f} + \ks \mathcal{E}_{N+2}^\sigma\f.
\end{multline}

 \end{Proposition}
\begin{proof}
We first fix $0\le  j\le 2N-1$ and then take $0 \le k\le 4N-2 j-1$ and $0\le k'\le k$.  Let $\alpha\in \mathbb{N}^{2}$ so that $|\alpha|\le 4N-2 j-1-k' $. Applying $\partial^\alpha\pa_3^{k'}\pa_t^ j$  to $\eqref{density 0}$ and multiplying the resulting by $\partial^\alpha\pa_3^{k'+1}\pa_t^ j(h'(\bar\rho) \q)+\partial^\alpha\pa_3^{k'+1}\pa_t^ j(h'(\bar\rho)\mathcal{Q})$ and then integrating over $\Omega$, we obtain
\begin{equation}\label{dtdensity1}
 I + II + III = IV,
\end{equation}
where
\begin{equation}
 I = \int_\Omega  \partial^\alpha\pa_3^{k'}\pa_t^ j\left(\frac{ 4\mu/3+\mu'  }{h'(\bar\rho)\bar\rho^2}
\pa_3\left(h'(\bar\rho)\mathcal{Q} \right)\right)\partial^\alpha\pa_3^{k'+1}\pa_t^ j(h'(\bar\rho) \q),
\end{equation}
\begin{equation}
 II = \int_\Omega \partial^\alpha\pa_3^{k'}\pa_t^ j\left(\frac{ 4\mu/3+\mu'  }{h'(\bar\rho)\bar\rho^2}
\pa_3\left(h'(\bar\rho)\mathcal{Q} \right)\right)\partial^\alpha\pa_3^{k'+1}\pa_t^ j\left(h'(\bar\rho) \mathcal{Q}\right),
\end{equation}
\begin{equation}
 III = \int_\Omega  \abs{\partial^\alpha\pa_3^{k'+1}\pa_t^ j(h'(\bar\rho) \q)}^2+\int_\Omega  \partial^\alpha\pa_3^{k'+1}\pa_t^ j(h'(\bar\rho) \q)\partial^\alpha\pa_3^{k'+1}\pa_t^ j\left(h'(\bar\rho)\mathcal{Q}\right),
\end{equation}
and
\begin{multline}
 IV = \int_\Omega \left\{\partial^\alpha\pa_3^{k'+1}\pa_t^ j(h'(\bar\rho) \q)+\partial^\alpha\pa_3^{k'+1}\pa_t^ j\left(h'(\bar\rho)\mathcal{Q} \right)\right\}
\\
 \times\partial^\alpha\pa_3^{k'}\pa_t^ j\left\{\frac{ 4\mu/3+\mu'  }{\bar\rho^2} \pa_3  G^{1,2} +\frac{ 1  }{\bar\rho } G^2_3+\frac{ 4\mu/3+\mu'  }{h'(\bar\rho)\bar\rho^2}
\pa_3h'(\bar\rho) \mathcal{Q} -  \partial_t    u_3\right.
\\
 \left. -\frac{ 4\mu/3+\mu'  }{\bar\rho^2}(\diverge (\pa_3\bar{\rho} u)+\pa_3 \bar{\rho}\pa_3u_3)+\frac{ \mu }{\bar\rho }(\pa_{11} u _3+\pa_{22}u_3- \pa_{31} u_1-\pa_{32}  u_2 ) \right\}.
\end{multline}

We will now estimate $I, II, III, IV$.  First, using the Cauchy-Schwarz inequality, we may easily estimate
\begin{multline}\label{rev_1}
 IV \ls \left\{\norm{\partial^\alpha\pa_3^{k'+1}\pa_t^ j(h'(\bar\rho) \q)}_0 +\sum_{k''\le k'} \norm{\partial^\alpha\pa_3^{k''+1}\pa_t^ j\mathcal{Q}}_0\right\}
\\
 \qquad \times\left\{\norm{\pa_t^{ j }  G^{1,2}}_{4N-2 j} +\norm{\pa_t^{ j }  G^2}_{4N-2 j-1}+\sum_{k''\le k'}\norm{\partial^\alpha\pa_3^{k''}\pa_t^ j \mathcal{Q}}_0 \right.
\\
 \qquad\quad  \left.+\norm{\pa_t^{ j+1} u  }_{4N-2 j-1} +\sum_{k''\le k'}  \norm{\partial^\alpha\pa_3^{k''}\pa_t^ j u}_1+\norm{\partial^\alpha\pa_3^{k''}\na_\ast\na\pa_t^ j u}_0  \right\}.
\end{multline}
For the last term in $III$ we recall  the definition of $\mathcal{Q}$ from \eqref{dt} in order to  rewrite
\begin{equation}\label{dens0}
\begin{split}
& \int_\Omega  \partial^\alpha\pa_3^{k'+1}\pa_t^ j(h'(\bar\rho) \q)\partial^\alpha\pa_3^{k'+1}\pa_t^ j\left(h'(\bar\rho)\mathcal{Q}\right)
\\&\quad=\int_\Omega  \partial^\alpha\pa_3^{k'+1}\pa_t^ j(h'(\bar\rho) \q)\partial^\alpha\pa_3^{k'+1}\pa_t^ j
\left(h'(\bar\rho)(\dt\q-G^{1,1})\right)
\\&\quad=\frac{1}{2}\frac{d}{dt} \int_\Omega \abs {\partial^\alpha\pa_3^{k'+1}\pa_t^ j(h'(\bar\rho) \q)}^2
-\int_\Omega  \partial^\alpha\pa_3^{k'+1}\pa_t^ j(h'(\bar\rho) \q)\partial^\alpha\pa_3^{k'+1}\pa_t^ j
\left(h'(\bar\rho) G^{1,1} \right).
\end{split}
 \end{equation}
For $II$ we estimate by expanding with Leibniz:
\begin{equation} \label{dens1}
II \ge \int_\Omega \frac{ 4\mu/3+\mu'  }{h'(\bar\rho)\bar\rho^2}\abs{\partial^\alpha\pa_3^{k'+1}\pa_t^ j \left(h'(\bar\rho)\mathcal{Q} \right)}^2
-C\norm{\partial^\alpha\pa_3^{k'+1}\pa_t^ j \left( h'(\bar\rho)\mathcal{Q} \right)}_0\sum_{k''\le k'}\norm{\partial^\alpha\pa_3^{k''}\pa_t^ j \mathcal{Q}}_0.
 \end{equation}
For $I$, we have
\begin{multline} \label{dens2}
I  \ge \int_\Omega \frac{ 4\mu/3+\mu'  }{h'(\bar\rho)\bar\rho^2}\partial^\alpha\pa_3^{k'+1}\pa_t^ j
\left(h'(\bar\rho)\mathcal{Q} \right) \partial^\alpha\pa_3^{k'+1}\pa_t^ j(h'(\bar\rho) \q)
\\
-C\norm{\partial^\alpha\pa_3^{k'+1}\pa_t^ j \left( h'(\bar\rho)\q \right)}_0\sum_{k''\le k'}\norm{\partial^\alpha\pa_3^{k''}\pa_t^ j \mathcal{Q}}_0
\\  = \frac{1}{2}\frac{d}{dt} \int_\Omega  \frac{ 4\mu/3+\mu'  }{h'(\bar\rho)\bar\rho^2}
\abs{\partial^\alpha\pa_3^{k'+1}\pa_t^ j(h'(\bar\rho) \q)}^2
\\
-\int_\Omega \frac{ 4\mu/3+\mu'  }{h'(\bar\rho)\bar\rho^2} \partial^\alpha\pa_3^{k'+1}\pa_t^ j(h'(\bar\rho) \q)\partial^\alpha\pa_3^{k'+1}\pa_t^ j
\left(h'(\bar\rho) G^{1,1} \right)
\\
-C\norm{\partial^\alpha\pa_3^{k'+1}\pa_t^ j \left( h'(\bar\rho)\q \right)}_0\sum_{k''\le k'}\norm{\partial^\alpha\pa_3^{k''}\pa_t^ j \mathcal{Q}}_0.
 \end{multline}

Combining the estimates \eqref{rev_1}--\eqref{dens2} with \eqref{dtdensity1} and applying Cauchy's inequality in order to absorb the term $\norm{\partial^\alpha\pa_3^{k'+1}\pa_t^ j \left( h'(\bar\rho)\q \right)}_0$  onto the left, we arrive at the inequality
\begin{multline}\label{dens_10}
\frac{d}{dt}\norm{\sqrt{1+\frac{ 4\mu/3+\mu'  }{h'(\bar\rho)\bar\rho^2} }
 \partial^\alpha\pa_3^{k'+1}\pa_t^ j(h'(\bar\rho) \q)}_0^2
 +\hal \norm{\partial^\alpha\pa_3^{k'+1}\pa_t^ j \left( h'(\bar\rho)\q \right)}_0^2 \\
+ \hal \norm{\partial^\alpha\pa_3^{k'+1}\pa_t^ j \left( h'(\bar\rho)\mathcal{Q}\right)}_0^2
\lesssim \sum_{k''\le k'}\norm{\partial^\alpha\pa_3^{k''}\pa_t^ j \mathcal{Q}}_0^2
\\
+\abs{\int_\Omega \left(1+\frac{ 4\mu/3+\mu'  }{h'(\bar\rho)\bar\rho^2}\right) \partial^\alpha\pa_3^{k'+1}\pa_t^ j(h'(\bar\rho) \q)\partial^\alpha\pa_3^{k'+1}\pa_t^ j \left(h'(\bar\rho) G^{1,1} \right)}
\\
+\norm{\pa_t^{ j }  G^{1,2}}_{4N-2 j}^2+\norm{\pa_t^{ j }  G^2}_{4N-2 j-1}^2
+ \norm{\pa_t^{ j+1} u  }_{4N-2 j-1}^2
+ \norm{{\na_{\ast}}^{4N-2 j-k'} \pa_t^ j    u }_{k'+1}^2.
\end{multline}
Owing to the Leibniz rule and the properties of $\bar{\rho}$, we may estimate
\begin{multline}
 \ns{\p_3^{k'+1} \p^\al \dt^j \mathcal{Q}}_0 \ls  \ns{h'(\bar{\rho}) \p_3^{k'+1} \p^\al \dt^j \mathcal{Q}}_0
\ls  \ns{ \p_3^{k'+1} \p^\al \dt^j (h'(\bar{\rho})\mathcal{Q})}_0 + \sum_{k'' \le k'} \ns{\p_3^{k''} \p^\al \dt^j \mathcal{Q}}_0.
\end{multline}

Combining this with \eqref{dens_10} and summing over all $\alpha$ with $\abs{\alpha}\le 4N-2 j-1-k'$, we deduce that there exist universal constants $\beta_{k',j}, \gamma_{k',j}>0$ so that
\begin{multline}\label{Cell}
\frac{d}{dt}\norm{\sqrt{1+\frac{ 4\mu/3+\mu'  }{h'(\bar\rho)\bar\rho^2} }
 \nab_{\ast}^{4N-2j-k'-1}
 \pa_3^{k'+1}\pa_t^ j(h'(\bar\rho) \q)}_0^2
 +\beta_{k',j} \norm{\nab_{\ast}^{4N-2j-k'-1} \pa_3^{k'+1}\pa_t^j \left( h'(\bar\rho)\q \right)}_0^2 \\
+ \beta_{k',j} \norm{\nab_{\ast}^{4N-2j-k'-1} \pa_3^{k'+1}\pa_t^j  \mathcal{Q}}_0^2
\le \gamma_{k',j} \sum_{k''\le k'}\norm{\nab_{\ast}^{4N-2j-k'-1} \pa_3^{k''}\pa_t^ j \mathcal{Q}}_0^2
\\
+\gamma_{k',j}\sum_{\abs{\alpha}\le 4N-2 j-1-k'} \abs{\int_\Omega \left(1+\frac{ 4\mu/3+\mu'  }{h'(\bar\rho)\bar\rho^2}\right) \p^\al \pa_3^{k'+1}\pa_t^ j(h'(\bar\rho) \q)\partial^\alpha\pa_3^{k'+1}\pa_t^ j \left(h'(\bar\rho) G^{1,1} \right)}
\\
+\gamma_{k',j} \left(\norm{\pa_t^{ j }  G^{1,2}}_{4N-2 j}^2+\norm{\pa_t^{ j }  G^2}_{4N-2 j-1}^2
+ \norm{\pa_t^{ j+1} u  }_{4N-2 j-1}^2
+ \norm{{\na_{\ast}}^{4N-2 j-k'} \pa_t^ j    u }_{k'+1}^2\right)
\end{multline}
for every $0 \le k' \le k$.  We may then use Lemma \ref{est_alg}  to deduce that there exist constants $\lambda_{k',j,\alpha}>0$ (depending on $\beta_{k',j,\alpha}$ and $\gamma_{k',j,\alpha}$ as in \eqref{es_al_03}) and
\begin{equation}
\bar{\lambda}_{k,j} := \sum_{k'\le k} \lambda_{k',j} (\gamma_{k',j} + 1)
\end{equation}
so that
\begin{multline} \label{rhoes1}
\frac{d}{dt} \sum_{k'\le k} \lambda_{k',j} \norm{\sqrt{1+\frac{ 4\mu/3+\mu'  }{h'(\bar\rho)\bar\rho^2} }
 \nab_{\ast}^{4N-2j-k'-1} \pa_3^{k'+1}\pa_t^ j(h'(\bar\rho) \q)}_0^2 \\
 + \sum_{k'\le k} \norm{\nab_{\ast}^{4N-2j-k'-1} \pa_3^{k'+1}\pa_t^j \left( h'(\bar\rho)\q \right)}_0^2
+ \sum_{k'\le k} \norm{\nab_{\ast}^{4N-2j-k'-1}  \pa_3^{k'+1}\pa_t^j  \mathcal{Q}}_0^2 \\
\le
\bar{\lambda}_{k,j}\sum_{\substack{\abs{\alpha}\le 4N-2 j-1-k' \\ k'\le k} }  \abs{\int_\Omega \left(1+\frac{ 4\mu/3+\mu'  }{h'(\bar\rho)\bar\rho^2}\right) \partial^\alpha\pa_3^{k'+1}\pa_t^ j(h'(\bar\rho) \q)\partial^\alpha\pa_3^{k'+1}\pa_t^ j \left(h'(\bar\rho) G^{1,1} \right)}
\\
+\bar{\lambda}_{k,j} \left(\norm{\pa_t^{ j }  G^{1,2}}_{4N-2 j}^2+\norm{\pa_t^{ j }  G^2}_{4N-2 j-1}^2
+ \norm{\pa_t^{ j+1} u  }_{4N-2 j-1}^2 \right) \\
+\bar{\lambda}_{k,j} \left( \norm{\nab_{\ast}^{4N-2j-k-1} \pa_t^j \mathcal{Q}}_0^2
+ \sum_{k'\le k} \norm{{\na_{\ast}}^{4N-2 j-k'} \pa_t^j    u }_{k'+1}^2     \right).
\end{multline}

Finally, we will estimate the nonlinear terms in the right hand side of \eqref{rhoes1}. We use the estimates \eqref{p_G_e_00} of Lemma \ref{p_G2N_estimates} to estimate, for $0\le  j\le 2N-1$,
\begin{equation}\label{rhoes2}
\norm{\pa_t^{ j }  G^{1,2}}_{4N-2 j}^2+\norm{\pa_t^{ j }  G^2}_{4N-2 j-1}^2\lesssim \ks {\mathcal{E}_{2N}^\sigma }\mathcal{D}_{2N}^\sigma+  \mathcal{E}_{N+2}^\sigma\f .
\end{equation}
Then we use Lemma \ref{G11_weighted} to bound
\begin{multline}\label{rhoes25}
\bar{\lambda}_{k,j}\sum_{\substack{\abs{\alpha}\le 4N-2 j-1-k' \\ k'\le k} }  \abs{\int_\Omega \left(1+\frac{ 4\mu/3+\mu'  }{h'(\bar\rho)\bar\rho^2}\right) \partial^\alpha\pa_3^{k'+1}\pa_t^ j(h'(\bar\rho) \q)\partial^\alpha\pa_3^{k'+1}\pa_t^ j \left(h'(\bar\rho) G^{1,1} \right)}
\\
    \ls \ks \sqrt{ \mathcal{D}_{2N}^\sigma }\sqrt{ {\mathcal{E}_{2N}^\sigma }\mathcal{D}_{2N}^\sigma+  \mathcal{E}_{N+2}^\sigma\f}
\end{multline}
Plugging the nonlinear estimates \eqref{rhoes2} and \eqref{rhoes25}  into \eqref{rhoes1} then yields \eqref{density es2N}.
 \end{proof}

We then record a similar result at the $N+2$ level.
\begin{Proposition}\label{i_rho_evolution N+2}
Fix $0\le  j\le N+1$ and $0\le k\le 2(N+2)-2j-1$.  Then there exist universal constants $\lambda_{k',j}$ for $0 \le k' \le k$ so that
\begin{multline}\label{density esN+2}
\frac{d}{dt} \sum_{k'\le k} \lambda_{k',j} \norm{\sqrt{1+\frac{ 4\mu/3+\mu'  }{h'(\bar\rho)\bar\rho^2} }
 \nab_{\ast}^{ 2(N+2)-2 j-1-k'} \pa_3^{k'+1}\pa_t^ j(h'(\bar\rho) \q)}_0^2 \\
 + \sum_{k'\le k} \norm{\nab_{\ast}^{ 2(N+2)-2 j-1-k'} \pa_3^{k'+1}\pa_t^j \left( h'(\bar\rho)\q \right)}_0^2
+ \sum_{k'\le k} \norm{\nab_{\ast}^{ 2(N+2)-2 j-1-k'} \pa_3^{k'+1}\pa_t^j  \mathcal{Q}}_0^2
\\
\ls
\norm{\pa_t^{ j+1} u  }_{4N-2 j-1}^2
+\norm{\nab_{\ast}^{ 2(N+2)-2 j-1-k} \pa_t^ j \mathcal{Q} }_{0}^2
\\
+\sum_{k'\le k}\norm{ \nab_{\ast}^{ 2(N+2)-2 j-k'} \pa_t^ j    u }_{k'+1}^2
 + \ks\sqrt{\mathcal{E}_{2N}^\sigma }\mathcal{D}_{N+2}^\sigma.
\end{multline}
\end{Proposition}
\begin{proof}
The proof proceeds along the same lines as that of Proposition \ref{i_rho_evolution 2N} except that we instead use the estimates \eqref{p_G_e_h_00} of Lemma \ref{p_GN+2_estimates}.
\end{proof}

\begin{Remark}
Propositions \ref{i_rho_evolution 2N} and \ref{i_rho_evolution N+2}  provide two important bits of control:  energy estimates of $\pa_3q$ and  dissipation estimates for $\pa_3\mathcal{Q}$.  These are crucial for improving the horizontal energy and dissipation estimates derived in the previous section into the full ones in later sections, respectively.
\end{Remark}

\section{Combined energy evolution estimates}\label{sec_combo}

Now we combine our previous estimates with the elliptic regularity theory of a particular Stokes problem in order to derive an intermediate energy-dissipation estimate.  We again assume throughout this section that the solutions obey the estimate $\gs(T) \le \delta$, where $\delta \in (0,1)$ is given in  Lemma \ref{eta_small}.

\subsection{The Stokes problem}

We first derive  elliptic estimates.  We deduce from \eqref{ns_perturb} that
\begin{equation}
\begin{split}
 &\diverge ( \bar{\rho}   u)=G^{1,2}-\mathcal{Q},
 \\
 & -\frac{\mu}{\bar{\rho}}\Delta u- \frac{\mu/3+\mu'}{\bar{\rho}} \na \diverge u+ \nabla   \left(h'(\bar{\rho})\q\right) =\frac{1}{\bar{\rho}}G^2-   \partial_t    u.
  \end{split}
\end{equation}
Direct calculations give the form of the Stokes problem we shall use:
\begin{equation}\label{stoke problem}
\begin{cases}
 \displaystyle-\mu\Delta\left (\frac{u}{\bar{\rho}} \right)+ \nabla  \left(h'(\bar{\rho})\q\right)
 =\frac{1}{\bar{\rho}}G^2-   \partial_t    u-\mu\left(2\pa_3\left (\frac{1}{\bar{\rho}} \right)\pa_3u+ \pa_{33}\left (\frac{1}{\bar{\rho}} \right) u\right)
 \\\qquad\qquad\qquad\qquad\qquad\quad\ \displaystyle +\frac{\mu/3+\mu'}{\bar{\rho}}\na \left (\displaystyle\frac{1}{\bar{\rho}} \left(G^{1,2}-\displaystyle\mathcal{Q}-\pa_3\bar{\rho}u_3\right)\right)&\text{in }\Omega_\pm\\\displaystyle
 \diverge \left (\frac{u}{\bar{\rho}} \right)=\frac{1}{\bar{\rho}^2}\left(G^{1,2}-\mathcal{Q}-2\pa_3\bar{\rho}u_3\right)&\text{in }\Omega_\pm
\\u=u&\text{on }\pa\Omega_\pm.
  \end{cases}
\end{equation}

We now prove the Stokes estimates at the $2N$ level.

\begin{lemma}\label{lemmau2N}
Fix $0\le  j\le 2N-1$. Then for any $1\le k\le 4N-2 j$,
\begin{equation}\label{u es2N}
\begin{split}
 &\norm{ {\na_{\ast}}^{4N-2 j-k}\pa_t^ j  u }_{k+1}^2 + \norm{\na {\na_{\ast}}^{4N-2 j-k}\pa_t^ j\left(h'(\bar{\rho})\q\right) }_{ k-1}^2
\\& \quad\lesssim\norm{\pa_t^{ j+1} u  }_{4N-2 j-1}^2+ \norm{{\na_{\ast}}^{4N-2 j-k} \pa_t^{ j } \mathcal{Q} }_{k}^2+ \norm{\bar{\na}_{\ast}^{4N}   u }_1^2
+\ks  \mathcal{E}_{2N}^\sigma  \mathcal{D}_{2N}^\sigma+  \mathcal{E}_{N+2}^\sigma\f.
 \end{split}
\end{equation}
\end{lemma}
\begin{proof}
We first fix $0\le  j\le 2N-1$ and then take $1 \le k\le 4N-2 j$. Let $\alpha\in \mathbb{N}^{2}$ so that $|\alpha|\le 4N-2 j-k$. Applying $\partial^\alpha\pa_t^ j$  to the equations $\eqref{stoke problem}$ in $\Omega_\pm$ respectively, and then applying the elliptic estimates of Lemma \ref{stokes reg} with $r=k'+1\ge2$ for any $1\le k'\le k$, by the trace theory, we obtain
\begin{equation}\label{u claim es}
\begin{split}
&\norm{ \pa^\al\pa_t^ j  u }_{k'+1}^2+ \norm{\na \pa^\al\pa_t^ j   \left(h'(\bar{\rho})\q\right)  }_{k'-1}^2
 \lesssim \norm{ \pa^\al\pa_t^ j \left (\frac{u}{\bar{\rho}} \right) }_{k'+1}^2+ \norm{\na \pa^\al\pa_t^ j  \left(h'(\bar{\rho})\q\right) }_{k'-1}^2
\\
&\quad\lesssim \norm{\pa^\al\pa_t^{ j} G^2  }_{k'-1}^2+\norm{\pa^\al\pa_t^{ j+1} u  }_{k'-1}^2+\norm{ \pa^\al\pa_t^ j  u }_{k'}^2
\\
&\qquad+ \norm{\pa^\al\pa_t^{ j} G^{1,2}  }_{k'}^2+\norm{\pa^\al \pa_t^{ j } \mathcal{Q} }_{k'}^2+ \norm{\pa^\al\pa_t^ j u }_{H^{k'+1/2}(\Sigma)}^2
\\
&\quad\lesssim \norm{\pa^\al\pa_t^{ j} G^{1,2}  }_{k}^2+\norm{\pa^\al\pa_t^{ j} G^2  }_{k-1}^2+\norm{\pa^\al\pa_t^{ j+1} u  }_{k-1}^2 +\norm{\pa^\al \pa_t^{ j } \mathcal{Q} }_{k}^2+\norm{ \pa^\al\pa_t^ j  u }_{k'}^2
\\
& \qquad+ \norm{{\na_{\ast}}^{k'}\pa^\al\pa_t^ j u }_{H^{1/2}(\Sigma)}^2
\\
&\quad\lesssim \norm{\pa_t^{ j} G^{1,2}  }_{4N-2 j}^2  +\norm{ \pa_t^{ j} G^2  }_{4N-2 j-1}^2+\norm{\pa_t^{ j+1} u  }_{4N-2 j-1}^2+ \norm{{\na_{\ast}}^{4N-2 j-k} \pa_t^{ j } \mathcal{Q} }_{k}^2
\\
&\qquad+ \norm{\bar{\na}_{\ast}^{\,\,4N}   u }_1^2 +\norm{ \pa^\al\pa_t^ j  u }_{k'}^2.
\end{split}
\end{equation}

A simple induction based on the above yields that
\begin{equation}\label{u claim}
\begin{split}
\norm{\pa^\al\pa_t^ j  u }_{k+1}^2&+ \norm{\na \pa^\al\pa_t^ j   \q }_{ k-1}^2
    \lesssim \norm{\pa_t^{ j} G^{1,2}  }_{4N-2 j}^2
 +\norm{ \pa_t^{ j} G^2  }_{4N-2 j-1}^2
\\&\quad+\norm{\pa_t^{ j+1} u  }_{4N-2 j-1}^2+ \norm{{\na_{\ast}}^{4N-2 j-k} \pa_t^{ j } \mathcal{Q} }_{k}^2 + \norm{\bar{\na}_{\ast}^{4N}   u }_1^2.
\end{split}
\end{equation}

Finally, we use the estimates \eqref{p_G_e_00} of Lemma \ref{p_G2N_estimates} to have
\begin{equation}\label{u claim es2}
\norm{\pa_t^{ j} G^{1,2}  }_{4N-2 j}^2
 +\norm{ \pa_t^{ j} G^2  }_{4N-2 j-1}^2\lesssim  \ks {\mathcal{E}_{2N}^\sigma }\mathcal{D}_{2N}^\sigma+  \mathcal{E}_{N+2}^\sigma\f
 \end{equation}
We then  sum over such $|\alpha|\le 4N-2 j-k$ to conclude \eqref{u es2N}.
\end{proof}

We then record a similar result at the $N+2$ level.
\begin{lemma}\label{lemmauN+2}
Fix $0\le  j\le N+1$. Then for any $1\le k\le 2(N+2)-2 j$,
\begin{equation}\label{u esN+2}
\begin{split}
 &\norm{ {\na_{\ast}}^{2(N+2)-2 j-k}\pa_t^ j  u }_{2(N+2)+1}^2 + \norm{\na {\na_{\ast}}^{2(N+2)-2 j-k}\pa_t^ j  \q }_{ k-1}^2
\\& \quad\lesssim\norm{\pa_t^{ j+1} u  }_{2(N+2)-2 j-1}^2+ \norm{{\na_{\ast}}^{2(N+2)-2 j-k} \pa_t^{ j } \mathcal{Q} }_{k}^2 + \norm{\bar{\na}_{\ast}^{\,\,\,2(N+2)}   u }_1^2
+ \ks  \mathcal{E}_{2N}^\sigma  \mathcal{D}_{N+2}^\sigma.
 \end{split}
\end{equation}
\end{lemma}
\begin{proof}
The proof proceeds similarly as Proposition \ref{lemmau2N} by using instead the estimates
\eqref{p_G_e_h_00} of Lemma \ref{p_GN+2_estimates}.
\end{proof}

\subsection{Synthesis}

We will now chain the energy evolution estimates of Section \ref{sec_horiz} with the $\p_3 q$ estimate of Section \ref{sec_aux} and the estimates of Lemmas \ref{lemmau2N}--\ref{lemmauN+2}. The full dissipation estimates of $u$ will be obtained, and also some estimates of $\q$ will be improved along the way. To do so, we first introduce some notation. For $n=2N$ or $n=N+2$ we write
\begin{multline}\label{E_frak}
\mathfrak{E}_{n}^\sigma =
\int_\Omega \bar{\rho}\abs{\bar{\nab}_{\ast}^{2n} u}^2 +h'(\bar{\rho})\abs{ \bar{\nab}_{\ast}^{2n} \q}^2
+\int_{\Sigma_+} \rho_1  g\abs{\bar{\nab}_{\ast}^{2n} \eta_+}^2+\sigma_+\abs{\nab_\ast \bar{\nab}_{\ast}^{2n} \eta_+}^2
\\
+ \int_{\Sigma_-}- \rj g\abs{\bar{\nab}_{\ast}^{2n} \eta_-}^2+ \sigma_-\abs{\nab_\ast \bar{\nab}_{\ast}^{2n}\eta_-}^2
\\
+ \int_\Omega (\bar{\rho}(J-1)K+ \q+\p_3\bar\rho\theta) J\abs{\partial_t^{n} u }^2  + h'(\bar{\rho})(J-1) \abs{\dt^{n}\q}^2
\end{multline}
and
\begin{equation}\label{D_frak}
\mathfrak{D}_{n}^\sigma =  \int_{\Omega}\frac{\mu}{2} \abs{\sgz \bar{\nab}_{\ast}^{2n} u}^2+\mu' \abs{\diverge \bar{\nab}_{\ast}^{2n} u}^2
\end{equation}
for the various terms appearing in Propositions  \ref{i_temporal_evolution 2N} and \ref{i_spatial_evolution 2N}.  Similarly, for $n=2N$ or $n=N+2$ and integers $0 \le j \le n-1$ and $0 \le k \le 2n-2j -1$ we write
\begin{equation}\label{A_frak}
 \mathfrak{A}_{n}^{j,k} := \sum_{k'=0}^{ k} \lambda_{k',j} \norm{\sqrt{1+\frac{ 4\mu/3+\mu'  }{h'(\bar\rho)\bar\rho^2} }
  \nab_{\ast}^{2n-2j-k'-1} \pa_3^{k'+1}\pa_t^ j(h'(\bar\rho) \q)}_0^2,
\end{equation}
with the constants $\lambda_{k',j}$ the same as in Proposition \ref{i_rho_evolution 2N} in the case $n=2N$ and as in Proposition \ref{i_rho_evolution N+2} in the case $n=N+2$.  We also write
\begin{multline}\label{B_frak}
 \mathfrak{B}_{n}^{j,k} :=   \sum_{k'=0}^{k} \norm{\nab_{\ast}^{2n-2j-k'-1}  \pa_3^{k'+1}\pa_t^j \left( h'(\bar\rho)\q \right)}_0^2
+ \sum_{k'=1}^k \ns{ \nab \nab_\ast^{2n-2j-k'} \dt^j(h'(\bar{\rho}) \q)}_{k'-1}
\\
+ \norm{ \bar{\na}_{\ast}^{2n}  \mathcal{Q}}_0^2  + \sum_{k'=0}^{k} \norm{\nab_{\ast}^{2n-2j-k'-1}  \pa_3^{k'+1}\pa_t^j  \mathcal{Q}}_0^2,
\end{multline}
where $\mathcal{Q}$ is as defined in \eqref{dt}. In addition, we introduce the following intermediate energies:
\begin{multline}\label{Enn}
\bar{\mathcal{E}}_{n}^\sigma := \norm{ {\bar\na}_{\ast}^{2n} u}_0^2  + \sum_{j=0}^{n} \norm{ \pa_t^j \q}_{2n-2j}^2
+\mathcal{H}(-\rj)\left[ \sum_{j=0}^{n} \ns{\dt^j \eta}_{2n-2j} + \min\{1,\sigma\} \sum_{j=0}^n \ns{\dt^j \eta}_{2n-2j+1}  \right]
\\
+\mathcal{H}(\rj)  \min\{1,\sigma_+,\sigma_- - \sigma_c\} \sum_{j=0}^n \ns{\dt^j \eta}_{2n-2j+1}
\end{multline}
\begin{equation}\label{Dnn}
\bar{\mathcal{D}}_{n}:= \sum_{j=0}^{n} \ns{ \dt^j u}_{2n-2j+1} + \sum_{j=0}^{n-1} \ns{\nab \dt^j (h'(\bar{\rho}) \q)}_{2n-2j-1}.
\end{equation}
The rest of the section is devoted to the derivation of the energy bounds for $\bar{\mathcal{E}}_{n}^\sigma$ and $\bar{\mathcal{D}}_{n}$ based on the energy evolutions for $\mathfrak{E}_{n}^\sigma, \mathfrak{D}_{n}^\sigma, \mathfrak{A}_{n}^{j,k},  \mathfrak{B}_{n}^{j,k}$.
First we present the main energy evolution result at the $2N$ level with the improved dissipation energy $\bar{\mathcal{D}}_{2N}$.

\begin{Proposition}\label{boostrap2N}
Let $\Ef, \Af, \Bf$ be defined as above and let $\bar{\mathcal{D}}_{2N}$ be as defined by \eqref{Dnn}.  There exist universal constants $\gamma_{2N}>0$ and $\beta_{2N;j,k}>0$ for $j=0,\dotsc,2N-1$, $k=0,\dotsc,4N-2j-1$ so that
\begin{multline}\label{ebar2N}
\frac{d}{dt} \left( \Ef + \sum_{j=0}^{2N-1} \sum_{k=0}^{4N-2j-1} \beta_{2N;j,k}  \Af \right) + \gamma_{2N} \bar{\mathcal{D}}_{2N}
\\
\ls
\ks \sqrt{\mathcal{E}_{2N}^\sigma }\mathcal{D}_{2N}^\sigma + \ks \sqrt{ \mathcal{D}_{2N}^\sigma \mathcal{E}_{N+2}^\sigma\f} + \ks \mathcal{E}_{N+2}^\sigma\f,
\end{multline}

\end{Proposition}
\begin{proof}

First, we sum the result of Proposition  \ref{i_temporal_evolution 2N}, with $j=2N$, with the result of Proposition
 \ref{i_spatial_evolution 2N}, for all $\alpha \in \mathbb{N}^{1+2}$ with $\alpha_0 \le 2N-1$ and $\abs{\alpha} \le 4N$; this yields the estimate
\begin{equation}\label{boo_1}
 \frac{d}{dt} \Ef + \Df \ls \ks \sqrt{\mathcal{E}_{2N}^\sigma }\mathcal{D}_{2N}^\sigma + \ks \sqrt{ \mathcal{D}_{2N}^\sigma \mathcal{E}_{N+2}^\sigma\f},
\end{equation}
where $\Ef$ and $\Df$ are as defined in \eqref{E_frak} and \eqref{D_frak}.  Note that, owing to the Korn inequality of Proposition \ref{layer_korn}, we may bound
\begin{equation}\label{boo_2}
 \ns{\nab_\ast^{4N} u}_1 \ls \Df.
\end{equation}

Next, we recall the notation $\mathcal{Q}$ in \eqref{dt}.   We may use the estimates \eqref{p_G_e_00} of Lemma \ref{p_G2N_estimates} to obtain the bound
\begin{equation}\label{drho}
\norm{ \bar{\na}_{\ast}^{4N}  \mathcal{Q}}_0^2 \ls \norm{\bar{\na}_{\ast}^{4N}\diverge(\bar\rho u)}_0^2 +
\norm{\bar{\na}_{\ast}^{4N} G^{1,2}}_0^2
\lesssim \norm{\bar{\na}_{\ast}^{4N} u }_1^2 + \ks   \mathcal{E}_{2N}^\sigma \mathcal{D}_{2N}^\sigma+  \mathcal{E}_{N+2}^\sigma\f.
\end{equation}
Then for $0\le  j\le 2N-1$ and  $0\le k\le 4N-2 j-1$, we may combine the results of Proposition \ref{i_rho_evolution 2N} and Lemma \ref{lemmau2N} (summed over  $1 \le k' \le k$) with \eqref{drho} to see that
\begin{multline}\label{boo_3}
 \frac{d}{dt} \Af + \Bf  \ls
 \norm{\pa_t^{ j+1} u  }_{4N-2 j-1}^2
+ \sum_{k'=1}^{k}\norm{{\na_{\ast}}^{4N-2 j-k'} \pa_t^{ j } \mathcal{Q} }_{k'}^2
  + \norm{\bar{\na}_{\ast}^{4N} u }_1^2 \\
 +\ks \left[ \sqrt{\mathcal{E}_{2N}^\sigma }\mathcal{D}_{2N}^\sigma+\sqrt{ \mathcal{D}_{2N}^\sigma \mathcal{E}_{N+2}^\sigma\f}+\mathcal{E}_{N+2}^\sigma\f \right],
\end{multline}
where we have written $\Af$ and $\Bf$ as in \eqref{A_frak} and \eqref{B_frak} and employed  Lemma \ref{lemmau2N} to control the term $\norm{{\na_{\ast}}^{4N-2 j-k'} \pa_t^j    u }_{k'+1}^2$ in the right hand side of \eqref{density es2N}. Notice that if we
write
\begin{equation}\label{boo_5de}
 \Hf := \ns{ \bar{\na}_{\ast}^{4N}  \mathcal{Q}}_0 + \ns{\nab_\ast^{4N-2j-k-1} \dt^j \mathcal{Q}}_{k+1},
\end{equation}
then we particularly have
\begin{equation}\label{boo_7}
\Hf \ls  \norm{ \bar{\na}_{\ast}^{4N}  \mathcal{Q}}_0^2  + \sum_{k'=0}^{k} \norm{\nab_{\ast}^{4N-2j-k'-1}  \pa_3^{k'+1}\pa_t^j  \mathcal{Q}}_0^2\le\Bf.
\end{equation}

Let $0\le  j\le 2N-1$ and  $0\le k\le 4N-2 j-1$.  We  sum the estimates \eqref{boo_1} and \eqref{boo_3}; employing \eqref{boo_7} in the resulting estimate, we deduce that there exists a universal constant $C_{1}>0$ such that
\begin{multline}\label{boo_4}
 \frac{d}{dt}( \Ef + \Af) + \Df +  C_{1} \Hf \ls
\sum_{k'=0}^{k}\norm{{\na_{\ast}}^{4N-2 j-k'} \pa_t^{ j } \mathcal{Q} }_{k'}^2
+ \norm{\pa_t^{ j+1} u  }_{4N-2 j-1}^2
\\
+\norm{\bar{\na}_{\ast}^{4N} u }_1^2
+\ks \left[ \sqrt{\mathcal{E}_{2N}^\sigma }\mathcal{D}_{2N}^\sigma+\sqrt{ \mathcal{D}_{2N}^\sigma \mathcal{E}_{N+2}^\sigma\f}+\mathcal{E}_{N+2}^\sigma\f \right]
\end{multline}
for all such $j,k$.

Now, for fixed $0\le  j\le 2N-1$, owing to the definition \eqref{boo_5de} of $\Hf$, we may view \eqref{boo_4} as a sequence of estimates indexed by $0\le k\le 4N-2 j-1$, of the form considered in Lemma \ref{est_alg}.  Applying the lemma, we deduce that there exist universal constants $C_{2;j,k}>0$ so that
\begin{multline}\label{boo_9}
 \frac{d}{dt}\sum_{k=0}^{4N-2j-1} C_{2;j,k} ( \Ef + \Af) +\sum_{k=0}^{4N-2j-1} \left(\Df + C_{1} \Hf \right)
\ls
\norm{\pa_t^{ j+1} u  }_{4N-2 j-1}^2
\\+\norm{{\na_{\ast}}^{4N-2 j} \pa_t^{ j } \mathcal{Q} }_{0}^2
+\norm{\bar{\na}_{\ast}^{4N} u }_1^2
+\ks \left[ \sqrt{\mathcal{E}_{2N}^\sigma }\mathcal{D}_{2N}^\sigma+\sqrt{ \mathcal{D}_{2N}^\sigma \mathcal{E}_{N+2}^\sigma\f}+\mathcal{E}_{N+2}^\sigma\f \right]
\\
\ls \norm{\pa_t^{ j+1} u  }_{4N-2 j-1}^2
+  \Df
+ \ks \left[ \sqrt{\mathcal{E}_{2N}^\sigma }\mathcal{D}_{2N}^\sigma+\sqrt{ \mathcal{D}_{2N}^\sigma \mathcal{E}_{N+2}^\sigma\f}+\mathcal{E}_{N+2}^\sigma\f \right],
\end{multline}
where the second estimate follows from \eqref{drho} and \eqref{boo_2}.

It is easy to see, using the definitions of $\Df$ and $\Hf$, that
\begin{equation}
 \ns{\bar{\nab}_\ast^{4N} u}_1
+ \ns{ \bar{\na}_{\ast}^{4N}  \mathcal{Q}}_0 + \ns{\dt^j \mathcal{Q}}_{4N-2j}
\ls
 \sum_{k=0}^{4N-2j-1} \left( \Df + C_{1} \Hf \right).
\end{equation}
Adding $\ns{\dt^{j+1} u}_{4N-2j-1} + \ks \mathcal{E}_{2N}^\sigma \mathcal{D}_{2N}^\sigma + \mathcal{E}_{N+2}^\sigma\f$ to both sides, and using Lemma \ref{lemmau2N} (with $k = 4N-2j$), we then have that
\begin{multline}
 \ns{ \dt^j u}_{4N-2j+1}  + \ns{\nab \dt^j (h'(\bar{\rho}) \q)}_{4N-2j-1}
\\
\ls
 \sum_{k=0}^{4N-2j-1} \left( \Df + C_{1} \Hf \right) + \ns{\dt^{j+1} u}_{4N-2j-1} + \ks \mathcal{E}_{2N}^\sigma \mathcal{D}_{2N}^\sigma + \mathcal{E}_{N+2}^\sigma\f.
\end{multline}
Hence there is a universal constant $C_{3;j}>0$ so that
\begin{multline}\label{boo_11}
\frac{d}{dt}\sum_{k=0}^{4N-2j-1} C_{2;j,k} ( \Ef + \Af)
+ C_{3;j}\left( \ns{ \dt^j u}_{4N-2j+1}  + \ns{\nab \dt^j (h'(\bar{\rho}) \q)}_{4N-2j-1}  \right)
\\
\ls \norm{\pa_t^{ j+1} u  }_{4N-2 j-1}^2 + \Df
+\ks \left[ \sqrt{\mathcal{E}_{2N}^\sigma }\mathcal{D}_{2N}^\sigma+\sqrt{ \mathcal{D}_{2N}^\sigma \mathcal{E}_{N+2}^\sigma\f}+\mathcal{E}_{N+2}^\sigma\f \right]
\end{multline}
for all $0 \le j \le 2N-1$.

Counting backward from $2N-1$ to $0$, we may view \eqref{boo_11} as a sequence of inequalities of the form considered in Lemma \ref{est_alg}.  Applying the lemma, we find that there exist universal constants $C_{4;j}>0$ so that
\begin{multline}\label{boo_12}
\frac{d}{dt} \sum_{j=0}^{2N-1} C_{4;j} \sum_{k=0}^{4N-2j-1} C_{2;j,k} ( \Ef + \Af)
+\sum_{j=0}^{2N-1} C_{3;j}\left( \ns{ \dt^j u}_{4N-2j+1}  + \ns{\nab \dt^j (h'(\bar{\rho}) \q)}_{4N-2j-1}  \right)
\\
\ls \norm{\pa_t^{2N} u  }_{1}^2 + \Df
+\ks \left[ \sqrt{\mathcal{E}_{2N}^\sigma }\mathcal{D}_{2N}^\sigma+\sqrt{ \mathcal{D}_{2N}^\sigma \mathcal{E}_{N+2}^\sigma\f}+\mathcal{E}_{N+2}^\sigma\f \right] \\
\ls \Df
+\ks \left[ \sqrt{\mathcal{E}_{2N}^\sigma }\mathcal{D}_{2N}^\sigma+\sqrt{ \mathcal{D}_{2N}^\sigma \mathcal{E}_{N+2}^\sigma\f}+\mathcal{E}_{N+2}^\sigma\f \right],
\end{multline}
where the last inequality follows from \eqref{boo_2}.

Let $C_5>0$ denote the universal constant appearing on the right side of the last inequality in \eqref{boo_12}.  We multiply \eqref{boo_1} by $2C_5$ and add the resulting inequality to \eqref{boo_12}.  This results in the bound
\begin{multline}\label{boo_13}
\frac{d}{dt} \left( C_6 \Ef + \sum_{j=0}^{2N-1} \sum_{k=0}^{4N-2j-1} C_{4;j} C_{2;j,k}   \Af \right) \\
+C_5 \Df + \sum_{j=0}^{2N-1} C_{3;j}\left( \ns{ \dt^j u}_{4N-2j+1}  + \ns{\nab \dt^j (h'(\bar{\rho}) \q)}_{4N-2j-1}  \right)
\\
\ls
\ks \left[ \sqrt{\mathcal{E}_{2N}^\sigma }\mathcal{D}_{2N}^\sigma+\sqrt{ \mathcal{D}_{2N}^\sigma \mathcal{E}_{N+2}^\sigma\f}+\mathcal{E}_{N+2}^\sigma\f \right],
\end{multline}
where
\begin{equation}
 C_6 = 2 C_5 + \sum_{j=0}^{2N-1}  \sum_{k=0}^{4N-2j-1}C_{4;j} C_{2;j,k}.
\end{equation}
Finally, we estimate
\begin{multline}
\sum_{j=0}^{2N} \ns{ \dt^j u}_{4N-2j+1} + \sum_{j=0}^{2N-1} \ns{\nab \dt^j (h'(\bar{\rho}) \q)}_{4N-2j-1}
\\
\ls C_5 \Df + \sum_{j=0}^{2N-1} C_{3;j}\left( \ns{ \dt^j u}_{4N-2j+1}  + \ns{\nab \dt^j (h'(\bar{\rho}) \q)}_{4N-2j-1}  \right).
\end{multline}
This and \eqref{boo_13} imply \eqref{ebar2N} upon dividing by $C_6$ and renaming the universal constants.
\end{proof}

Next we  record a similar result at the $N+2$ level.

\begin{Proposition}\label{boostrapN+2}
There exist universal constants $\gamma_{N+2}>0$ and $\beta_{N+2;j,k}>0$ for $j=0,\dotsc,N+1$, $k=0,\dotsc,2(N+2)-2j-1$ so that
\begin{equation}\label{ebarN+2}
\frac{d}{dt} \left( \mathfrak{E}_{N+2}^\sigma + \sum_{j=0}^{N+1} \sum_{k=0}^{2(N+2)-2j-1} \beta_{N+2;j,k}  \mathfrak{A}_{N+2}^{j,k} \right) + \gamma_{N+2} \bar{\mathcal{D}}_{N+2}
\ls
\ks \sqrt{\mathcal{E}_{2N}^\sigma }\mathcal{D}_{N+2}^\sigma.
\end{equation}
\end{Proposition}

\begin{proof}
The proof proceeds as Proposition \ref{boostrap2N} using instead Propositions \ref{i_temporal_evolution N+2}, \ref{i_spatial_evolution N+2}, \ref{i_rho_evolution N+2}, Lemma \ref{lemmauN+2} and the estimates \eqref{p_G_e_h_00} of Lemma \ref{p_GN+2_estimates}.
\end{proof}


Next, we show  that, up to some error terms, ${\mathfrak{E}}_{n}^\sigma$ is comparable to $ \bar{\mathcal{E}}_{n}^\sigma$  for both $n=2N$ and $n=N+2$.



\begin{lemma}\label{en_equiv}
Let $\bar{\mathcal{E}}_{n}^\sigma$ be defined by \eqref{Enn}, $\mathfrak{E}_{n}^\sigma$ be defined by \eqref{E_frak}, and $\mathfrak{A}_n^{j,k}$ be defined by \eqref{A_frak}.  Further let $\beta_{2N;j,k}$ and $\beta_{N+2;j,k}$ be the constants appearing in Propositions \ref{boostrap2N} and \ref{boostrapN+2}, respectively, for appropriate integers $j,k$.  Then
\begin{multline}\label{ee_01}
\bar{\mathcal{E}}_{2N}^\sigma - \sqrt{\mathcal{E}_{2N}^\sigma} \min\{\mathcal{E}_{2N}^\sigma, \mathcal{D}_{2N}^\sigma \}  \ls  \Ef + \sum_{j=0}^{2N-1} \sum_{k=0}^{4N-2j-1} \beta_{2N;j,k}  \Af \\
\ls \ks \bar{\mathcal{E}}_{2N}^\sigma + \sqrt{\mathcal{E}_{2N}^\sigma} \min\{\mathcal{E}_{2N}^\sigma, \mathcal{D}_{2N}^\sigma \}
\end{multline}
and
\begin{multline}\label{ee_02}
 \bar{\mathcal{E}}_{N+2}^\sigma - \sqrt{\mathcal{E}_{2N}^\sigma} \min\{\mathcal{E}_{N+2}^\sigma, \mathcal{D}_{N+2}^\sigma \}  \ls   \mathfrak{E}_{N+2}^\sigma + \sum_{j=0}^{N+1} \sum_{k=0}^{2(N+2)-2j-1} \beta_{N+2;j,k}  \mathfrak{A}_{N+2}^{j,k}
\\
\ls \ks \bar{\mathcal{E}}_{N+2}^\sigma  + \sqrt{\mathcal{E}_{2N}^\sigma} \min\{\mathcal{E}_{N+2}^\sigma, \mathcal{D}_{N+2}^\sigma \} .
\end{multline}
\end{lemma}
\begin{proof}
We will only prove \eqref{ee_01}; the proof of \eqref{ee_02} follows from a similar argument.  Let us compactly write
\begin{equation}\label{ee_11}
\z:= \sum_{j=0}^{2N-1} \sum_{k=0}^{4N-2j-1} \beta_{2N;j,k}  \Af.
\end{equation}
First note that
\begin{multline}\label{ee_12}
\z = \sum_{j=0}^{2N-1} \sum_{k=0}^{4N-2j-1} \beta_{2N;j,k}  \Af \asymp \sum_{j=0}^{2N-1} \sum_{k=0}^{4N-2j-1}\sum_{k'=0}^k \ns{\nab_\ast^{4N-2j-k'-1} \p_3^{k'+1} \dt^j(h'(\bar{\rho})\q)}_0
\\
= \sum_{j=0}^{2N-1} \sum_{k=0}^{4N-2j-1} \ns{\nab_\ast^{4N-2j-k-1} \p_3 \dt^j(h'(\bar{\rho})\q)}_k
\asymp \sum_{j=0}^{2N-1} \ns{\p_3 \dt^j(h'(\bar{\rho})\q)}_{4N-2j-1}.
\end{multline}
Since
\begin{equation}
 \p_3 \dt^j \q = \frac{1}{h'(\bar{\rho})} \left[ \p_3 \dt^j (h'(\bar{\rho}) \q) - \p_3(h'(\bar{\rho}) )\dt^j\q  \right].
\end{equation}
and $h'(\bar{\rho})$ is smooth on $[-b,0]$ and $[0,\ell]$ and bounded below, we may estimate
\begin{equation}
 \ns{ \p_3 \dt^j q}_{0} \ls \ns{ \p_3 \dt^j(h'(\bar{\rho}) \q)}_0 +  \ns{\dt^j \q}_0 \ls \z + \ns{\dt^j \q}_0 \ls \z + \ns{\bar{\nab}_\ast^{4N-1} \q}_0
\end{equation}
and similarly
\begin{multline}
 \ns{ \p_3 \dt^j q}_{i} \ls \ns{ \p_3 \dt^j(h'(\bar{\rho}) \q)}_i +  \ns{\dt^j \q}_i \ls \z+ \ns{\nab_\ast^i \dt^j \q}_0 + \ns{ \p_3 \dt^j \q}_{i-1} \\
\ls \z+ \ns{\bar{\nab}_\ast^{4N-1} \q}_0 + \ns{ \p_3 \dt^j \q}_{i-1}
\end{multline}
for $i=1,\dotsc,4N-2j-1$.  This constitutes a sequence of estimates of the form considered in Lemma \ref{est_alg} (setting $X_i=0$ there to remove the time derivative terms).  Applying the lemma, we find that
\begin{equation}\label{ee_1}
\ns{\p_3 \dt^j  \q}_{4N-2j-1} \le \sum_{i=0}^{4N-2j-1} \ns{\p_3 \dt^j  \q}_{i}  \ls \z + \ns{\bar{\nab}_\ast^{4N-1} \q}_0.
\end{equation}
On the other hand,
\begin{equation}
 \ns{\dt^j \q}_{4N-2j} \le \ns{\bar{\nab}_\ast^{4N} \q}_{0} + \ns{\p_3 \dt^j \q}_{4N-2j-1},
\end{equation}
so summing  \eqref{ee_1} yields the estimate
\begin{equation}\label{ee_2}
 \sum_{j=0}^{2N}  \ns{\dt^j \q}_{4N-2j} \ls \z  + \ns{\bar{\nab}_\ast^{4N} \q}_{0}.
\end{equation}

Next we note that Proposition \ref{Energy positivity pro} implies that
\begin{multline}\label{ee_3}
 \Ef  \gss \ns{\bar{\nab}_\ast^{4N} \q}_{0} + \ns{\bar{\nab}_\ast^{4N} u}_{0} +
\mathcal{H}(-\rj)\left[   \ns{\bar{\nab}_\ast^{4N} \eta}_{0} + \sigma  \ns{\nab_\ast \bar{\nab}_\ast^{4N} \eta}_{0}  \right]
\\
+\mathcal{H}(\rj)\left[ \min\{1,\sigma_+,\sigma_- - \sigma_c\}  \ns{ \bar{\nab}_\ast^{4N} \eta}_{1} -  \sqrt{\mathcal{E}_{2N}^\sigma} \min\{\mathcal{E}_{2N}^\sigma, \mathcal{D}_{2N}^\sigma \} \right]
\\
+ \int_\Omega (\bar{\rho}(J-1)K+ \q+\p_3\bar\rho\theta) J\abs{\partial_t^{2N} u }^2  + h'(\bar{\rho})(J-1) \abs{\dt^{2N}\q}^2.
\end{multline}
We may easily estimate
\begin{equation}\label{ee_4}
 \abs{\int_\Omega (\bar{\rho}(J-1)K+ \q+\p_3\bar\rho\theta) J\abs{\partial_t^{2N} u }^2  + h'(\bar{\rho})(J-1) \abs{\dt^{2N}\q}^2} \ls \sqrt{\mathcal{E}_{2N}^\sigma} \min\{\mathcal{E}_{2N}^\sigma, \mathcal{D}_{2N}^\sigma\},
\end{equation}
\begin{equation}\label{ee_5}
 \ns{\bar{\nab}_\ast^{4N} \eta}_{0} + \sigma  \ns{\nab_\ast \bar{\nab}_\ast^{4N} \eta}_{0}   \gss \sum_{j=0}^{n} \ns{\dt^j \eta}_{2n-2j} +  \min\{1,\sigma\} \sum_{j=0}^n \ns{\dt^j \eta}_{2n-2j+1},
\end{equation}
and
\begin{equation}\label{ee_6}
 \ns{ \bar{\nab}_\ast^{4N} \eta}_{1} \gss \sum_{j=0}^n \ns{\dt^j \eta}_{2n-2j+1}.
\end{equation}
Plugging \eqref{ee_4}--\eqref{ee_6} into \eqref{ee_3} then yields the improved bound
\begin{multline}\label{ee_7}
  \Ef  \gss \ns{\bar{\nab}_\ast^{4N} \q}_{0} + \ns{\bar{\nab}_\ast^{4N} u}_{0} +
\mathcal{H}(-\rj)\left[ \sum_{j=0}^{n} \ns{\dt^j \eta}_{2n-2j} + \min\{1,\sigma\} \sum_{j=0}^n \ns{\dt^j \eta}_{2n-2j+1}    \right]
\\
+\mathcal{H}(\rj) \min\{1,\sigma_+,\sigma_- - \sigma_c\}   \sum_{j=0}^n \ns{\dt^j \eta}_{2n-2j+1}
-\sqrt{\mathcal{E}_{2N}^\sigma} \min\{\mathcal{E}_{2N}^\sigma, \mathcal{D}_{2N}^\sigma \} .
\end{multline}

To conclude the proof of the first inequality in \eqref{ee_01}, we sum \eqref{ee_2} and \eqref{ee_7} and employ \eqref{ee_11}.  The second inequality in \eqref{ee_01} follows easily from \eqref{ee_12}, \eqref{ee_4}, and the definition of $\mathfrak{E}_{n}^\sigma$ in \eqref{E_frak} and $\bar{\mathcal{E}}_{2N}^\sigma$ in \eqref{Enn}.

\end{proof}

\section{Comparison estimates}\label{sec_comparison}

In this section, we shall show  that ${\mathcal{E}}_{n}^\sigma$ is comparable to $ \bar{\mathcal{E}}_{n}^\sigma$ and that ${\mathcal{D}} _{n}^\sigma$ is comparable to $ \bar{\mathcal{D}}_{n}$, for both $n=2N$ and $n=N+2$.  This will be achieved by deriving further elliptic estimates.

We again assume throughout this section that the solutions obey the estimate $\gs(T) \le \delta$, where $\delta \in (0,1)$ is given in  Lemma \ref{eta_small}.

\subsection{Energy comparison}

We begin with our comparison result for the energy.

\begin{theorem}\label{eth}
Let ${\mathcal{E}}_{n}^\sigma$ and $\bar{\mathcal{E}}_{n}^\sigma$ be as defined in \eqref{p_energy_def} and \eqref{Enn}, respectively.  It holds that
\begin{equation}\label{e2n}
{\mathcal{E}}_{2N}^\sigma \lesssim
\ks \left(\bar{\mathcal{E}}_{2N}^\sigma  + (\mathcal{E}_{2N}^\sigma)^{2}+\mathcal{E}_{N+2}^\sigma\f + (\mathcal{E}_{2N}^\sigma)^{3/2} \right)
\end{equation}
and
\begin{equation}\label{en+2}
{\mathcal{E}}_{N+2}^\sigma \ls \ks \left(\bar{\mathcal{E}}_{N+2}^\sigma + \mathcal{E}_{2N}^\sigma  \mathcal{E}_{N+2}^\sigma + (\mathcal{E}_{N+2}^\sigma)^{3/2} \right).
\end{equation}
\end{theorem}
\begin{proof}
We let $n$ denote either $2N$ or $N+2$ throughout the proof, and we compactly write
\begin{equation}\label{n1}
 \mathcal{W}_n =  \ns{ \bar{\nab}^{2n-2} G^1}_{1}  +  \ns{ \bar{\nab}^{2n-2}  G^2}_{0} +
 \ns{ \bar{\nab}_{\ast}^{  2n-2} G^3}_{1/2} +
 \ns{\bar{\nab}_{\ast }^{  2n-1} G^4}_{1/2} .
\end{equation}
Note that by the definitions of $\bar{\mathcal{E}}_{n}^\sigma$ in \eqref{Enn}, we have
\begin{multline}\label{n2}
\bar{\mathcal{E}}_n^\sigma \ge \ns{\dt^n u}_0 + \sum_{j=0}^{n} \norm{ \pa_t^j \q}_{2n-2j}^2
+\mathcal{H}(-\rj)\left[ \sum_{j=0}^{n} \ns{\dt^j \eta}_{2n-2j} + \min\{1,\sigma\} \sum_{j=0}^n \ns{\dt^j \eta}_{2n-2j+1}  \right]
\\
+\mathcal{H}(\rj)\left[ \min\{1,\sigma_+,\sigma_- - \sigma_c\} \sum_{j=0}^n \ns{\dt^j \eta}_{2n-2j+1} -(\mathcal{E}_{n}^\sigma)^{3/2} \right]
\end{multline}

We first estimate $\dt^j u$ for $j=0,\dots,n-1$. The key is to use the elliptic regularity theory of the following two-phase Lam\'e system derived from \eqref{ns_perturb}:
\begin{equation}\label{lame}
\begin{cases}
 -\mu\Delta u-(\mu/3+\mu')\nabla \diverge u =G^2-\bar{\rho} \partial_t    u  - \bar{\rho}\nabla \left(h'(\bar{\rho})\q\right)    & \text{in }
\Omega  \\
 - \S(  u) e_3  = (-P'(\bar\rho)\q+\rho_1  g \eta_+ -\sigma_+ \Delta_\ast \eta_+ ) e_3 +G_+^3
 & \text{on } \Sigma_+
 \\ -\jump{\S(u)}e_3
=(\jump{P'(\bar\rho)\q}+\rj g\eta_- +\sigma_- \Delta_\ast \eta_-)e_3-G_-^3&\hbox{on }\Sigma_-
\\\jump{u}=0 &\hbox{on }\Sigma_-
\\ u_-=0 &\hbox{on }\Sigma_b.
\end{cases}
\end{equation}
We let $  j=0,\dots, n-1$ and then apply $\dt^ j$ to the problem \eqref{lame} and use the elliptic estimates of Lemma \ref{lame reg} with $r=2n-2 j\ge2$, by \eqref{n1}--\eqref{n2} and trace theory to obtain
\begin{multline}\label{n3}
 \norm{\pa_t^ j u}_{2n-2 j}^2 \lesssim  \norm{\pa_t^{ j } G^2}_{2n-2 j-2}^2+\norm{\pa_t^{ j+1} u}_{2n-2 j-2}^2
 +\norm{ \pa_t^{ j } \q }_{2n-2 j-1}^2+\norm{\pa_t^{ j }  \q }_{H^{2n-2 j-3/2}(\Sigma) }^2
   \\
+\norm{\pa_t^{ j }  \eta }_{2n-2 j-3/2 }^2+\sigma^2\norm{\pa_t^{ j }  \eta }_{2n-2 j+1/2 }^2+\norm{\pa_t^{ j } G^3}_{2n-2 j-3/2 }^2
 \\
\lesssim \norm{\pa_t^{ j+1} u}_{2n-2 ( j+1)}^2
+ \ks \left(\bar{\mathcal{E}}_{n}^\sigma + \mathcal{W}_n + (\mathcal{E}_{n}^\sigma)^{3/2}\right).
\end{multline}
Using a simple induction based on the estimate \eqref{n3}, utilizing the $\ns{\dt^n u}_0$ estimate in \eqref{n2} for the base case, we easily deduce that for $ j=0,\dots,n$,
\begin{equation}\label{n5}
 \norm{ \pa_t^ j u}_{2n-2 j}^2\lesssim  \ks  \left(\bar{\mathcal{E}}_{n}^\sigma + \mathcal{W}_n + (\mathcal{E}_{n}^\sigma)^{3/2} \right).
\end{equation}

We then estimate $\pa_t^ j\q$ and $\pa_t^ j\eta$ for $ j=1,\dots,n$ to get an improvement. By the first equation of $\eqref{ns_perturb}$, using the estimates \eqref{n5} and \eqref{n1}, we have that for $ j=1,\dots,n$,
\begin{multline}\label{n6}
\norm{ \pa_t^ j \q}_{2n-2 j+1}^2 \lesssim   \norm{ \pa_t^{ j-1} u}_{2n-2 j+2}^2+\norm{\pa_t^{ j-1} G^1}_{2n-2 j+1}^2
 \\ = \norm{ \pa_t^{ j-1} u}_{2n-2({ j-1})}^2+\norm{\pa_t^{ j-1} G^1}_{2n-2({ j-1})-1}^2
\ls  \ks \left(\bar{\mathcal{E}}_{n}^\sigma + \mathcal{W}_n + (\mathcal{E}_{n}^\sigma)^{3/2} \right).
\end{multline}
Now by the kinematic boundary condition
\begin{equation}
\partial_t\eta=u_3+G^4\text{ on }\Sigma,
\end{equation}
we have that for $j=1,\dots,n$, by the trace theory, \eqref{n5} and \eqref{n1},
\begin{multline}\label{n7}
\ns{\partial_t^j\eta}_{2n-2j+3/2}
\le \ns{\dt^{j-1}u_3}_{H^{2n-2j+3/2}(\Sigma)}+\ns{\dt^{j-1} G^4}_{2n-2j+3/2}
\\
\ls \ns{\dt^{j-1}u }_{2n-2(j-1)}+\ns{\dt^{j-1} G^4}_{2n-2(j-1)-1/2}
\ls  \ks \left(\bar{\mathcal{E}}_{n}^\sigma + \mathcal{W}_n + (\mathcal{E}_{n}^\sigma)^{3/2} \right).
\end{multline}

Consequently, summing  the estimates \eqref{n2}, \eqref{n5}, \eqref{n6} and \eqref{n7},  we conclude
\begin{equation}\label{claim}
\mathcal{E}_{n}^\sigma \lesssim   \ks  \left(\bar{\mathcal{E}}_{n}^\sigma + \mathcal{W}_n + (\mathcal{E}_{n}^\sigma)^{3/2} \right).
\end{equation}
Setting $n=2N$ in \eqref{claim}, and using \eqref{p_G_e_0} of Lemma \ref{p_G2N_estimates} to bound $\mathcal{W}_{2N} \lesssim \ks[(\mathcal{E}_{2N}^\sigma)^{2}+\mathcal{E}_{N+2}^\sigma\f]$, we obtain \eqref{e2n}; setting $n=N+2$ in \eqref{claim}, and using \eqref{p_G_e_h_0} of Lemma \ref{p_GN+2_estimates} to bound $\mathcal{W}_{N+2}\lesssim \ks {\mathcal{E}} _{2N}^\sigma {\mathcal{E}}_{N+2}^\sigma$, we obtain \eqref{en+2}.
\end{proof}

\subsection{Dissipation comparison}

Next we consider a similar result for the dissipation.  To begin we prove a lemma involving norms of $\q$.

\begin{lemma}\label{q_iteration}
Let $\sdb{n}$ be as defined in \eqref{Dnn}.  Then
\begin{equation}
 \ns{\q}_{2n} \ls \sdb{n} + \ns{\q}_0.
\end{equation}
\end{lemma}
\begin{proof}
 We compute
\begin{equation}
 \nab \q = \frac{1}{h'(\bar{\rho})} \left[ \nab(h'(\bar{\rho}) \q) - \nab(h'(\bar{\rho}) )\q  \right].
\end{equation}
Since $h'(\bar{\rho})$ is smooth on $[-b,0]$ and $[0,\ell]$ and bounded below, we may estimate
\begin{equation}
 \ns{\nab q}_{0} \ls \ns{\nab(h'(\bar{\rho}) \q)}_0 +  \ns{\q}_0 \ls \sdb{n} + \ns{\q}_0
\end{equation}
and similarly
\begin{equation}
 \ns{\nab q}_{i} \ls \ns{\nab(h'(\bar{\rho}) \q)}_i +  \ns{\q}_i \ls \sdb{n} + \ns{\q}_0 + \ns{\nab \q}_{i-1}
\end{equation}
for $i=1,\dotsc,2n-1$.  This constitutes a sequence of estimates of the form considered in Lemma \ref{est_alg} (setting $X_i=0$ there to remove the time derivative terms).  Applying the lemma, we find that
\begin{equation}
\ns{\nab \q}_{2n-1} \le \sum_{i=0}^{2n-1} \ns{\nab \q}_{i}  \ls \sdb{n} + \ns{\q}_0,
\end{equation}
which yields  the desired estimate upon adding $\ns{\q}_0$ to both sides.
\end{proof}

Now we can record the dissipation comparison results.

\begin{theorem}\label{dth}
Let ${\mathcal{D}}_{n}^\sigma$ and $\bar{\mathcal{D}}_{n}$ be as defined in \eqref{p_dissipation_def} and \eqref{Dnn}, respectively. It holds that
\begin{equation}\label{d2n}
{\mathcal{D}}_{2N}^\sigma
\ls \ks \left(  \bar{\mathcal{D}}_{2N} + \mathcal{E}_{2N}^\sigma \mathcal{D}_{2N}^\sigma + \sqrt{\mathcal{E}_{2N}^\sigma} \mathcal{D}_{2N}^\sigma   + \mathcal{E}_{N+2}^\sigma{\mathcal{F}}_{2N}\right)
 \end{equation}
and
\begin{equation}\label{dn+2}
\mathcal{D}_{N+2}^\sigma
\ls \ks \left(  \bar{\mathcal{D}}_{N+2} + \mathcal{E}_{2N}^\sigma \mathcal{D}_{N+2}^\sigma + \sqrt{\mathcal{E}_{N+2}^\sigma} \mathcal{D}_{N+2}^\sigma  \right).
\end{equation}
\end{theorem}

\begin{proof}
We again let $n$ denote either $2N$ or $N+2$ throughout the proof, and we compactly write
\begin{equation}\label{p_D_b_4}
\begin{split}
 \y_{n} = & \ns{ \bar{\nab}^{2n-1} G^1}_{0} + \ns{ \bar{\nab}^{2n-2}\dt G^1}_{0} +
 \ns{\bar{\nab}_{\ast}^{2n-1} G^3}_{1/2} \\&+ \ns{\bar{\nab}_{\ast }^{\  2n-1} G^4}_{1/2}
+ \ns{\bar{\nab}^{2n-2} \dt G^4}_{1/2}+ \sigma^2\ns{\nab_\ast^{2n}   G^4}_{1/2}.
\end{split}
\end{equation}
We now estimate the remaining parts of $\bar{\mathcal{D}}_{n}$ not contained in ${\mathcal{D}}_{n}^\sigma$, that is, to estimate $\eta$ and improve the estimates of $\q$.  We divide the proof into several steps.

Step 1 -- $\q$ estimates

We first notice that by the first equation of \eqref{ns_perturb},
\begin{equation}
\norm { \pa_t \q
 }_{2n-1}^2\le \norm {u }_{2n}^2+\norm {
G^{1}}_{2n -1}^2\lesssim \sdb{n}   +\y_n,
\end{equation}
and for $2\le  j\le n+1$,
\begin{equation}
\begin{split}
\norm { \pa_t^{ j}\q}_{2n -2 j+2}^2&\le \norm { \pa_t^{ j-1}u}_{2n -2 j+3}^2+\norm { \pa_t^{ j-1}G^{1}}_{2n -2 j+2}^2
\\&\le \norm { \pa_t^{ j-1}u}_{2n -2 j+3}^2+\norm { \pa_t^{ j-1}G^{1}}_{2n -2 j+2}^2\lesssim \sdb{n}   +\y_n.
\end{split}
\end{equation}
We may sum these estimates with the estimate of Lemma \ref{q_iteration} to see that
\begin{equation}\label{dth_4}
 \ns{\q}_{2n} + \ns{\dt \q}_{2n-1} + \sum_{j=2}^{n+2} \ns{\dt^j \q}_{2n-2j+2} \ls \sdb{n} + \y_n + \ns{\q}_0.
\end{equation}

Step 2 -- $\dt^j \eta$ estimates

We now derive estimates for time derivatives of $\eta$. For the term $\dt^j \eta$ for $j\ge 2$ we use the kinematic boundary condition
\begin{equation}\label{n61}
\partial_t\eta=u_3+G^4\text{ on }\Sigma.
\end{equation}
Indeed, for $j=2,\dots,n+1$, by the trace theory and \eqref{p_D_b_4}, we have
\begin{equation}\label{eta1}
\begin{split}
\ns{\partial_t^j\eta}_{2n-2j+5/2} & \ls  \ns{\partial_t^{j-1}u_3}_{H^{2n-2j+5/2}(\Sigma)} +\ns{\partial_t^{j-1}G^4}_{2n-2j+5/2}
\\
&\lesssim   \ns{\partial_t^{j-1}u }_{{2n-2(j-1)+1}} +\ns{\partial_t^{j-1}G^4}_{2n-2(j-1)+1/2}
 \lesssim\sdb{n}   +\y_n.
\end{split}
\end{equation}
For the term $\partial_t\eta$, we again use \eqref{n61}, the trace theory and \eqref{p_D_b_4} to find
\begin{multline}\label{eta2}
\sigma^2\ns{\partial_t \eta}_{2n+1/2}+\ns{\partial_t \eta}_{2n-1/2} \lesssim (1+\sigma^2)\ns{u_3}_{H^{2n+1/2}(\Sigma)}+\sigma^2\ns{ G^4}_{2n+1/2}+\ns{ G^4}_{2n-1/2}
\\ \lesssim (1+\max\{\sigma_+^2,\sigma_-^2\}) \ns{ u }_{2n+1}+\y_n\lesssim \ks \sdb{n}   +\y_n.
\end{multline}
Hence
\begin{equation}\label{dth_5}
\sigma^2\ns{\partial_t \eta}_{2n+1/2}+\ns{\partial_t \eta}_{2n-1/2}  + \sum_{j=2}^{n+1} \ns{\dt^j \eta}_{2n-2j+5/2} \ls \ks \sdb{n} + \y_n.
\end{equation}

Step 3 -- $\nab_\ast \eta$ estimates

For the remaining $\eta$ term without temporal derivatives we use the dynamic boundary condition
\begin{equation}\label{pb1}
 -\sigma_+\Delta_\ast\eta_++\rho_1  g\eta_+= P'_+(\rho_1 )\q_+-2 \mu_+\partial_3u_{3,+}-\mu'_+\diverge u_+ -G_{3,+}^3\text{ on }\Sigma_+
\end{equation}
 and
\begin{equation} \label{pb2}
-\sigma_-\Delta_\ast\eta_--\rj g\eta_-=-\jump{P'(\bar\rho)\q}+2\jump{\mu\partial_3u_3}+\jump{\mu'\diverge u}-G_{3,-}^3 \text{ on } \Sigma_-.
\end{equation}
Notice that at this point we do not have any bound of $\q$ on the boundary $\Sigma$, but we have bounded  $\nabla (h'(\bar\rho)\q)$ in $\Omega$.  As such, we first apply $\nab_\ast$ to \eqref{pb1} and \eqref{pb2}.  We use Lemma \ref{bndry_elliptic} on \eqref{pb1}; when $\rj <0$ we also use Lemma \ref{bndry_elliptic} on \eqref{pb2}, but when $\rj \ge 0$ we instead use \eqref{elliptt} of Lemma \ref{critical lemma}.  This, trace theory, and \eqref{p_D_b_4} then provide the estimate
 \begin{multline}\label{n51}
\mathcal{H}(\rj) \min\{ \sigma_+^2,(\sigma_--\sigma_c)^2 \} \ns{ \nab_\ast \eta}_{2n+1/2}
+ \mathcal{H}(-\rj) \left(\sigma^2 \ns{\nab_\ast \eta}_{2n+1/2} + \ns{ \nab_\ast \eta}_{2n-3/2}\right)
\\
\lesssim \ns{  \nab_\ast (h'(\bar\rho)\q) }_{H^{2n-3/2}(\Sigma)}  + \ns{  \nab_\ast \nabla u }_{H^{2n-3/2}(\Sigma)} +\ns{\nab_\ast G^3_3}_{ 2n-3/2 }
\\
 \lesssim \ns{\nabla (h'(\bar\rho)\q)}_{2n-1} + \ns{u }_{2n+1}  + \ns{G^3}_{2n-1/2} \lesssim \sdb{n}   +\y_n,
\end{multline}
where $\mathcal{H} = \chi_{(0,\infty)}$ denotes the Heaviside function.

Step 4 -- Recovering $\ns{\q}_0$ and $\ns{\eta}_0$

Next we seek to control $\norm{\q}_0^2$  and  $\norm{\eta}_0^2$. To this end, we recall the conservations of mass \eqref{cons1}--\eqref{cons2} and the definition of $\Phi$ in \eqref{Phi_def}.  Standard arguments allow us  to find $w=w_\pm$, a solution to the following problems:
\begin{equation}\label{pro1}
 \diverge w_+=\q_+ + \Phi_+  \text{ in }\Omega_+,\ w_{3,+}=  -\rho_1 \eta_+ \text{ on }\Sigma_+ ,
 w_{3,+}= -\rho^+\eta_- \text{ on }\Sigma_-
\end{equation}
and
\begin{equation}\label{pro2}
 \diverge w_-=\q_- + \Phi_- \text{ in }\Omega_-,\ w_{3,-}= -\rho^-\eta_- \text{ on }\Sigma_- ,\quad w_{-}=0\ \text{ on }\Sigma_b.
\end{equation}
Note that the equalities \eqref{cons1} and \eqref{cons2} provide the necessary compatibility conditions for the solvability of the problems \eqref{pro1} and \eqref{pro2}, respectively. Moreover, we have the following estimates
\begin{equation}\label{w es}
\norm{w}_{1}\lesssim   \norm{\q}_0+\norm{\Phi}_0+\norm{\eta}_{1/2}.
\end{equation}

Now we take the dot product of the second equation of $\eqref{ns_perturb}$ with ${(\bar\rho)}^{-1}w$, we obtain
\begin{equation}\label{d1}
\int_\Omega
 \nabla  (h'(\bar\rho) \q)\cdot w-\diverge  \S(  u)\cdot  {(\bar\rho)}^{-1} w =\int_\Omega({(\bar\rho)}^{-1} G^2-  \partial_t    u )w.
\end{equation}
Using the equations for $w=w_\pm$, we have
\begin{equation}\label{d2}
\begin{split}
&\int_{\Omega_+} \nabla  (h_+'(\bar{\rho}_+)\q_+)\cdot w_+-\diverge  \S_+(  u_+)\cdot  {(\bar\rho_+)}^{-1} w_+
= \int_{\Sigma_+} {(\bar\rho_+)}^{-1}(P_+^\prime(\bar\rho_+)\q_+ I-\S_+(u_+))e_3\cdot w_+
\\
&-\int_{\Sigma_-} {(\bar\rho_+)}^{-1}(P_+^\prime(\bar\rho_+)\q_+ I-\S_+(u_+))e_3\cdot w_+
-\int_{\Omega_+}   h_+'(\bar{\rho}_+)\q_+ \diverge w_+- \S_+(  u_+): \nabla ({(\bar\rho_+)}^{-1}w_+)
\\
&= -\int_{\Sigma_+}  \eta_+\left[(P_+^\prime(\bar\rho_+)\q_+ I-\S_+(u_+))e_3\right]\cdot e_3
     -(\rho_1 )^{-1} \sum_{i=1,2} (\S_+(u_+))_{i3}w_{i,+}
\\
&\quad+\int_{\Sigma_-}  \eta_-\left[(P_+^\prime(\bar\rho_+)\q_+ I-\S_+(u_+))e_3\right]\cdot e_3  -(\bar{ \rho}_+)^{-1} \sum_{i=1,2} (\S_+(u_+))_{i3}w_{i,+}
\\
&\quad-\int_{\Omega_+} h_+'(\bar\rho_+)   \q_+ (\q_+ + \Phi_+)-   \S_+(  u_+): \nabla (({\bar\rho}_+)^{-1}w_+),
\end{split}
\end{equation}
and similarly,
\begin{equation}\label{d3}
\begin{split}
&\int_{\Omega_-} \nabla  (h_-'(\bar{\rho}_-)\q_-)\cdot w_- -\diverge  \S_-(  u_-)\cdot  ({\bar\rho_-})^{-1} w_-
=-\int_{\Sigma_-} \eta_-\left[(P_-^\prime(\bar\rho_-)\q_- I-\S_-(u_-))e_3\right]\cdot e_3
\\&     - \sum_{i=1,2} \int_{\Sigma_-}  (\bar{ \rho}_-)^{-1} (\S_-(u_-))_{i3}w_{i,-}
-\int_{\Omega_-} h_-'(\bar\rho_-)     \q_- (\q_- + \Phi_-)-   \S_-(  u_-): \nabla (({\bar\rho}_-)^{-1}w_-).
\end{split}
\end{equation}
Collecting \eqref{d1}--\eqref{d3}, and the boundary conditions of $\eqref{ns_perturb}$, we get
\begin{equation}
\begin{split}
 \int_\Omega & ({(\bar\rho)}^{-1} G^2-  \partial_t    u )w
=-\int_{\Omega} h'(\bar\rho) \q(\q + \Phi)-   \S(  u): \nabla ({(\bar\rho)}^{-1}w)
\\
&-\int_{\Sigma_+} \eta_+\left[(P_+^\prime(\bar\rho_+)\q_+ I-\S_+(u_+))e_3\right]\cdot e_3
- \sum_{i=1,2}  ({ \rho}_1)^{-1} (\S_+(u_+))_{i3}w_{i,+}
\\
 &+\int_{\Sigma_-} \eta_-\jump{P^\prime(\bar\rho)\q I-\S(u))e_3}\cdot e_3
-\sum_{i=1,2}  \jump{(\bar{ \rho})^{-1} (\S(u))_{i3}w_{i}}
\\
 =&-\int_{\Omega} h'(\bar\rho) \q(\q + \Phi)-   \S(  u): \nabla ({(\bar\rho)}^{-1}w)
\\
&-\int_{\Sigma_+} \eta_+(\rho_1 g   \eta_+ -\sigma_+ \Delta_\ast \eta_++G^3_{3,+})
- \sum_{i=1,2}  ({ \rho}_1)^{-1} (\S_+(u_+))_{i3}w_{i,+}
\\
&+\int_{\Sigma_-} \eta_-  (\rj g   \eta_- +\sigma_- \Delta_\ast \eta_--G^3_{3,-})
     -\sum_{i=1,2} \jump{(\bar{ \rho})^{-1} (\S(u))_{i3}w_{i}}.
\end{split}
\end{equation}

We then further deduce from the above that
\begin{equation}
\begin{split}
& \int_{\Omega} h'(\bar\rho)
     |\q|^2+\int_{\Sigma_+} \rho_1 g   |\eta_+|^2- \int_{\Sigma_-} \rj g|\eta_-|^2
      +\int_\Sigma \sigma |\nabla_\ast \eta|^2
    \\&  \quad=-\int_\Omega({(\bar\rho)}^{-1} G^2-  \partial_t    u )w-\int_{\Omega} h'(\bar\rho)
     \q  \Phi -   \S(  u): \nabla ({(\bar\rho)}^{-1}w)
     - \int_\Sigma  \eta G^3_{3}
     \\&\qquad- \sum_{i=1,2}\int_{\Sigma_+}  ({ \rho}_1)^{-1} (\S_+(u_+))_{i3}w_{i,+}
     -\sum_{i=1,2}\int_{\Sigma_-}  \jump{(\bar{ \rho})^{-1} (\S(u))_{i3}w_{i}}.
     \end{split}
\end{equation}
Using Proposition \ref{Energy positivity pro}, trace theory, and Cauchy's inequality, we deduce from the above that for every $\ep >0$
\begin{multline}\label{dth_1}
\norm{\q}_0^2 +
\mathcal{H}(\rj)  \min\{1,\sigma_+, \sigma_--\sigma_c \} \norm{ \eta}_1^2
+ \mathcal{H}(-\rj)  \left[ \norm{\eta}_0^2 + \sigma \norm{\nab_\ast \eta}_0^2 \right]
\\
\ls \norm{G^2}_0\norm{w}_0+\norm{\p_tu}_0\norm{w}_0+\norm{\q}_0\norm{\Phi}_0
+\norm{u}_1\norm{w}_1\\
+\norm{\eta}_0\norm{G^3}_0+\norm{\na u}_{H^{0}(\Sigma)}\norm{w}_{H^{0}(\Sigma)}
+ \mathcal{H}(\rj)\sqrt{\mathcal{E}_{n}^\sigma} \min\{\mathcal{E}_{n}^\sigma, \mathcal{D}_{n}^\sigma \}
\\
\le \frac{C}{\ep}\left( \norm{\p_tu}_0^2+\norm{u}_2^2+\norm{G^2}_0^2+\norm{G^3}_0^2+\norm{\Phi}_0^2\right)
+\varepsilon \left(\norm{w}_1^2 +\norm{\q}_0^2+\norm{\eta}_0^2\right)
\\+ \mathcal{H}(\rj)\sqrt{\mathcal{E}_{n}^\sigma} \mathcal{D}_{n}^\sigma,
\end{multline}
where again $\mathcal{H} = \chi_{(0,\infty)}$.

Step 5 -- Handling \eqref{dth_1} with cases

To proceed we must break to cases depending on the sign of $\rj$.  If $\rj >0$ then we employ \eqref{w es} and Lemma \ref{Phi_est} to successively estimate
\begin{equation}
 \varepsilon \left(\norm{w}_1^2 +\norm{\q}_0^2+\norm{\eta}_0^2\right) \le C \ep (\ns{\q}_0 + \ns{\eta}_1 + \ns{\Phi}_0)
\le  C \ep (\ns{\q}_0 + \ns{\eta}_1 + \sqrt{\mathcal{E}_{n}^\sigma} \mathcal{D}_{n}^\sigma ).
\end{equation}
We then choose $\ep>0$ according to
\begin{equation}
 C\ep = \frac{1}{2} \min\{1, \sigma_+,  (\sigma_--\sigma_c) \},
\end{equation}
so that we may absorb $\ns{\q}_0 + \ns{\eta}_1$ onto the left in \eqref{dth_1}.  This yields the estimate (again estimating $\Phi$ with Lemma \ref{Phi_est})
\begin{multline}\label{dth_2}
 \norm{\q}_0^2 +  \min\{1,\sigma_+, \sigma_--\sigma_c \} \norm{ \eta}_1^2   \\
 \ls
\ks \left( \norm{\p_tu}_0^2+\norm{u}_2^2+\norm{G^2}_0^2+\norm{G^3}_0^2 + \sqrt{\mathcal{E}_{n}^\sigma} \mathcal{D}_{n}^\sigma \right)
\\
\ls
\ks \left(  \sdb{n}   +\y_n + \sqrt{\mathcal{E}_{n}^\sigma} \mathcal{D}_{n}^\sigma \right),
\end{multline}
where the constant $\ks$ is as in \eqref{ks_def} when $\rj >0$.

On the other hand, if $\rj <0$, then we use \eqref{w es} and  Lemma \ref{Phi_est} (to estimate $\Phi$) to bound
\begin{equation}
\begin{split}
  \varepsilon \left(\norm{w}_1^2 +\norm{\q}_0^2+\norm{\eta}_0^2\right)  &\le C\varepsilon \left( \norm{\q}_0^2+\norm{\Phi}_0^2+\norm{\eta}_0^2+ \ns{\nab_\ast \eta}_0 \right)\\
& \le C \ep( \ns{\q}_0 + \ns{\eta}_0 + \ns{\nab_\ast \eta}_0 + \sqrt{\mathcal{E}_{n}^\sigma} \mathcal{D}_{n}^\sigma).
\end{split}
\end{equation}
In this case we choose $\ep>0$ so that  $C \ep = \hal,$ which allows us to absorb $\ns{\q}_0 + \ns{\eta}_0$ onto the left of \eqref{dth_1}.  From this we deduce (again using Lemma \ref{Phi_est}) that
\begin{multline}\label{dth_3}
 \norm{\q}_0^2 +  \norm{\eta}_0^2 + \sigma \norm{\nab_\ast \eta}_0^2
\\
\ls  \norm{\p_tu}_0^2+\norm{u}_2^2+\norm{G^2}_0^2+\norm{G^3}_0^2+\norm{\nab_\ast \eta}_0^2+ \sqrt{\mathcal{E}_{n}^\sigma} \mathcal{D}_{n}^\sigma
\\
\ls \sdb{n}   +\y_n + \sqrt{\mathcal{E}_{n}^\sigma} \mathcal{D}_{n}^\sigma,
\end{multline}
where in the last estimate we have used \eqref{n51} to estimate $\ns{\nab_\ast \eta}_0$.

Combining \eqref{dth_2} and \eqref{dth_3} with \eqref{n51}, we find that
\begin{multline}\label{dth_6}
  \ns{\q}_0  +  \mathcal{H}(\rj) \min\{1,\sigma_+, \sigma_--\sigma_c, \sigma_+^2,(\sigma_- - \sigma_c)^2\} \ns{ \eta}_{2n+3/2} \\
+ \mathcal{H}(-\rj) \left(\ns{\eta}_{2n-1/2} + \min\{1,\sigma^2\} \ns{ \eta}_{2n+3/2} \right) \\
\ls \ks  \left(  \sdb{n}   +\y_n + \sqrt{\mathcal{E}_{n}^\sigma} \mathcal{D}_{n}^\sigma \right).
\end{multline}

Step 6 -- Completion

Now we sum \eqref{dth_4} and \eqref{dth_5} and add to the resulting inequality a positive multiple of \eqref{dth_6}, where the multiplier is chosen large enough to absorb onto the left the term $\ns{\q}_0$ on the right side of the sum of \eqref{dth_4}.  This results in the estimate
\begin{multline}\label{dth_7}
  \ns{\q}_{2n} + \ns{\dt \q}_{2n-1} + \sum_{j=2}^{n+2} \ns{\dt^j \q}_{2n-2j+2}
+  \mathcal{H}(\rj) \min\{1,\sigma_+, \sigma_--\sigma_c, \sigma_+^2,(\sigma_- - \sigma_c)^2\} \ns{ \eta}_{2n+3/2}
 \\
+ \mathcal{H}(-\rj) \left(\ns{\eta}_{2n-1/2} + \min\{1,\sigma^2\} \ns{ \eta}_{2n+3/2} \right)
+\ns{\partial_t \eta}_{2n-1/2} + \sigma^2\ns{\partial_t \eta}_{2n+1/2}  \\ + \sum_{j=2}^{n+1} \ns{\dt^j \eta}_{2n-2j+5/2}
\ls \ks \left(  \sdb{n}   +\y_n + \sqrt{\mathcal{E}_{n}^\sigma} \mathcal{D}_{n}^\sigma \right).
\end{multline}

When $n=2N$ we use \eqref{p_G_e_00} of Lemma \ref{p_G2N_estimates} to estimate
$\mathcal{Y}_{2N}\lesssim \ks {\mathcal{E}}_{2N}^\sigma {\mathcal{D}}_{2N}^\sigma + {\mathcal{E}}_{N+2}^\sigma {\mathcal{F}}_{2N}$.  Then we obtain \eqref{d2n} from \eqref{dth_7} by adding $\sdb{2N}$ to both sides and recalling that by definition, \eqref{Dnn}, $\sdb{2N}$ controls all of the $u$ terms in $\mathcal{D}_{2N}^\sigma$.  When $n=N+2$ we use \eqref{p_G_e_h_00} of Lemma \ref{p_GN+2_estimates} to estimate $\mathcal{Y}_{N+2}\lesssim \ks {\mathcal{E}}_{2N}^\sigma {\mathcal{D}}_{N+2}^\sigma$.  Then we obtain \eqref{dn+2} from \eqref{dth_7} by adding $\sdb{N+2}$ to both sides.

\end{proof}

\section{A priori estimates}\label{sec_apriori}

We are now ready to derive the global-in-time bounds and decay of high order energy $\mathcal{E}_{2N}^\sigma$ and $\mathcal{E}_{N+2}^\sigma$ based on the previous estimates on various energies.

Recall that $\gs$ is defined by \eqref{G_def}.

\subsection{Boundedness at the $2N$ level}
In this subsection, we shall  show the boundedness at the $2N$ level.   We first recall $\f$ defined in \eqref{fff}. We will make use of the following transport estimates for $\f$.

\begin{Proposition}\label{p_f_bound}
There exists a universal constant $0 < \delta_0 < 1$ so that if $\gs(t) \le \delta \le  \delta_0$, then
\begin{equation}\label{p_f_b_0}
\f(r) \ls
\f(0) +   r \int_0^r \sd{2N}^\sigma
\end{equation}
for $0 \le r \le t$.
\end{Proposition}
\begin{proof}
The estimate is recorded in Proposition 7.2 of \cite{GT_per}.  The proof is based on the transport theory of the kinematic boundary condition.
\end{proof}

We now present the a priori energy bounds for $\mathcal{E}_{2N}^\sigma,$  $\mathcal{D}_{2N}^\sigma,$ and $\f$.

\begin{Proposition} \label{Dgle}
There exists a constant $\ks$ of the form \eqref{ks_def} and a universal constant $\delta_0 >0$ so that if $0 < \delta \le \delta_0 /\ks$  and $\gs(t) \le \delta$, then
\begin{equation}\label{Dg}
\sup_{0\le r \le t}\mathcal{E}_{2N}^\sigma(r)+\int_0^t\mathcal{D}_{2N}^\sigma(r)dr + \sup_{0\le r \le t}  \frac{\f(r)}{(1+r)}
\lesssim \ks \mathcal{E}_{2N}^\sigma(0) + \mathcal{F}_{2N}(0).
\end{equation}
\end{Proposition}

\begin{proof}

Throughout this proof let us denote by $\tilde{\mathcal{E}}_{2N}^\sigma$ the quantity appearing in the time derivative in \eqref{ebar2N} of Proposition \ref{boostrap2N}.  By assumption we have that $\sup_{0\le r\le t}\mathcal{E}_{2N}^\sigma(r) \le \delta$.  In what follows we will restrict the size of $\delta$ in order to prove \eqref{Dg}.  We assume initially that $\delta \le 1$ so that $(\mathcal{E}_{2N}^\sigma)^{p_1} \le (\mathcal{E}_{2N}^\sigma)^{p_2}$ if $p_1 \ge p_2$.  This allows us to simplify many of the subsequent estimates by retaining only the lowest power of $\mathcal{E}_{2N}^\sigma$ appearing in inequalities.   We also assume that $\delta$ is as small as in Lemma \ref{eta_small}.

We may  integrate the estimates \eqref{ebar2N} of Proposition \ref{boostrap2N} to see that
\begin{equation} \label{dg_1}
\tilde{\mathcal{E}}_{2N }^\sigma(t)+\int_0^t\bar{\mathcal{D}}_{2N}
\lesssim \tilde{\mathcal{E}}_{2N}^\sigma(0)  + \ks \int_0^t \left[ \sqrt{\mathcal{E}_{2N}^\sigma }\mathcal{D}_{2N}^\sigma+\sqrt{ \mathcal{D}_{2N}^\sigma\mathcal{E}_{N+2}^\sigma\f}+\mathcal{E}_{N+2}^\sigma\f \right].
\end{equation}
According to the first inequality of \eqref{ee_01} in Lemma \ref{en_equiv}, combined with estimate \eqref{e2n} of Theorem \ref{eth} we have that
\begin{equation}\label{dg_2}
 \mathcal{E}_{2N}^\sigma \ls \ks \left[   \tilde{\mathcal{E}}_{2N}^\sigma + (\mathcal{E}_{2N}^\sigma)^{3/2} + \mathcal{E}_{N+2}^\sigma \f   \right].
\end{equation}
Estimate \eqref{d2n} of Theorem \ref{dth} yields the bound
\begin{equation}\label{dg_3}
  \mathcal{D}_{2N}^\sigma \ls \ks \left[   \bar{\mathcal{D}}_{2N}^\sigma  + \sqrt{\mathcal{E}_{2N}^\sigma}\mathcal{D}_{2N}^\sigma + \mathcal{E}_{N+2}^\sigma \f   \right].
\end{equation}
Additionally, the second inequality of \eqref{ee_01} in Lemma \ref{en_equiv}, together with the trivial bound $\bar{\mathcal{E}}_{2N}^\sigma \le \mathcal{E}_{2N}^\sigma$, implies that
\begin{equation}\label{dg_4}
\tilde{\mathcal{E}}_{2N}^\sigma   \ls \ks  \mathcal{E}_{2N}^\sigma.
\end{equation}
We may assume, by taking the maximum, that the constants $\ks \ge 1$ appearing in \eqref{dg_1}--\eqref{dg_4} are identical.

Let us assume that $\sqrt{\delta} \ks$ is sufficiently small to absorb the second terms on the right sides of \eqref{dg_2} and \eqref{dg_3} onto the left with a factor of $1/2$.  Then \eqref{dg_2} and \eqref{dg_3} improve to
\begin{equation}\label{dg_5}
 \mathcal{E}_{2N}^\sigma \ls \ks \left[   \tilde{\mathcal{E}}_{2N}^\sigma  + \mathcal{E}_{N+2}^\sigma \f   \right]
\end{equation}
and
\begin{equation}\label{dg_6}
  \mathcal{D}_{2N}^\sigma \ls \ks \left[   \bar{\mathcal{D}}_{2N}^\sigma  + \mathcal{E}_{N+2}^\sigma \f   \right].
\end{equation}
We then plug \eqref{dg_4}--\eqref{dg_6} into \eqref{dg_1} to see that
\begin{multline}\label{dg_7}
 \mathcal{E}_{2N}^\sigma(t) + \int_0^t \mathcal{D}_{2N}
\ls \ks  \mathcal{E}_{2N}^\sigma(0)   + \ks \mathcal{E}_{N+2}^\sigma(t) \f(t)
\\
+ \ks^2 \int_0^t \left[ \sqrt{\mathcal{E}_{2N}^\sigma }\mathcal{D}_{2N}^\sigma+\sqrt{ \mathcal{D}_{2N}^\sigma\mathcal{E}_{N+2}^\sigma\f}+\mathcal{E}_{N+2}^\sigma\f \right].
\end{multline}
For any $\ep >0$ we may apply Cauchy's inequality to bound
\begin{equation}
 \ks^2 \int_0^t  \sqrt{ \mathcal{D}_{2N}^\sigma\mathcal{E}_{N+2}^\sigma\f}  \le \ep \int_0^t \mathcal{D}_{2N}^\sigma + \frac{\ks^4}{4\ep} \int_0^t \mathcal{E}_{N+2}^\sigma\f.
\end{equation}
Taking $\ep$ sufficiently small, we may absorb the $\int_0^t \mathcal{D}_{2N}^\sigma$ term onto the left in \eqref{dg_7}, resulting in the bound
\begin{multline}\label{dg_8}
  \mathcal{E}_{2N}^\sigma(t) + \int_0^t \mathcal{D}_{2N}
\ls \ks  \mathcal{E}_{2N}^\sigma(0)   + \ks \mathcal{E}_{N+2}^\sigma(t) \f(t)
\\
+ \ks^2 \int_0^t \left[ \sqrt{\mathcal{E}_{2N}^\sigma }\mathcal{D}_{2N}^\sigma  + (1 + \ks^2)  \mathcal{E}_{N+2}^\sigma\f \right].
\end{multline}

The decay of $\mathcal{E}_{N+2}^\sigma(t)$ guaranteed by the bound on $\gs(t)$, combined with the bound of Proposition \ref{p_f_bound}, easily imply that
\begin{multline}\label{dg_9}
\mathcal{E}_{N+2}^\sigma(t)\f(t) \lesssim \delta(1+t)^{-4N+8}\f(t) \lesssim \delta\mathcal{F}_{2N}(0)+\delta \int_0^t\mathcal{D}_{2N}^\sigma(r)dr \\
\ls \sqrt{\delta} \f(0) + \sqrt{\delta} \int_0^t\mathcal{D}_{2N}^\sigma(r)dr
\end{multline}
and
\begin{equation}\label{dg_10}
\int_0^t \mathcal{E}_{N+2}^\sigma(r)\mathcal{F}_{2N}(r)dr\lesssim \delta\mathcal{F}_{2N}(0)+\delta\int_0^t\mathcal{D}_{2N}^\sigma(r)dr \ls \sqrt{\delta} \f(0) + \sqrt{\delta}\int_0^t\mathcal{D}_{2N}^\sigma(r)dr.
\end{equation}
We plug \eqref{dg_9} and \eqref{dg_10} into \eqref{dg_8} and then further assume that $\sqrt{\delta} (\ks + \ks^2 +\ks^4)$ is sufficiently small to absorb all of the resulting $\int_0^t \mathcal{D}_{2N}^\sigma$ terms on the right.  This results in the estimate
\begin{equation}\label{dg_11}
\mathcal{E}_{2N}^\sigma(t) + \int_0^t\mathcal{D}_{2N}^\sigma(r) dr \lesssim \ks \mathcal{E}_{2N}^\sigma(0) + \mathcal{F}_{2N}(0).
\end{equation}
Then the estimate \eqref{Dg} follows from \eqref{dg_11} and \eqref{p_f_b_0}.
\end{proof}

\subsection{Decay at the $N+2$ level}

We next show the decay at the $N+2$ level.  The first step is an interpolation argument that allows us to control $\mathcal{E}_{N+2}^\sigma$ in terms of $\mathcal{D}_{N+2}^\sigma$ and $\mathcal{E}_{2N}^\sigma$.

\begin{lemma}\label{en_interp}
Let $\mathcal{E}_{N+2}^\sigma$ and $\mathcal{E}_{2N}^\sigma$ be as defined in \eqref{p_energy_def} and $\mathcal{D}_{N+2}^\sigma$ be as defined in \eqref{p_dissipation_def}, and let
\begin{equation}\label{en_interp_00}
 \theta = \frac{4N-8}{4N-7} \in (0,1).
\end{equation}
If  $\rj >0$, $\sigma_+>0$, and $\sigma_- > \sigma_c$, then
\begin{equation}\label{en_interp_01}
 \mathcal{E}_{N+2}^\sigma \le \max\left\{1,\frac{\min\{1,\sigma_+,\sigma_--\sigma_c\} }{\min\{1,\sigma_+,\sigma_--\sigma_c,\sigma_+^2,(\sigma_- - \sigma_c)^2\}} \right\} \mathcal{D}_{N+2}^\sigma.
\end{equation}
If  $\rj <0$ and $\sigma_+,\sigma_- >0$, then
\begin{equation}\label{en_interp_02}
 \mathcal{E}_{N+2}^\sigma \le \max\left\{1,\frac{\min\{1,\sigma_+,\sigma_-\} }{\min\{1,\sigma_+^2,\sigma_-^2\}} \right\}  \mathcal{D}_{N+2}^\sigma.
\end{equation}
If $\rj < 0$, then
\begin{equation} \label{en_interp_03}
{\mathcal{E}}_{N+2}^\sigma\lesssim({\mathcal{D}}_{N+2}^\sigma)^\theta({\mathcal{E}}_{2N}^\sigma)^{1-\theta}.
\end{equation}
In either case,
\begin{equation} \label{en_interp_04}
{\mathcal{E}}_{N+2}^\sigma\lesssim \ks ({\mathcal{D}}_{N+2}^\sigma)^\theta({\mathcal{E}}_{2N}^\sigma)^{1-\theta}.
\end{equation}

\end{lemma}

\begin{proof}
The estimates \eqref{en_interp_01} and \eqref{en_interp_02} follow directly from the definitions.

Suppose then that $\rj <0$.  First note that we may trivially estimate
\begin{multline}\label{en_interp_5}
 \sum_{j=0}^{N+2}  \ns{\dt^j u}_{2(N+2)-2j} + \ns{\q}_{2(N+2)}  + \sum_{j=1}^{N+2} \ns{ \dt^j \q}_{2(N+2)-2j+1} + \sum_{j=1}^{N+2} \ns{\dt^j \eta}_{2(N+2)-2j+3/2}
\\
\le
\sum_{j=0}^{N+2} \ns{\dt^j u}_{2(N+2)-2j+1}  +   \ns{\q}_{2(N+2)} + \ns{\dt \q}_{2(N+2)-1} \\+ \ns{\partial_t \eta}_{2(N+2)-1/2}
+ \sum_{j=2}^{N+3} \ns{\dt^j \q}_{2(N+2)-2j+2}  \le ({\mathcal{D}}_{N+2}^\sigma)^\theta({\mathcal{E}}_{2N}^\sigma)^{1-\theta}
\end{multline}
since all the terms appearing in the middle can be controlled by both $\mathcal{D}_{N+2}^\sigma$ and $\mathcal{E}_{2N}^\sigma.$
It remains only to estimate
\begin{equation}
 \ns{\eta}_{2(N+2)} + \min\{1,\sigma\}  \ns{\eta}_{2(N+2)+1}.
\end{equation}
To handle these remaining terms we must use Sobolev interpolation.   Indeed, we first have that
\begin{equation}\label{en_interp_1}
\ns{\eta}_{2(N+2)}\le  \norm{\eta}_{2(N+2)-1/2} ^{2\theta} \norm{\eta}_{4N}^{2(1-\theta)}
\lesssim( {\mathcal{D}_{N+2}^\sigma})^\theta({\mathcal{E}_{2N}^\sigma})^{1-\theta}
\end{equation}
where $\theta$ is as in \eqref{en_interp_00}.  To handle the second term, we interpolate twice:
\begin{equation}\label{en_interp_2}
\begin{split}
 \min\{1,\sigma\}  \ns{\eta}_{2(N+2)+1}&\le  \min\{1,\sigma\}  \norm{\eta}_{2(N+2)+3/2}  \norm{\eta}_{2(N+2)+1/2}
\\
&\le  \min\{1,\sigma\}  \norm{\eta}_{2(N+2)+3/2}    \norm{\eta}_{2(N+2)-1/2}^\frac{4N-9}{4N-7} \norm{\eta}_{4N}^\frac{2}{4N-7}
\\
& = \left( \min\{1,\sigma^2\} \ns{\eta}_{2(N+2)+3/2}   \right)^{1/2}  \left( \ns{\eta}_{2(N+2)-1/2}\right)^\frac{4N-9}{8N-14} \left(\ns{\eta}_{4N}\right)^\frac{1}{4N-7}
\\
& \le( {\mathcal{D}_{N+2}^\sigma})^\theta({\mathcal{E}_{2N}^\sigma})^{1-\theta},
\end{split}
\end{equation}
where the last inequality holds because
\begin{equation}
 \hal + \frac{4N-9}{8N-14} = \frac{4N-8}{4N-7}.
\end{equation}
Then \eqref{en_interp_03} follows from \eqref{en_interp_5}, \eqref{en_interp_1}, and \eqref{en_interp_2}.

Now when $\rj <0$, \eqref{en_interp_04} follows from \eqref{en_interp_03} since $\ks = 1$ is a constant of the form   \eqref{ks_def}.  When $\rj >0$ we chain the trivial bound $\mathcal{D}_{N+2}^\sigma \le ({\mathcal{D}}_{N+2}^\sigma)^\theta({\mathcal{E}}_{2N}^\sigma)^{1-\theta}$ with \eqref{en_interp_01} and note that
\begin{equation}
 \ks = \max\left\{1,\frac{\min\{1,\sigma_+,\sigma_--\sigma_c\} }{\min\{1,\sigma_+,\sigma_--\sigma_c,\sigma_+^2,(\sigma_- - \sigma_c)^2\}} \right\}
\end{equation}
is a constant of the form \eqref{ks_def}.

\end{proof}

Next we deduce algebraic decay of $\mathcal{E}_{N+2}^\sigma$.

\begin{Proposition} \label{decaylm}
There exists a constant $\ks$ of the form \eqref{ks_def} and a universal constant $\delta_0 >0$ so that if $0 < \delta \le \delta_0 /\ks$  and $\gs(t) \le \delta$, then
\begin{equation}\label{dec_00}
\sup_{0\le r \le t} (1+r)^{4N-8} \mathcal{E}_{N+2}^\sigma(r)\lesssim \ks \mathcal{E}_{2N}^\sigma(0)+ \mathcal{F}_{2N}(0).
\end{equation}
\end{Proposition}

\begin{proof}
First we note that again $\sup_{0\le r \le t}\mathcal{E}_{2N}^\sigma(r)\le \delta$, and that we will restrict $\delta$ to prove the desired result.   Let us initially assume that $\delta$ is as small as in  Proposition \ref{Dgle}, which in particular means that $\delta \le 1$.    In this proof we will also need to explicitly keep track of various universal constants $C_j>0$.  Throughout the proof we will write $\tilde{\mathcal{E}}_{N+2}^\sigma$ for the term appearing with the time derivative in \eqref{ebarN+2} of Proposition \ref{boostrapN+2}, and we will write
\begin{equation}\label{dec_99}
s:= \frac{1}{4N-8}.
\end{equation}

We begin by enumerating various estimates proved previously.  From Proposition \ref{boostrapN+2} we have that
\begin{equation}\label{dec_1}
\frac{d}{dt}  \tilde{\mathcal{E}}_{N+2}^\sigma +  C_1 \bar{\mathcal{D}}_{N+2}^\sigma \le C_2 \ks \sqrt{\mathcal{E}_{2N}^\sigma} \mathcal{D}_{N+2}^\sigma.
\end{equation}
 By assuming that $\sqrt{\delta} \ks$ is sufficiently small, we may deduce from \eqref{en+2} of Theorem \ref{eth} and \eqref{dn+2} of Theorem \ref{dth} that
\begin{equation}\label{dec_2}
 \mathcal{E}_{N+2}^\sigma \le C_3 \ks \bar{\mathcal{E}}_{N+2}^\sigma
\end{equation}
and
\begin{equation}\label{dec_3}
  \mathcal{D}_{N+2}^\sigma \le C_4 \ks \bar{\mathcal{D}}_{N+2}^\sigma.
\end{equation}
Further restricting $\delta$ in this manner, we may use the first estimate in \eqref{ee_02} of  Lemma \ref{en_equiv} and \eqref{dec_2} to see that
\begin{equation}\label{dec_4}
 0 \le \hal \bar{\mathcal{E}}_{N+2}^\sigma \le \bar{\mathcal{E}}_{N+2}^\sigma - \sqrt{\mathcal{E}_{2N}^\sigma} \mathcal{E}_{N+2}^\sigma  \le C_5 \tilde{\mathcal{E}}_{N+2}^\sigma
\end{equation}
Similarly, the second estimate in \eqref{ee_02} of  Lemma \ref{en_equiv} and the trivial estimate $\bar{\mathcal{E}}_{N+2}^\sigma \le \mathcal{E}_{N+2}^\sigma$ imply that
\begin{equation}\label{dec_5}
 \tilde{\mathcal{E}}_{N+2}^\sigma \le C_6 \ks  \mathcal{E}_{N+2}^\sigma.
\end{equation}
Next, from \eqref{en_interp_04} of Lemma \ref{en_interp} we know that
\begin{equation}\label{dec_6}
 \mathcal{E}_{N+2}^\sigma \le C_7 \ks (\mathcal{D}_{N+2}^\sigma)^{\frac{4N-8}{4N-7}} (\mathcal{E}_{2N}^\sigma)^{\frac{1}{4N-7}}.
\end{equation}
Finally, since $\delta$ is as small as in Proposition \ref{Dgle}, we know that
\begin{equation}\label{dec_7}
\sup_{0\le r\le t}\mathcal{E}_{2N}^\sigma(r) \le C_8\left(
\ks \mathcal{E}_{2N}^\sigma(0)+ \mathcal{F}_{2N}(0) \right):= \mathcal{M}_0.
\end{equation}
As before, we are free to assume that the constants $\ks \ge 1$ appearing in \eqref{dec_1}--\eqref{dec_3} and \eqref{dec_5}--\eqref{dec_7} are all the same.  We may also assume, increasing the stated constants if need be, that
\begin{equation}\label{dec_75}
 C_4 \ge C_1 \text{ and } C_6, C_7 \ge 1.
\end{equation}

Again restricting $\delta$, we may guarantee that  $C_2 C_4\ks^2 \sqrt{\mathcal{E}_{2N}^\sigma} \le C_1/2$, which allows us to chain together \eqref{dec_1} and \eqref{dec_3} to find that
\begin{equation}
 \frac{d}{dt}  \tilde{\mathcal{E}}_{N+2}^\sigma + \frac{C_1}{2} \bar{\mathcal{D}}_{N+2}^\sigma \le 0,
\end{equation}
which in turn implies, by way of \eqref{dec_3}, that
\begin{equation}\label{dec_8}
  \frac{d}{dt}  \tilde{\mathcal{E}}_{N+2}^\sigma + \frac{C_1}{2C_4 \ks}  \mathcal{D}_{N+2}^\sigma \le 0.
\end{equation}
Chaining together \eqref{dec_5}--\eqref{dec_7} and using the positivity of $\tilde{\mathcal{E}}_{N+2}^\sigma$ guaranteed by \eqref{dec_4} shows that
\begin{equation}\label{dec_9}
 \frac{1}{(C_6 C_7 \ks^2)^{1+s} \mathcal{M}_0^s} (\tilde{\mathcal{E}}_{N+2}^\sigma)^{1+s} \le \mathcal{D}_{N+2}^\sigma,
\end{equation}
where $0<s<1$ is as in \eqref{dec_99}.  We may then combine \eqref{dec_8} and \eqref{dec_9} to see that
\begin{equation}\label{dec_10}
 \frac{d}{dt}  \tilde{\mathcal{E}}_{N+2}^\sigma + \z(\tilde{\mathcal{E}}_{N+2}^\sigma)^{1+s} \le 0,
\end{equation}
where we have written
\begin{equation}
 \z := \frac{C_1}{2C_4 (C_6 C_7)^{1+s}    \ks^{3+2s} \mathcal{M}_0^s}.
\end{equation}

Now we view \eqref{dec_10} as a differential inequality for $\tilde{\mathcal{E}}_{N+2}^\sigma$, which is positive by virtue of \eqref{dec_4}.  We may integrate \eqref{dec_10} to find that for any $0 \le r \le t$,
\begin{equation}\label{dec_11}
 \tilde{\mathcal{E}}_{N+2}^\sigma(r) \le \frac{\tilde{\mathcal{E}}_{N+2}^\sigma(0) }{[1 + s \z(\tilde{\mathcal{E}}_{N+2}^\sigma(0))^s r]^{1/s} }
\end{equation}
From \eqref{dec_5} and \eqref{dec_7} we know that
\begin{equation}\label{dec_12}
\tilde{\mathcal{E}}_{N+2}^\sigma(0)\le  C_6 \mathcal{E}_{N+2}^\sigma(0) \le  C_6 \ks \mathcal{E}_{2N}^\sigma(0)       \le C_6 \ks \mathcal{M}_0,
\end{equation}
which implies that
\begin{multline}\label{dec_13}
 s \z(\tilde{\mathcal{E}}_{N+2}^\sigma(0))^s = \frac{s C_1 (\tilde{\mathcal{E}}_{N+2}^\sigma(0))^s }{2C_4 (C_6 C_7)^{1+s}    \ks^{3+2s} \mathcal{M}_0^s} \le \frac{s C_1 (C_6 \ks \mathcal{M}_0)^s }{2C_4 (C_6 C_7)^{1+s}    \ks^{3+2s} \mathcal{M}_0^s}
\\
= \frac{s C_1  }{2C_4 C_6 (C_7)^{1+s}    \ks^{3+s} } \le 1.
\end{multline}
Here in the last inequality we have used \eqref{dec_75} and the fact that $\ks \ge 1$ and $s <1$.  A simple computation reveals that for $0 \le M \le 1$
\begin{equation}
 \sup_{r>0} \frac{(1+r)^{1/s}}{(1+Mr)^{1/s}} = \frac{1}{M^{1/s}}.
\end{equation}
From this and  \eqref{dec_11}--\eqref{dec_13} we then know that
\begin{multline}\label{dec_14}
 (1+r)^{1/s} \tilde{\mathcal{E}}_{N+2}^\sigma(r) \le \frac{(1+r)^{1/s}}{[1 + s \z(\tilde{\mathcal{E}}_{N+2}^\sigma(0))^s r]^{1/s} } \tilde{\mathcal{E}}_{N+2}^\sigma(0) \le \frac{1}{s^s \z^s \tilde{\mathcal{E}}_{N+2}^\sigma(0)} \tilde{\mathcal{E}}_{N+2}^\sigma(0)
\\
= \frac{(2C_4)^{1/s} (C_6 C_7)^{1/s+1}    \ks^{3/s+2} } {(s C_1)^{1/s}} \mathcal{M}_0.
\end{multline}
This immediately yields \eqref{dec_00} upon recalling the definition of $\mathcal{M}_0$ in \eqref{dec_7}.

\end{proof}

Finally, we deduce exponential decay of $\mathcal{E}_{N+2}^\sigma$ in the case with surface tension.

\begin{Proposition}\label{exp_decay}
Suppose that either $\rj <0$ and $\sigma_+,\sigma_->0$ or else $\rj > 0$ and $\sigma_+>0$, $\sigma_- > \sigma_c$.  Define  $M(\sigma,\rj)>0$ according to
\begin{equation}\label{ex_dec_01}
 M(\sigma,\rj) = \max\left\{1,\frac{\min\{1,\sigma_+,\sigma_--\sigma_c\} }{\min\{1,\sigma_+,\sigma_--\sigma_c,\sigma_+^2,(\sigma_- - \sigma_c)^2\}} \right\}
\end{equation}
if $\rj >0$ and
\begin{equation}\label{ex_dec_02}
 M(\sigma,\rj) =   \max\left\{1,\frac{\min\{1,\sigma_+,\sigma_-\} }{\min\{1,\sigma_+^2,\sigma_-^2\}} \right\}
\end{equation}
if $\rj <0$.

There exists a constant $\ks$ of the form \eqref{ks_def} and a universal constant $\delta_0 >0$ so that if $0 < \delta \le \delta_0 /\ks$ and $\gs(t) \le \delta$, then
\begin{equation}\label{ex_dec_00}
\sup_{0\le r \le t} \exp\left(\frac{r}{\ks  M(\sigma,\rj)} \right) \mathcal{E}_{N+2}^\sigma(r) \lesssim \ks \mathcal{E}_{N+2}^\sigma(0).
\end{equation}
\end{Proposition}

\begin{proof}
Assume that $\delta$ is as small as in Propositions \ref{Dgle} and  \ref{decaylm}.  Arguing as in Proposition \ref{decaylm} (and renaming the constants), we know that
\begin{equation}\label{ex_dec_1}
  \frac{d}{dt}  \tilde{\mathcal{E}}_{N+2}^\sigma + \frac{C_1}{ \ks}  \mathcal{D}_{N+2}^\sigma \le 0
\end{equation}
and
\begin{equation}\label{ex_dec_2}
 0 \le \frac{1}{C_2 \ks} \mathcal{E}_{N+2}^\sigma \le      \tilde{\mathcal{E}}_{N+2}^\sigma \le C_3 \ks \mathcal{E}_{N+2}^\sigma.
\end{equation}
Also, from \eqref{en_interp_01} and \eqref{en_interp_02} of Lemma \ref{en_interp} we know that
\begin{equation}\label{ex_dec_3}
 \mathcal{E}_{N+2}^\sigma  \le C_4 M(\sigma,\rj) \mathcal{D}_{N+2}^\sigma,
\end{equation}
where $M(\sigma,\rj)$ is as defined above.

Combining \eqref{ex_dec_1}--\eqref{ex_dec_3} leads to the differential inequality
\begin{equation}
   \frac{d}{dt}  \tilde{\mathcal{E}}_{N+2}^\sigma + \frac{C_1}{C_3 C_4 \ks^2  M(\sigma,\rj)} \tilde{\mathcal{E}}_{N+2}^\sigma \le 0.
\end{equation}
Upon integrating and again using \eqref{ex_dec_2}, we then find that for $0 \le r \le t$,
\begin{equation}
 \frac{\mathcal{E}_{N+2}^\sigma(r)}{C_2 \ks} \le \tilde{\mathcal{E}}_{N+2}^\sigma(r) \le  C_3 \ks \mathcal{E}_{N+2}^\sigma(0) \exp\left(  \frac{r C_1}{C_3 C_4 \ks^2  M(\sigma,\rj)}  \right).
\end{equation}
This immediately yields \eqref{ex_dec_00}.

\end{proof}

\subsection{A priori estimates for $\gs$}

Now we combine the results of Propositions \ref{p_f_bound}, \ref{Dgle}, and \ref{decaylm} into a single a priori estimate involving the functional $\gs$, as defined in \eqref{G_def}.

\begin{theorem}\label{aprioris}
There exists a constant $\ks$ of the form \eqref{ks_def} and a universal constant $\delta_0 >0$ so that if $0 < \delta \le \delta_0 /\ks$  and $\gs(t) \le \delta$, then
\begin{equation}\label{apri_00}
\gs(t) \ls \ks \mathcal{E}_{2N}^\sigma(0)+ \mathcal{F}_{2N}(0).
\end{equation}
\end{theorem}
\begin{proof}
 Let $\delta_0 >0$ be as small as in Propositions \ref{p_f_bound}, \ref{Dgle}, and \ref{decaylm}.  Then the estimates of each Proposition hold, and they may be summed to deduce \eqref{apri_00}.
\end{proof}

\section{Proof of Main results}\label{section_proof}

\subsection{Proof of Theorem \ref{th_gwp}  }\label{proof1}

Before presenting the proof of Theorem \ref{th_gwp}, we record a technical result that we will use to transition from the bounds provided by the local well-posedness result, Theorem \ref{lwp}, to bounds on $\gs$, as defined by \eqref{G_def}.

\begin{Proposition}\label{G_estimate}
Suppose that $N \ge 3$.  There exists a universal constant $C>0$  with the following two properties.

First, if $0 \le T$, then we have the estimate
\begin{equation}\label{ge_00}
 \gs(T) \le   \sup_{0 \le t \le T} \se{2N}^\sigma(t)  + \int_{0}^{T} \sd{2N}^\sigma(t) dt
+  \sup_{0 \le t \le T} \f(t)  + C  (1+T)^{4N-8} \sup_{0 \le t \le T} \se{2N}^\sigma(t).
\end{equation}

Second, if $0 < T_1 \le T_2$  then we have the estimate
\begin{multline}\label{ge_01}
 \gs(T_2) \le C \gs(T_1) +  \sup_{T_1 \le t \le T_2} \se{2N}^\sigma(t)  + \int_{T_1}^{T_2} \sd{2N}^\sigma(t) dt \\
+ \frac{1}{(1+T_1)} \sup_{T_1 \le t \le T_2} \f(t)  + C (T_2-T_1)^2 (1+T_2)^{4N-8} \sup_{T_1 \le t \le T_2} \se{2N}^\sigma(t).
\end{multline}
\end{Proposition}
\begin{proof}
Although the exact form of the energy and dissipation terms is slightly different, the result follows from the same argument used in the proof of Proposition 9.1 of \cite{GT_per}.  As such, we omit further detail.
\end{proof}

With this result in hand we may now turn to one of our two main results.

\begin{proof}[Proof of Theorem \ref{th_gwp}]

The structure of the estimates in our local well-posedness Theorem \ref{lwp}, our a priori estimates Theorem \ref{aprioris}, and our Proposition \ref{G_estimate} are essentially the same as those found in corresponding results from the study of the incompressible viscous surface wave problem carried out in \cite{GT_per}.  As such, we may argue as in the proof of Theorem 1.3 of \cite{GT_per} to prove \eqref{th_gwp_04}.  Rather than replicate the argument here, we will provide only a sketch; we refer to \cite{GT_per} for a more detailed version of the argument.

Step 1 -- Local theory

Set $\delta>0$ to be the  smaller of the $\delta_0$ constants appearing in Theorem \ref{aprioris} and Proposition \ref{exp_decay} divided by the larger of the $\ks$ constants appearing there.  By choosing $\kappa\le \delta_0$, where $\delta_0$ comes from Theorem \ref{lwp},  we may guarantee that Theorem \ref{lwp} provides a unique solution on an interval $[0,T]$ with $T<1$, satisfying the estimates \eqref{lwp_01}.   We then use estimate \eqref{ge_00} of Proposition \ref{G_estimate} and further restrict $\kappa$ to guarantee that   $\gs(T) \le \delta$.  By the choice of $\delta$, the estimates of Theorem \ref{aprioris} and Proposition \ref{exp_decay} hold.

Step 2 -- Global existence

We define
\begin{multline}
 T_*(\kappa) = \sup \{  T>0 \;\vert\; \text{for every choice of initial data satisfying the compatibility} \\
\text{conditions and } \se{2N}^\sigma(0) + \f(0) < \kappa \text{ there exists a unique solution on } [0,T] \\
\text{that achieves the data and satisfies } \gs(T) \le \delta  \}.
\end{multline}
By Step 1, we know that $T_*(\kappa)>0$, provided $\kappa$ is small enough.  Clearly, $T_*$ is non-increasing, and so it suffices to establish that there is a $\kappa_0>0$ so that $T_*(\kappa_0) = +\infty$.

To prove this we argue by contradiction, showing that  $T_*(\kappa_0) < +\infty$ leads to a contradiction if $\kappa_0>0$ is chosen appropriately.  More precisely, we  choose a $0 < T_1 < T_*(\kappa_0)$ and then use Theorem \ref{aprioris} to guarantee that the hypotheses of the local existence theorems are verified by $(q(T_1),u(T_1),\eta(T_1))$.  We then apply the local theory to construct a solution on $[T_1, T_2]$.  By using estimate \eqref{ge_01} of Proposition \ref{G_estimate} and choosing $T_1$ sufficiently close to $T_*(\kappa_0)$, we may guarantee that $T_*(\kappa_0) < T_2$ and $\gs(T_2) \le \delta$.  This contradicts the definition of $T_*(\kappa_0)$.  So, $T_*(\kappa_0) = + \infty$.  It's clear from this analysis that $1/\kappa_0$ is of the form $\ks$.

Step 3 -- Proof of \eqref{th_gwp_04}

To deduce the bound \eqref{th_gwp_04} we simply apply Theorem \ref{aprioris} on the interval $[0,\infty)$.  This is possible since the previous step guarantees that $\gs(+\infty) \le \delta$.

Step 4 --  Proof of \eqref{th_gwp_05}

Since $\delta$ was chosen as small as in Proposition \ref{exp_decay}, the hypotheses of this Proposition are satisfied in the cases \eqref{th_gwp_02} and \eqref{th_gwp_03}.  As such, the estimate \eqref{ex_dec_00} holds, from which \eqref{th_gwp_05} easily follows.

\end{proof}

\subsection{Proof of Theorem  \ref{th_vanish_st}  }\label{proof2}

We can now establish the vanishing surface tension limit.

\begin{proof}[Proof of Theorem \ref{th_vanish_st}]
For the sake of brevity we will only present a sketch of the proof.

The convergence of the data occurs at sufficiently high regularity to pass to the limit in the compatibility conditions, which shows that  $(q_0,u_0,\eta_0)$ satisfy these with $\sigma_\pm =0$.  Then the data convergence and the bound \eqref{th_gwp_04} show that
\begin{equation}
 \limsup_{(\sigma_+,\sigma_-) \to0} \gs(\infty) \ls \se{2N}^0 + \f(0).
\end{equation}
These bounds, when combined with the usual embedding and interpolation results, show that the solutions $(q^\sigma,u^\sigma,\eta^\sigma)$ converge strongly to a triple $(q,u,\eta)$ in a regularity class high enough to pass to the limit in \eqref{ns_geometric}.  The limiting $(q,u,\eta)$ then solve \eqref{ns_geometric} with $\sigma_\pm =0$, and hence agree with the solution obtained from Theorem \ref{th_gwp} in this case.
\end{proof}

\appendix

\section{Analytic tools}\label{section_appendix}

\subsection{Poisson extensions}

We will now define the appropriate Poisson integrals that allow us to extend $\eta_\pm$, defined on the surfaces $\Sigma_\pm$, to functions defined on $\Omega$, with ``good'' boundedness.

Suppose that $\Sigma_+ = \mathrm{T}^2\times \{\ell\}$, where $\mathrm{T}^2:=(2\pi L_1 \mathbb{T}) \times (2\pi L_2 \mathbb{T})$. We define the Poisson integral in $\mathrm{T}^2 \times (-\infty,\ell)$ by
\begin{equation}\label{P-1def}
\mathcal{P}_{-,1}f(x) = \sum_{\xi \in    (L_1^{-1} \mathbb{Z}) \times
(L_2^{-1} \mathbb{Z}) }  \frac{e^{i \xi \cdot x' }}{2\pi \sqrt{L_1 L_2}} e^{|\xi|(x_3-\ell)} \hat{f}(\xi),
\end{equation}
where for $\xi \in  (L_1^{-1} \mathbb{Z}) \times (L_2^{-1} \mathbb{Z})$ we have written
\begin{equation}
 \hat{f}(\xi) = \int_{\mathrm{T}^2} f(x')  \frac{e^{- i \xi \cdot x' }}{2\pi \sqrt{L_1 L_2}} dx'.
\end{equation}
Here ``$-$'' stands for extending downward and ``$\ell$'' stands for extending at $x_3=\ell$, etc. It is well-known that $\mathcal{P}_{-,\ell}:H^{s}(\Sigma_+) \rightarrow H^{s+1/2}(\mathrm{T}^2 \times (-\infty,\ell))$ is a bounded linear operator for $s>0$. However, if restricted to the domain $\Omega$, we can have the following improvements.

\begin{lemma}\label{Poi}
Let $\mathcal{P}_{-,\ell}f$ be the Poisson integral of a function $f$ that is either in $\dot{H}^{q}(\Sigma_+)$ or
$\dot{H}^{q-1/2}(\Sigma_+)$ for $q \in \mathbb{N}=\{0,1,2,\dots\}$, where we have written $\dot{H}^s(\Sigma_+)$ for the homogeneous Sobolev space of order $s$.  Then
\begin{equation}
\norm{\nabla^q \mathcal{P}_{-,\ell}f }_{0} \lesssim \norm{f}_{\dot{H}^{q-1/2}(\mathrm{T}^2)}^2 \text{ and }
\norm{\nabla^q \mathcal{P}_{-,\ell}f }_{0} \lesssim \norm{f}_{\dot{H}^{q}(\mathrm{T}^2)}^2.
\end{equation}
\end{lemma}
\begin{proof}
 See Lemma A.3 of \cite{GT_per}.
\end{proof}

We extend $\eta_+$ to be defined on $\Omega$ by
\begin{equation}\label{P+def}
\bar{\eta}_+(x',x_3)=\mathcal{P}_+\eta_+(x',x_3):=\mathcal{P}_{-,\ell}\eta_+(x',x_3),\text{ for } x_3\le \ell.
\end{equation}
Then Lemma \ref{Poi} implies in particular that if $\eta_+\in H^{s-1/2}(\Sigma_+)$ for $s\ge 0$, then $\bar{\eta}_+\in H^{s}(\Omega)$.

Similarly, for $\Sigma_- = \mathrm{T}^2\times \{0\}$ we define the Poisson integral in $\mathrm{T}^2 \times (-\infty,0)$ by
\begin{equation}\label{P-0def}
\mathcal{P}_{-,0}f(x) = \sum_{\xi \in    (L_1^{-1} \mathbb{Z}) \times (L_2^{-1} \mathbb{Z}) }  \frac{e^{ i \xi \cdot x' }}{2\pi \sqrt{L_1 L_2}} e^{ |\xi|x_3} \hat{f}(\xi).
\end{equation}
It is clear that $\mathcal{P}_{-,0}$ has the  same regularity properties as $\mathcal{P}_{-,\ell}$. This allows us to extend $\eta_-$ to be defined on $\Omega_-$. However, we do not extend $\eta_-$ to the upper domain $\Omega_+$ by the reflection  since this will result in the discontinuity of the partial derivatives in $x_3$ of the extension. For our purposes,  we instead to do the extension through the following. Let $0<\lambda_0<\lambda_1<\cdots<\lambda_m<\infty$ for $m\in \mathbb{N}$ and define the $(m+1) \times (m+1)$ Vandermonde matrix $V(\lambda_0,\lambda_1,\dots,\lambda_m)$ by $V(\lambda_0,\lambda_1,\dots,\lambda_m)_{ij} = (-\lambda_j)^i$ for $i,j=0,\dotsc,m$.  It is well-known that the Vandermonde matrices are invertible, so we are free to let $\alpha=(\alpha_0,\alpha_1,\dots,\alpha_m)^T$ be the solution to
\begin{equation}\label{Veq}
V(\lambda_0,\lambda_1,\dots,\lambda_m)\,\alpha=q_m,\ q_m=(1,1,\dots,1)^T.
\end{equation}
Now we define the specialized Poisson integral in $\mathrm{T}^2 \times (0,\infty)$
by
\begin{equation}\label{P+0def}
\mathcal{P}_{+,0}f(x) = \sum_{\xi \in    (L_1^{-1} \mathbb{Z}) \times
(L_2^{-1} \mathbb{Z}) }  \frac{e^{ i \xi \cdot x' }}{2\pi \sqrt{L_1 L_2}}  \sum_{j=0}^m\alpha_j
e^{- |\xi|\lambda_jx_3} \hat{f}(\xi).
\end{equation}
It is easy to check that, due to \eqref{Veq}, $\partial_3^l\mathcal{P}_{+,0}f(x',0)=
\partial_3^l\mathcal{P}_{-,0}f(x',0)$  for all $0\le l\le m$ and hence
\begin{equation}
\partial^\alpha\mathcal{P}_{+,0}f(x',0)=
\partial^\alpha\mathcal{P}_{-,0}f(x',0), \ \forall\, \alpha\in \mathbb{N}^3 \text{ with }0\le |\alpha|\le m.\end{equation}
These facts allow us to  extend $\eta_-$ to be defined on $\Omega$
by
\begin{equation}\bar{\eta}_-(x',x_3)=
\mathcal{P}_-\eta_-(x',x_3):=\left\{\begin{array}{lll}\mathcal{P}_{+,0}\eta_-(x',x_3),\quad
x_3> 0 \\
\mathcal{P}_{-,0}\eta_-(x',x_3),\quad x_3\le
0.\end{array}\right.\label{P-def}\end{equation}
It is clear now that if $\eta_-\in H^{s-1/2}(\Sigma_-)$ for $ 0\le s\le m$, then $\bar{\eta}_-\in H^{s}(\Omega)$.  Since we will only work with $s$ lying in a finite interval, we may assume that $m$ is sufficiently large in \eqref{Veq} for $\bar{\eta}_- \in H^s(\Omega)$ for all $s$ in the interval.

\subsection{Estimates of Sobolev norms}

We will need some estimates of the product of functions in Sobolev spaces.

\begin{lemma}\label{sobolev}
Let $U$ denote a domain either of the form $\Omega_\pm$ or of the form $\Sigma_\pm$.
\begin{enumerate}
 \item Let $0\le r \le s_1 \le s_2$ be such that  $s_1 > n/2$.  Let $f\in H^{s_1}(U)$, $g\in H^{s_2}(U)$.  Then $fg \in H^r(U)$ and
\begin{equation}\label{i_s_p_01}
 \norm{fg}_{H^r} \lesssim \norm{f}_{H^{s_1}} \norm{g}_{H^{s_2}}.
\end{equation}

\item Let $0\le r \le s_1 \le s_2$ be such that  $s_2 >r+ n/2$.  Let $f\in H^{s_1}(U)$, $g\in H^{s_2}(U)$.  Then $fg \in H^r(U)$ and
\begin{equation}\label{i_s_p_02}
 \norm{fg}_{H^r} \lesssim \norm{f}_{H^{s_1}} \norm{g}_{H^{s_2}}.
\end{equation}
\end{enumerate}
\end{lemma}
\begin{proof}
These results are standard and may be derived, for example, by use of the Fourier characterization of the $H^s$ spaces and extensions.
\end{proof}


%

\subsection{Coefficient estimates}

Here we are concerned with how the size of $\eta$ can control the ``geometric'' terms that appear in the equations.
\begin{lemma}\label{eta_small}
There exists a universal $0 < \delta < 1$ so that if $\ns{\eta}_{5/2} \le \delta,$ then
\begin{equation}\label{es_01}
\begin{split}
 & \norm{J-1}_{L^\infty(\Omega)} +\norm{A}_{L^\infty(\Omega)} + \norm{B}_{L^\infty(\Omega)} \le \hal, \\
 & \norm{\n-1}_{L^\infty(\Gamma)} + \norm{K-1}_{L^\infty(\Gamma)} \le \hal, \text{ and }  \\
 & \norm{K}_{L^\infty(\Omega)} + \norm{\mathcal{A}}_{L^\infty(\Omega)} \ls 1.
 \end{split}
\end{equation}
Also, the map $\Theta$ defined by \eqref{cotr} is a diffeomorphism.
\end{lemma}
\begin{proof}
The estimate \eqref{es_01} is guaranteed by Lemma 2.4 of \cite{GT_per}.
\end{proof}

\subsection{Korn inequality}

Consider the following operators acting on functions $u: \Rn{n} \supseteq U \to \Rn{n}$:
\begin{equation*}
 \sg{u} = \nab u + \nab u^T \text{ and } \sgz{u} = \sg{u} - \frac{2 \diverge{u}}{n} I,
\end{equation*}
or in components
\begin{equation*}
 \sg{u}_{ij} = \p_i u_j + \p_j u_i \text{ and } \sgz{u}_{ij} = \p_i u_j + \p_j u_i  - \frac{2 \p_k u_k}{n} \delta_{ij}.
\end{equation*}
Note that $\trace(\sgz{u})=0$.  As such, the operator $\sgz$ is referred to as the ``deviatoric part of the symmetric gradient.''

In what follows we will compute the kernel of $\sg$ and $\sgz$ in $H^1(\Omega)$.  In doing so we will actually compute the kernel on $C^3(\Omega)$, but the results hold as well in $H^1(\Omega)$ via an approximation argument.

\begin{lemma}\label{sg_ker}
 \begin{equation}
  \ker(\sg) = \{ u(x) = a + A x \;\vert\; a \in \Rn{n}, A = -A^T\}
 \end{equation}
\end{lemma}
\begin{proof}
We have that $\sg{u}=0$ if and only if $\p_i u_j = -\p_j u_i$ for all $i,j$.  Then
\begin{equation*}
 \p_i \p_k u_j = -\p_k \p_j u_i = -\p_j (-\p_i u_k) = \p_j \p_i u_k = \p_i (-\p_k u_j) = -\p_i \p_k u_j,
\end{equation*}
and hence $\nab^2 u =0$.  Then $u(x) = a + A x$ for some $a \in \Rn{n}$ and $A \in \Rn{n\times n}$.  But then
\begin{equation}
 0 = \sg{u} = A + A^T \Rightarrow A = -A^T.
\end{equation}
So, if $u \in \ker(\sg)$ then $u(x) = a + A x$ for $A$ antisymmetric.  The converse also clearly holds.
\end{proof}

We may also  compute the kernel of $\sgz$.

\begin{lemma}\label{sgd_ker}
If $n=2$ then
\begin{equation}\label{sgdk_01}
 \ker(\sgz) = \{ u: \mathbb{C} \simeq \Rn{2} \supseteq \Omega \to \Rn{2} \simeq \mathbb{C} \;\vert\; u \text{ is holomorphic}  \}.
\end{equation}
If $n \ge 3$ then
 \begin{equation}\label{sgdk_02}
  \ker(\sgz) = \{ u(x) = a + A x  + \gamma x + (b\cdot x) x - b\abs{x}^2/2   \;\vert\; \gamma \in \Rn{}, a,b \in \Rn{n}, A = -A^T\}.
 \end{equation}
\end{lemma}
\begin{proof}
Consider first the case $n=2$.  Then $\sgz{u}=0$ if and only if
\begin{equation}
 \begin{pmatrix}
  2\p_1 u_1 & \p_1 u_2 + \p_2 u_1 \\
  \p_1 u_2 + \p_2 u_1 & 2 \p_2 u_2
 \end{pmatrix}
=
\begin{pmatrix}
 \p_1 u_1 + \p_2 u_2 & 0 \\
 0 & \p_1 u_2 + \p_2 u_2,
\end{pmatrix}
\end{equation}
which holds if and only if
\begin{equation}
 \p_1 u_1 = \p_2 u_2 \text{ and } \p_2 u_1 = -\p_1 u_2,
\end{equation}
which are the Cauchy-Riemann equations in $\Omega$.  Regarding $\Rn{2} \simeq \mathbb{C}$ then shows that $u$ is holomorphic in $\Omega$, which is \eqref{sgdk_01}.

Now consider the case $n \ge 3$.  Throughout the rest of the proof we \emph{do not} employ the summation convention.  Note first that if $\sgz{u}=0$ then
\begin{equation}
2\p_i u_i =  \sg{u} e_i \cdot e_i = \frac{2 \diverge{u}}{n} Ie_i \cdot e_i  = \frac{1}{n} \diverge{u}
\end{equation}
for each $i=1,\dotsc,n$.  This means that
\begin{equation}\label{sgdk_1}
 \p_1 u_1 = \p_2 u_2 = \cdots = \p_n u_n = \frac{1}{n} \diverge{u} := M.
\end{equation}
This allows us to rewrite $\sgz{u}=0$ in components as
\begin{equation}\label{sgdk_2}
 \p_i u_j = -\p_j u_i + 2M \delta_{ij}.
\end{equation}

Next we compute
\begin{equation}
\begin{split}
 \p_i \p_k u_j & = \p_k (-\p_j u_i + 2M \delta_{ij}) = -\p_j \p_k u_i + 2\p_k M \delta_{ij} \\
&= -\p_j( -\p_i u_k + 2 M \delta_{ik}) +  2\p_k M \delta_{ij}  \\
& = \p_i \p_j u_k +  2\p_k M \delta_{ij}  - 2\p_j M \delta_{ik} \\
& = \p_i (-\p_k u_j + 2 M \delta_{jk}) +  2\p_k M \delta_{ij}  - 2\p_j M \delta_{ik} \\
& = -\p_i \p_k u_j +  2\p_k M \delta_{ij}  - 2\p_j M \delta_{ik} + 2\p_i M \delta_{jk},
\end{split}
\end{equation}
and hence
\begin{equation}\label{sgdk_3}
 \p_i \p_k u_j = \p_k M \delta_{ij}  - \p_j M \delta_{ik} + \p_i M \delta_{jk}.
\end{equation}
Applying $\p_\ell$ to \eqref{sgdk_3}, we find that
\begin{equation}\label{sgdk_4}
  \p_i \p_k \p_\ell u_j = \p_k \p_\ell M \delta_{ij}  - \p_j \p_\ell M \delta_{ik} + \p_i \p_\ell M \delta_{jk}.
\end{equation}
On the other hand, we can switch the index from $k$ to $\ell$ in \eqref{sgdk_3} to see that
\begin{equation}
 \p_i \p_\ell u_j = \p_\ell M \delta_{ij}  - \p_j M \delta_{i \ell} + \p_i M \delta_{j\ell},
\end{equation}
which reveals, upon applying $\p_k$, that
\begin{equation}\label{sgdk_5}
  \p_i \p_k \p_\ell u_j = \p_k  \p_\ell M \delta_{ij}  - \p_j \p_k  M \delta_{i \ell} + \p_i \p_k  M \delta_{j\ell}.
\end{equation}

Now we equate \eqref{sgdk_4} and \eqref{sgdk_5} to see that
\begin{equation}
 \p_k \p_\ell M \delta_{ij}  - \p_j \p_\ell M \delta_{ik} + \p_i \p_\ell M \delta_{jk}  =  \p_k  \p_\ell M \delta_{ij}  - \p_j \p_k  M \delta_{i \ell} + \p_i \p_k  M \delta_{j\ell},
\end{equation}
which allows us to cancel the $\p_k \p_\ell M \delta_{ij}$ terms to deduce that
\begin{equation}
 \p_i \p_\ell M \delta_{jk}  - \p_j \p_\ell M \delta_{ik}  = \p_i \p_k  M \delta_{j\ell} - \p_j \p_k  M \delta_{i \ell}.
\end{equation}
Multiplying by $\delta_{jk}$ and summing over $j,k$ then reveals that
\begin{equation}\label{sgdk_6}
 (n-2) \p_i \p_\ell M = -\Delta M \delta_{i\ell},
\end{equation}
from which we deduce, by multiplying by $\delta_{i\ell}$ and summing over $i,\ell$, that
\begin{equation}
 (n-2) \Delta M = -n \Delta M.
\end{equation}
Hence $\Delta M=0$, and upon replacing in \eqref{sgdk_6} and using the fact that $n \ge 3$, we find that
\begin{equation}\label{sgdk_7}
 \nab^2 M = 0 \Rightarrow M(x) = \gamma + b \cdot x \text{ for some } \gamma\in \Rn{n}, b \in \Rn{n}.
\end{equation}
Returning to \eqref{sgdk_4}, we also find that
\begin{equation}\label{sgdk_8}
 \nab^3 u =0.
\end{equation}

From \eqref{sgdk_7} and \eqref{sgdk_1} we deduce that
\begin{equation}\label{sgdk_9}
 u_i(x) = g_i(\hat{x}_i) + \gamma x_i + (b\cdot \hat{x}_i) x_i + b_i \frac{x_i^2}{2},
\end{equation}
where $\hat{x}_i = x - x_i e_i$, and $g_i$ is (by virtue of \eqref{sgdk_8}) a quadratic polynomial in $\hat{x}_i$, i.e. in all of the variables except $x_i$.

Since whenever $i\neq j$ we have $\p_j u_i = -\p_i u_j$, we may use \eqref{sgdk_9} to compute
\begin{equation}\label{sgdk_10}
 \p_j g_i(\hat{x}_i) + x_i b_j = -\p_i g_j(\hat{x}_j) - x_j b_i \Rightarrow \p_j g_i(\hat{x}_i) + x_j b_i = -(\p_i g_j(\hat{x}_j) + x_i b_j) \text{ for }i\neq j.
\end{equation}
Define $G:U \to \Rn{n}$ by
\begin{equation}
 G_j(x) = g_j(\hat{x}_j) + \frac{\abs{\hat{x}_j}^2}{2} b_j \text{ for }j=1,\dotsc,n.
\end{equation}
Then
\begin{equation}
 \p_i G_j(x) =
\begin{cases}
 \p_i g_j(\hat{x}_j) + x_i b_j & i\neq j \\
 0 & i=j,
\end{cases}
\end{equation}
which together with \eqref{sgdk_10} implies that $\sg{G} =0$ in $\Omega$.  Then by Lemma \ref{sg_ker} we have that $G(x) = a + A x$ for $a \in \Rn{n}$ and $A =-A^T \in \Rn{n\times n}$.  Then
\begin{equation}
 g_j(\hat{x}_j) =  a_j + (A x)_j - \frac{\abs{\hat{x}_j}^2}{2} b_j,
\end{equation}
which we may plug into \eqref{sgdk_9} to deduce that
\begin{equation}
 u(x) = a + Ax + \gamma x + (b\cdot x) x - \frac{\abs{x}^2}{2} b.
\end{equation}
So, every element of $\ker(\sgz)$ is of this form.  An easy computation reveals that every function of this form is also in the kernel.

\end{proof}

Next we record a version of Korn's inequality involving only the deviatoric part, $\sgz$.

\begin{lemma} \label{full_korn}
 Let $U \subset \Rn{n}$, $n \ge 3$, be bounded and open with Lipschitz boundary.  Then there exists a constant $C>0$ such that
\begin{equation}
 \norm{u}_{H^1}^2 \le C (\norm{u}_{L^2}^2 + \norm{\sgz{u}}_{L^2}^2)
\end{equation}
for all $u \in H^1(U)$.
\end{lemma}
\begin{proof}
See Theorem 1.1 of \cite{dain}.
\end{proof}

We now prove a first result on the way to our desired Korn-type inequality.

\begin{lemma}\label{bndry_korn}
Let $n \ge 3$ and
\begin{equation}
 \Omega := \Pi_{i=1}^{n-1} (L_i \mathbb{T}) \times (a,b)
\end{equation}
for $0< L_i< \infty$, $i=1,\dotsc,n-1$, and $-\infty < a < b < \infty$, and define the lower boundary by $\Sigma_a =  \Pi_{i=1}^{n-1} (L_i \mathbb{T}) \times \{ a\}$.   There exists a constant $C>0$ so that
\begin{equation}
 \ns{u}_{0} \le C (\ns{\sgz{u}}_0 + \ns{u}_{H^0(\Sigma_a)}  )
\end{equation}
for all $u \in H^1(\Omega)$.
\end{lemma}
\begin{proof}

Suppose not.  Then there exists a sequence $u_n \in H^1(\Omega)$ so that $\norm{u_n}_0 = 1$ but
\begin{equation}\label{bk_1}
 \ns{\sgz{u_n}}_0 + \ns{u_n}_{H^0(\Sigma_a)} \le \frac{1}{n}.
\end{equation}
Employing Lemma \ref{full_korn}, we find that $\sup_{n} \norm{u_n}_{H^1} < \infty$.  So, up to the extraction of a subsequence, we have that $u_n \rightharpoonup u$ weakly in $H^1$ and $u_n \to u$ strongly in $H^0$.  By \eqref{bk_1} we have that $\sgz u =0$ and $u=0$ on $\Sigma_a$.  Then by Lemma \ref{sgd_ker} we know that
\begin{equation}
 u(x) = a + Ax + \gamma x + (b\cdot x)x - b\abs{x}^2/2 \text{ for } \gamma \in \Rn{}, a,b\in \Rn{n}, A=-A^T \in \Rn{n\times n}.
\end{equation}
Since $u$ vanishes on $\Sigma_a$ we must have the zeroth, first, and second order parts of the polynomial all vanish, and hence
\begin{equation}\label{bk_2}
 \begin{cases}
  a =0 \\
  A e + \gamma e =0 \\
  (b\cdot e) e - b/2 =0
 \end{cases}
\end{equation}
for all unit vectors $e \in \Rn{n-1}$.  The second equality in \eqref{bk_2} implies, since $A$ is antisymmetric, that
\begin{equation}
 \gamma = \gamma \abs{e}^2 = \gamma e \cdot e = -A e \cdot e =0.
\end{equation}
Then $A e =0$ for all unit vectors $e \in \Rn{n-1}$, which means that $A=0$ since otherwise $Ae_n = y \neq 0$ implies that $y_i = y\cdot e_i = A e_n \cdot e_i = -e_n \cdot A e_i =0$ for $i=1,\dotsc,n-1$ so that $y = y_n e_n$, but then $y_n = A e_n \cdot e_n =0$.  Choosing $e = e_1$ in the third equality in \eqref{bk_2} implies that
\begin{equation}
  b_1 e_1 =  b_2 e_2 + \dotsc b_{n}e_{n} \Rightarrow b=0.
\end{equation}
Hence $u=0$ in $\Omega$, which contradicts the fact that $\norm{u}_{0}=1$.

\end{proof}

Now we extend this to layered domains.

\begin{Proposition}\label{layer_korn}
 Let $n \ge 3$,
\begin{equation}
 \Omega_- := \Pi_{i=1}^{n-1} (L_i \mathbb{T}) \times (-b,0) \text{ and }  \Omega_+ := \Pi_{i=1}^{n-1} (L_i \mathbb{T}) \times (0,\ell)
\end{equation}
for $0< L_i< \infty$, $i=1,\dotsc,n-1$, and $0 < b,\ell < \infty$, and define the bottom boundary by $\Sigma_b =  \Pi_{i=1}^{n-1} (L_i \mathbb{T}) \times \{ -b\}$ and the internal interface by  $\Sigma =  \Pi_{i=1}^{n-1} (L_i \mathbb{T}) \times \{0 \}$.    There exists a constant $C>0$ so that
\begin{equation}
 \ns{u_+}_{1} + \ns{u_-}_{1} \le C(\ns{\sgz{u_+}}_0 + \ns{\sgz{u_-}}_0)
\end{equation}
for all $u_\pm \in H^1(\Omega_\pm)$ with $\jump{u}=0$ along $\Sigma$ and $u_- =0$ on $\Sigma_b$.
\end{Proposition}
\begin{proof}
 By Lemmas \ref{full_korn}  and \ref{bndry_korn} we have that
\begin{equation}
 \ns{u_-}_{1} \le C(\ns{u_-}_0 + \ns{\sgz{u_-}}_0) \le C (\ns{\sgz{u_-}}_0 + \ns{u_-}_{H^0(\Sigma_b)}  ) =   C \ns{\sgz{u_-}}_0.
\end{equation}
By trace theory and the jump condition $\jump{u}=0$ on $\Sigma$, we may estimate
\begin{equation}
 \ns{u_+}_{H^0(\Sigma)} =  \ns{u_-}_{H^0(\Sigma)} \le C \ns{u_-}_1 \le C \ns{\sgz{u_-}}_0.
\end{equation}
Then to conclude we use Lemmas \ref{full_korn}  and \ref{bndry_korn} together with this estimate:
\begin{equation}
  \ns{u_+}_{1} \le C(\ns{u_+}_0 + \ns{\sgz{u_+}}_0) \le C (\ns{\sgz{u_+}}_0 + \ns{u_+}_{H^0(\Sigma)}  ) \le  C (\ns{\sgz{u_+}}_0 + \ns{\sgz{u_-}}_0 ).
\end{equation}

\end{proof}



\subsection{Linear combinations of estimates}

We now consider how to reduce a sequence of related estimates to a single estimate in a more useful form.

\begin{lemma}\label{est_alg}
Let $k \ge 1$ be an integer, $U \subset \Rn{}$ be an open set, and consider functions $X_i,Y_i, Z_i:U \to \Rn{}$ for $i=0,\dotsc,k$, $W_i : U \to [0,\infty)$ for $i=0,\dotsc,k$,  and $V_0:U \to \Rn{}$.  Suppose that
\begin{equation}\label{es_al_00}
  \frac{d}{dt} X_0 + \beta_0 (Y_0 + W_0) \le \gamma_0 V_0 +  Z_0
\end{equation}
and for $i=1,\dotsc,k$
\begin{equation}\label{es_al_01}
 \frac{d}{dt} X_i + \beta_i (Y_i + W_i) \le \gamma_i \left( V_0 + \sum_{j=0}^{i-1} Y_j\right) + Z_i,
\end{equation}
where $\beta_i, \gamma_i >0$ are constants for $i=0,\dotsc,k$.  Then there exist $\lambda_i>0$ for $i=0,\dotsc,k$ so that
\begin{equation}\label{es_al_02}
  \frac{d}{dt} \sum_{i=0}^k \lambda_i X_i +   \sum_{i=0}^k (Y_i+W_i) \le V_0 \sum_{i=0}^k \lambda_i \gamma_i + \sum_{i=0}^k \lambda_i Z_i.
\end{equation}
Moreover, $\lambda_k = 1/\beta_k$ and for $i=0,\dotsc,k-1$
\begin{equation}\label{es_al_03}
 \lambda_i =   P_i(\beta_{i+1},\dotsc,\beta_k, \gamma_{i+1},\dotsc,\gamma_k) \prod_{j=0}^k\beta_j^{-1}
\end{equation}
where $P_i$ is a polynomial with positive integer coefficients.

\end{lemma}
\begin{proof}
 Suppose that $\lambda_i >0$ for $i=0,\dotsc,k$.  Multiplying \eqref{es_al_01} by $\lambda_0$,  \eqref{es_al_01} by $\lambda_i$, and  summing, we find that
\begin{equation}\label{es_al_1}
\begin{split}
 \frac{d}{dt} \sum_{i=0}^k \lambda_i X_i + \sum_{i=0}^k \lambda_i \beta_i Y_i  + \sum_{i=0}^k \lambda_i \beta_i W_i &\le V_0 \sum_{i=1}^k \gamma_i \lambda_i +
\sum_{i=1}^k \lambda_i \gamma_i \sum_{j=0}^{i-1} Y_j + \sum_{i=1}^k \lambda_i Z_i \\
&= V_0 \sum_{i=1}^k \gamma_i \lambda_i + \sum_{i=0}^{k-1} Y_i \sum_{j={i+1}}^{k} \lambda_j \gamma_j + \sum_{i=0}^k \lambda_i Z_i.
\end{split}
\end{equation}
We then seek to choose the $\lambda_i$ so that
\begin{equation}\label{es_al_2}
\begin{split}
 \lambda_i \beta_i &= \sum_{j=1+i}^k \lambda_j \gamma_j + 1 \text{ for } i=0,\dotsc,k-1, \\
 \lambda_k \beta_k &= 1.
\end{split}
\end{equation}
This linear systems admits a unique solution with $\lambda_i >0$.  Indeed, we may set $\lambda_k = 1 / \beta_k >0$ and then iteratively solve for
\begin{equation}\label{es_al_3}
 \lambda_i = \frac{1}{\beta_i} \left(1 + \sum_{j=1+i}^k \gamma_j \lambda_j \right)>0
\end{equation}
for $i=k-1, \dotsc,0$.  Plugging the equations \eqref{es_al_2} into \eqref{es_al_1} then yields the estimate
\begin{equation}
  \frac{d}{dt} \sum_{i=0}^k \lambda_i X_i +   \sum_{i=0}^k Y_i + \sum_{i=0}^k \lambda_i \beta_i W_i \le V_0 \sum_{i=0}^k \lambda_i \gamma_i + \sum_{i=0}^k \lambda_i Z_i.
\end{equation}
The estimate \eqref{es_al_02} follows from this by observing  that by construction $\lambda_i \beta_i \ge 1$ and that $W_i \ge 0$.   The form \eqref{es_al_03} follows easily from \eqref{es_al_3} and elementary calculations.

\end{proof}

\subsection{Elliptic estimates}

Here we consider the two-phase elliptic problem
\begin{equation}\label{lame eq}
\begin{cases}
-\mu \Delta u -(\mu/3+\mu') \nabla\diverge u =F^2  &\hbox{ in }\Omega
\\-\S(u_+)e_3=F^3_+   &\hbox{ on }\Sigma_+
\\  -\jump{\S(u)}e_3=-F^3_- &\hbox{ on }\Sigma_-
\\ \jump{u}=0  &\hbox{ on }\Sigma_-
\\  u_-=0 &\hbox{ on }\Sigma_b.
\end{cases}
\end{equation}
We have the following elliptic regularity result.
\begin{lemma}\label{lame reg}
Let $r\ge 2$. If $F^2\in  {H}^{r-2}(\Omega), F^3\in  {H}^{r-3/2}(\Sigma)$, then the problem \eqref{lame eq} admits a unique strong solution $u\in {}_0H^1(\Omega)\cap {H}^r(\Omega)$. Moreover,
\begin{equation}
\norm{u}_{r}\lesssim \norm{F^2}_{r-2}+\norm{F^3}_{r-3/2}.
\end{equation}
\end{lemma}
\begin{proof}
We refer to \cite[Theorem 3.1]{WTK} for the case of two-phase Stokes problem, but the proof is the same here. It follows by making use of the flatness of the boundaries $\Sigma_\pm$ and applying the standard classical one-phase elliptic theory with Dirichlet boundary condition.
\end{proof}

We let $G$ denote a horizontal periodic slab with its boundary $\pa G$ (not necessarily flat) consisting of two smooth pieces. We shall recall the classical regularity theory for the Stokes problem with Dirichlet boundary conditions on $\pa G$,
\begin{equation}\label{stokes eq}
\begin{cases}
-\mu \Delta u +\nabla p =f  \quad &\hbox{in }G
\\\diverge{u} =h  \quad  &\hbox{in }G
\\u=\varphi\quad &\hbox{on }\pa G.
\end{cases}
\end{equation}
The following records the regularity theory for this problem.
\begin{lemma}\label{stokes reg}
Let $r\ge 2$. If $f\in H^{r-2}(G),  h\in H^{r-1}(G), \varphi\in H^{r-1/2}(\pa G)$ be given such that
\begin{equation}
\int_G h =\int_{\pa G} \varphi \cdot\nu ,
\end{equation}
then there exists unique $u\in H^r(G),  p\in H^{r-1}(G)$(up to
constants) solving \eqref{stokes eq}. Moreover,
\begin{equation}
\norm{u}_{H^r(G)}+\norm{\nabla p}_{H^{r-2}(G)}\lesssim\norm{f}_{H^{r-2}(G)}+\norm{h}_{H^{r-1}(G)}+\norm{\varphi}_{H^{r-1/2}(\pa G)}.
\end{equation}
\end{lemma}
\begin{proof}
 See \cite{L,T}.
\end{proof}

\begin{lemma}\label{bndry_elliptic}
Suppose that $\sigma\ge 0$ and $\kappa >0$.  If $\zeta$ satisfies $\int_{\mathrm{T}^2} \zeta =0$ and
\begin{equation}\label{be_01}
-\sigma \Delta_\ast \zeta + \kappa \zeta=\varphi\text{ on }\mathrm{T}^2,
\end{equation}
then for $r\ge 2$,
\begin{equation}\label{be_02}
\sigma  \norm{\zeta}_{r}  + \kappa \norm{\zeta}_{r-2} \lesssim \norm{\varphi}_{r-2}.
\end{equation}
\end{lemma}
\begin{proof}
We use \eqref{be_01} to see that
\begin{equation}
 (\sigma \abs{n}^2 + \kappa) \hat{\zeta}(n) = \hat{\varphi}(n) \text{ for all } n \in (L_1^{-1}\mathbb{Z})\times (L_2^{-1}\mathbb{Z}).
\end{equation}
In particular, this means that $\hat{\varphi}(0)=0$.  Since $\kappa >0$ we may estimate
\begin{equation}
 \sigma^2 \abs{n}^4 \abs{\hat{\zeta}(n)}^2 \le \abs{\hat{\varphi}(n)}^2
\end{equation}
and then argue as in Lemma \ref{critical lemma} to deduce the bound
\begin{equation}
 \sigma  \norm{\zeta}_{r}   \lesssim \norm{\varphi}_{r-2}.
\end{equation}
On the other hand, since $\sigma \ge 0$, we may bound $\kappa^2 \abs{\hat{\zeta}(n)}^2 \le  \abs{\hat{\varphi}(n)}^2$ to get the estimate
\begin{equation}
 \kappa \norm{\zeta}_{r-2} \lesssim \norm{\varphi}_{r-2}.
\end{equation}

\end{proof}

\end{document}